\documentclass[11pt,oneside,english]{amsart}
\usepackage[T1]{fontenc}
\usepackage[latin9]{inputenc}
\usepackage{geometry}
\usepackage{color}
\usepackage{babel}
\usepackage{verbatim}
\usepackage{amsbsy}
\usepackage{amstext}
\usepackage{amsthm}
\usepackage{amssymb}
\usepackage[numbers]{natbib}
\usepackage[unicode=true,pdfusetitle,
 bookmarks=true,bookmarksnumbered=false,bookmarksopen=false,
 breaklinks=false,pdfborder={0 0 1},backref=false,colorlinks=false]
 {hyperref}

\makeatletter
\numberwithin{equation}{section}
\numberwithin{figure}{section}
\theoremstyle{definition}
\newtheorem{defn}{\protect\definitionname}
\theoremstyle{remark}
\newtheorem*{rem*}{\protect\remarkname}
\theoremstyle{plain}
\newtheorem{thm}{\protect\theoremname}
\theoremstyle{plain}
\newtheorem{lem}{\protect\lemmaname}
\theoremstyle{plain}
\newtheorem{assumption}{\protect\assumptionname}
\theoremstyle{plain}
\newtheorem{prop}{\protect\propositionname}

\usepackage{enumitem}
\setlist{nosep}
\usepackage{graphicx}

\makeatother

\providecommand{\assumptionname}{Assumption}
\providecommand{\definitionname}{Definition}
\providecommand{\lemmaname}{Lemma}
\providecommand{\propositionname}{Proposition}
\providecommand{\remarkname}{Remark}
\providecommand{\theoremname}{Theorem}

\begin{document}
\newcommand{\R}{\mathbb{R}}
\title{Global Optimal Regularity for Variational Problems with Nonsmooth
Non-strictly Convex Gradient Constraints}
\author{Mohammad Safdari$\,{}^{1}$}
\begin{abstract}
We prove the optimal $W^{2,\infty}$ regularity for variational problems
with convex gradient constraints. We do not assume any regularity
of the constraints; so the constraints can be nonsmooth, and they
need not be strictly convex. When the domain is smooth enough, we
show that the optimal regularity holds up to the boundary. In this
process, we also characterize the set of singular points of the viscosity
solutions to some Hamilton-Jacobi equations. Furthermore, we obtain
an explicit formula for the second derivative of these viscosity solutions;
and we show that the second derivatives satisfy a monotonicity property.\medskip{}

\thanks{$^{1}\;$Department of Mathematical Sciences, Sharif University of
Technology, Tehran, Iran\protect \\
Email address: safdari@sharif.edu}
\end{abstract}

\maketitle
\tableofcontents{}

\section{Introduction}

Variational problems and differential equations with gradient constraints,
has been an active area of study, and has seen many progresses. An
important example among them is the famous elastic-plastic torsion
problem, which is the problem of minimizing the functional 
\begin{equation}
\int_{U}\frac{1}{2}|Dv|^{2}-v\,dx\label{eq: elast-plast functl}
\end{equation}
over the set 
\[
W_{B_{1}}:=\{v\in H_{0}^{1}(U):|Dv|\le1\textrm{ a.e.}\}.
\]
Here $U$ is a bounded open set in $\mathbb{R}^{n}$. (In the physical
problem $n=2$.) This problem is equivalent to finding $u\in W_{B_{1}}$
that satisfies the variational inequality 
\[
\int_{U}Du\cdot D(v-u)-(v-u)\,dx\ge0\qquad\textrm{ for every }v\in W_{B_{1}}.
\]

\citet{MR0239302} proved the $W^{2,p}$ regularity for the elastic-plastic
torsion problem. \citet{MR513957} obtained its optimal $W^{2,\infty}$
regularity. \citet{MR0385296} proved $W^{2,p}$ regularity for the
solution of a quasilinear variational inequality subject to the same
constraint as in the elastic-plastic torsion problem. \citet{MR697646}
proved $W^{2,p}$ regularity for the solution of a linear variational
inequality subject to a $C^{2}$ strictly convex gradient constraint.
\citet{MR529814} considered linear elliptic equations with pointwise
constraints of the form $|Dv(\cdot)|\le g(\cdot)$, and proved $W^{2,p}$
regularity for them. He also obtained $W^{2,\infty}$ regularity under
some additional restrictions. Those restrictions were removed by \citet{MR607553},
and some extended results were obtained by \citet{MR693645}. \citet{MR1310935,MR1315349}
proved $C^{1,\alpha}$ regularity for the solution to a quasilinear
variational inequality subject to a $C^{2}$ strictly convex gradient
constraint, and allowed the operator to be degenerate of the $p$-Laplacian
type. The paper \citep{MR1315349} is also the only work here that
is concerned with functionals with non-quadratic $p$-growth. In \citep{MR1},
following the approach of \citep{MR513957}, we obtained $W^{2,\infty}$
regularity for the minimizers of the functional (\ref{eq: elast-plast functl})
subject to gradient constraints satisfying some mild regularity.

Recently, there has been new interest in these types of problems.
\citet{Hynd} studied fully nonlinear equations with strictly convex
gradient constraints, which appear in singular stochastic control.
They obtained $W^{2,p}$ regularity in general, and $W^{2,\infty}$
regularity with some extra assumptions. \citet{Hynd-2012,Hynd2017}
also considered eigenvalue problems for equations with gradient constraints.
By using infinite dimensional duality, \citet{giuffre2015lagrange}
studied the Lagrange multipliers of quasilinear variational inequalities
subject to the same constraint as in the elastic-plastic torsion problem.
\citet{MR2605868} investigated the minimizers of some  functionals
subject to gradient constraints, arising in the study of random surfaces.
In their work, the functional is allowed to have certain kinds of
singularities, and there is no particular assumption about its growth
condition. Also, the constraints are given by convex polygons; so
they are not strictly convex. They showed that in two dimensions,
the minimizer is $C^{1}$ away from the obstacles. (Under mild conditions,
a variational problem with gradient constraint is equivalent to a
double obstacle problem. For the details see Appendix \ref{sec: Local-Optimal-Reg}.)

In this paper, we obtain the optimal $W^{2,\infty}$ regularity for
the minimizers of a large class of functionals subject to arbitrary
convex gradient constraints. We have only considered functionals with
quadratic growth here. However, we do not assume any regularity of
the constraints; so the constraints can be nonsmooth, and they need
not be strictly convex. We also show that the optimal regularity holds
up to the boundary, when the domain is smooth enough. Although our
functionals are smooth, we hope that our study sheds some new light
on the above-mentioned problem about random surfaces. Also, in \citep{SAFDARI2021358},
we have used the tools developed in this paper to obtain $W^{2,\infty}$
regularity for fully nonlinear elliptic equations with non-strictly
convex gradient constraints. These types of equations emerge in the
study of some singular stochastic control problems appearing in financial
models with transaction costs; see for example \citep{barles1998option,possamai2015homogenization}.

In addition to the works on the regularity of the elastic-plastic
torsion problem, \citet{MR0412940,MR0521411}, \citet{MR534111},
\citet{MR552267}, and \citet{MR563207}, have worked on the regularity
and the shape of its free boundary, i.e. the boundary of the set $\{|Du|<1\}$.
These works can also be found in \citep{MR679313}. In \citep{Safdari20151,safdari2017shape},
we extended some of these results to the more general case where the
functional is unchanged, but the constraint is given by the $p$-norm
\[
(|D_{1}v|^{p}+|D_{2}v|^{p})^{\frac{1}{p}}\le1.
\]
Similarly to \citep{Safdari20151}, our results in this paper can
be applied to imply the regularity of the free boundary in two dimensions,
when the functional (i.e. $F,g$ in (\ref{eq: fnctnl})) is analytic.
But the more general case requires extra analysis, so we leave the
question of the free boundary's regularity to future works.

Let us introduce the problem in more detail. Let $K$ be a compact
convex subset of $\mathbb{R}^{n}$ whose interior contains the origin.
We recall from convex analysis (see \citep{MR3155183}) that the \textbf{gauge}
function of $K$ is the convex function 
\begin{equation}
\gamma_{K}(x):=\inf\{\lambda>0:x\in\lambda K\}.\label{eq: gaug}
\end{equation}
The gauge function $\gamma_{K}$ is subadditive and positively 1-homogeneous,
so it looks like a norm on $\mathbb{R}^{n}$, except that $\gamma_{K}(-x)$
is not necessarily the same as $\gamma_{K}(x)$. Note that as $K$
is closed, $K=\{\gamma_{K}\le1\}$; and as $K$ has nonempty interior,
$\partial K=\{\gamma_{K}=1\}$.

Another notion is that of the \textbf{polar} of $K$ 
\begin{equation}
K^{\circ}:=\{x:\langle x,y\rangle\leq1\,\textrm{ for all }y\in K\},\label{eq: K0}
\end{equation}
where $\langle\,,\rangle$ is the standard inner product on $\mathbb{R}^{n}$.
$K^{\circ}$, too, is a compact convex set containing the origin as
an interior point.

Let $U\subset\mathbb{R}^{n}$ be a bounded open set with Lipschitz
boundary. Let $u$ be the minimizer of 
\begin{equation}
J[v]=J[v;U]:=\int_{U}F(Dv)+g(v)\,dx,\label{eq: fnctnl}
\end{equation}
over 
\begin{equation}
W_{K^{\circ},\varphi}=W_{K^{\circ},\varphi}(U):=\{v\in H^{1}(U):Dv\in K^{\circ}\textrm{ a.e., }v=\varphi\textrm{ on }\partial U\}.\label{eq: W_K}
\end{equation}
Here $\varphi:\mathbb{R}^{n}\to\mathbb{R}$ is a continuous function,
and the equality of $v,\varphi$ on $\partial U$ is in the sense
of trace. In order to ensure that $W_{K^{\circ},\varphi}$ is nonempty
we assume that 
\begin{equation}
-\gamma_{K}(y-x)\le\varphi(x)-\varphi(y)\le\gamma_{K}(x-y),\label{eq: phi Lip}
\end{equation}
for all $x,y\in\R^{n}$. Then by Lemma 2.1 of \citep{MR1797872} this
property implies that $\varphi$ is Lipschitz and $D\varphi\in K^{\circ}$
a.e.; so $\varphi\in W_{K^{\circ},\varphi}$. Also note that $Dv\in K^{\circ}$
is equivalent to $\gamma_{K^{\circ}}(Dv)\le1$. Thus $\gamma_{K^{\circ}}$
defines the gradient constraint.

We will show that (Proposition \ref{prop: equiv}) $u$ is also the
unique minimizer of $J$ over 
\begin{equation}
W_{\bar{\rho},\rho}=W_{\bar{\rho},\rho}(U):=\{v\in H^{1}(U):-\bar{\rho}\le v\leq\rho\textrm{ a.e., }v=\varphi\textrm{ on }\partial U\},\label{eq: W_rho}
\end{equation}
where 
\begin{align}
 & \rho(x)=\rho_{K,\varphi}(x;U):=\underset{y\in\partial U}{\min}[\gamma_{K}(x-y)+\varphi(y)],\nonumber \\
 & \bar{\rho}(x)=\bar{\rho}_{K,\varphi}(x;U):=\underset{y\in\partial U}{\min}[\gamma_{K}(y-x)-\varphi(y)].\label{eq: rho}
\end{align}
It is well known (see {[}\citealp{MR667669}, Section 5.3{]}) that
$\rho$ is the unique viscosity solution of the Hamilton-Jacobi equation
\begin{equation}
\begin{cases}
\gamma_{K^{\circ}}(Dv)=1 & \textrm{in }U,\\
v=\varphi & \textrm{on }\partial U.
\end{cases}\label{eq: H-J eq}
\end{equation}
Now, note that $-K$ is also a compact convex set whose interior contains
the origin. We also have $\bar{\rho}_{K,\varphi}=\rho_{-K,-\varphi}$,
since $\gamma_{-K}(\cdot)=\gamma_{K}(-\,\cdot)$. Thus we have a similar
characterization for $\bar{\rho}$ too.%

To simplify the notation, we will use the following conventions 
\[
\gamma:=\gamma_{K},\qquad\gamma^{\circ}:=\gamma_{K^{\circ}},\qquad\bar{\gamma}:=\gamma_{-K}.
\]
Thus in particular we have $\bar{\gamma}(x)=\gamma(-x)$. 
\begin{defn}
When $\rho(x)=\gamma(x-y)+\varphi(y)$ for some $y\in\partial U$,
we call $y$ a \textbf{$\boldsymbol{\rho}$-closest} point to $x$
on $\partial U$. Similarly, when $\bar{\rho}(x)=\gamma(y-x)-\varphi(y)$
for some $y\in\partial U$, we call $y$ a \textbf{$\boldsymbol{\bar{\rho}}$-closest}
point to $x$ on $\partial U$.
\end{defn}
\begin{rem*}
As shown in the proof of Proposition \ref{prop: equiv}, for $y\in\partial U$
we have $\rho(y)=\varphi(y)$ and $\bar{\rho}(y)=-\varphi(y)$. Therefore
$y$ is a $\rho$-closest point and a $\bar{\rho}$-closest point
on $\partial U$ to itself.
\end{rem*}

Next, we generalize the notion of ridge introduced by \citet{MR0184503},
and \citet{MR534111}. Intuitively, the $\rho$-ridge is the set of
singularities of $\rho$.
\begin{defn}
\label{def: ridge}The \textbf{$\boldsymbol{\rho}$-ridge} of $U$
is the set of all points $x\in U$ where $\rho(x)$ is not $C^{1,1}$
in any neighborhood of $x$. We denote it by 
\[
R_{\rho}.
\]
As shown in Lemma \ref{lem: rho not C1}, when $\gamma$ is strictly
convex and the strict Lipschitz property (\ref{eq: phi strct Lip})
for $\varphi$ holds, the points with more than one $\rho$-closest
point on $\partial U$ belong to $\rho$-ridge. This subset of the
$\rho$-ridge is denoted by 
\[
R_{\rho,0}.
\]
Similarly we define $R_{\bar{\rho}},R_{\bar{\rho},0}$.%
\end{defn}

The following definition is motivated by the physical properties of
the elastic-plastic torsion problem.
\begin{defn}
\label{def: plastic}Let 
\begin{eqnarray*}
P^{+}:=\{x\in U:u(x)=\rho(x)\}, &  & P^{-}:=\{x\in U:u(x)=-\bar{\rho}(x)\}.
\end{eqnarray*}
Then $P:=P^{+}\cup P^{-}$ is called the \textbf{plastic} region;
and 
\[
E:=\{x\in U:-\bar{\rho}(x)<u(x)<\rho(x)\}
\]
is called the \textbf{elastic} region. We also define the \textbf{free
boundary} to be $\Gamma:=\partial E\cap U$.
\end{defn}
\begin{rem*}
Suppose $u$ is locally $C^{1,1}$, which for example is true under
the assumptions of Theorem \ref{thm: Reg NonConv dom}, Theorem \ref{thm: Reg Conv dom},
or Theorem \ref{thm: Reg u}. Then similarly to (\ref{eq: diff ineq u_e}),
we can show that 
\begin{equation}
\begin{cases}
-D_{i}(D_{i}F(Du))+g'(u)=0 & \textrm{ in }E,\\
-D_{i}(D_{i}F(Du))+g'(u)\le0 & \textrm{ a.e. on }P^{+},\\
-D_{i}(D_{i}F(Du))+g'(u)\ge0 & \textrm{ a.e. on }P^{-}.
\end{cases}\label{eq: diff ineq}
\end{equation}
\end{rem*}
The paper is organized as follows. In Section \ref{sec: Notation-and-Prelim}
we introduce some notation, and we state some preliminary results.
We first show that $u$ is also the minimizer of $J$ over $W_{\bar{\rho},\rho}$,
i.e. it is the solution of a double obstacle problem. Then we review
some well-known facts about the regularity of $K$, and its relation
to the regularity of $K^{\circ},\gamma,\gamma^{\circ}$.  Then we
need to show that $u$ cannot touch the obstacles at any of their
points of singularity. To accomplish this, we have to study the function
$\rho$ more carefully. This has been done in Section \ref{sec: Reg Opstacles}.
We will compute the derivatives of $\rho$, and by using them we will
characterize the $\rho$-ridge. Results of this nature have appeared
in the literature before. For example, \citet{MR2336304} studied
this problem in the case of $\varphi=0$.  But the novelty of our
work is that we were able to find an explicit formula for $D^{2}\rho$.

Suppose $\partial U$ is at least $C^{2}$. Let $\nu$ be the inward
unit normal to $\partial U$. We will show that under appropriate
assumptions, for every $y\in\partial U$ there is a unique scalar
$\lambda(y)>0$ such that 
\[
\gamma^{\circ}\big(D\varphi(y)+\lambda(y)\nu(y)\big)=1.
\]
Then we set 
\begin{equation}
\mu(y):=D\varphi(y)+\lambda(y)\nu(y).\label{eq: mu}
\end{equation}
We also set 
\begin{equation}
X:=\frac{1}{\langle D\gamma^{\circ}(\mu),\nu\rangle}D\gamma^{\circ}(\mu)\otimes\nu,\label{eq: X}
\end{equation}
where $a\otimes b$ is the rank 1 matrix whose action on a vector
$z$ is $\langle z,b\rangle a$. Then in Proposition \ref{prop: D rho}
we will show that $\rho$ is differentiable at $x$ if and only if
$x\in U-R_{\rho,0}$. And in that case we have 
\begin{equation}
D\rho(x)=\mu(y),\label{eq: D rho (x)}
\end{equation}
where $y$ is the unique $\rho$-closest point to $x$ on $\partial U$.
\begin{rem*}
We will mostly state our results about $\gamma,\rho$, but it is obvious
that they also hold for $\bar{\gamma},\bar{\rho}$.
\end{rem*}
\begin{thm}
\label{thm: rho is C2 at y}Suppose the Assumption \ref{assu: 2}
holds. Let $y\in\partial U$. Then there is an open ball $B_{r}(y)$
such that $\rho$ is $C^{k,\alpha}$ on $\overline{U}\cap B_{r}(y)$.
Furthermore, $y$ is the $\rho$-closest point to some points in $U$,
and we have 
\begin{equation}
D\rho(y)=\mu(y).\label{eq: D rho (y)}
\end{equation}
In addition we have 
\begin{equation}
D^{2}\rho(y)=(I-X^{T})\big(D^{2}\varphi(y)+\lambda(y)D^{2}d(y)\big)(I-X),\label{eq: D2 rho (y)}
\end{equation}
where $I$ is the identity matrix, $d$ is the Euclidean distance
to $\partial U$, and $X$ is given by (\ref{eq: X}). Furthermore
we have 
\begin{equation}
D^{2}\rho(y)D\gamma^{\circ}(\mu(y))=0.\label{eq: D2 rho Dg0 =00003D 0}
\end{equation}
\end{thm}
\begin{rem*}
As a consequence of this theorem we get that $R_{\rho}$, and therefore
$R_{\rho,0}$, have a positive distance from $\partial U$. 
\end{rem*}
\begin{thm}
\label{thm: rho is C^2}Suppose the Assumption \ref{assu: 2} holds.
Suppose $x\in U-R_{\rho,0}$, and let $y$ be the unique $\rho$-closest
point to $x$ on $\partial U$. Let 
\begin{align}
 & W=W(y):=-D^{2}\gamma^{\circ}(\mu(y))D^{2}\rho(y),\nonumber \\
 & Q=Q(x):=I-\big(\rho(x)-\varphi(y)\big)W,\label{eq: W,Q}
\end{align}
where $I$ is the identity matrix. If $\det Q\ne0$ then $\rho$ is
$C^{k,\alpha}$ on a neighborhood of $x$. In addition we have 
\begin{equation}
D^{2}\rho(x)=D^{2}\rho(y)Q(x)^{-1}.\label{eq: D2 rho (x)}
\end{equation}
\end{thm}

\begin{thm}
\label{thm: ridge}Suppose the Assumption \ref{assu: 2} holds. Suppose
$x\in U-R_{\rho,0}$, and let $y$ be the unique $\rho$-closest point
to $x$ on $\partial U$. Then 
\[
x\in R_{\rho}\textrm{ if and only if }\det Q(x)=0,
\]
where $Q$ is defined by (\ref{eq: W,Q}).
\end{thm}
When $\varphi=0$, the function $\rho$ is the distance to $\partial U$
with respect to the Minkowski distance defined by $\gamma$. So this
case has a geometric interpretation. An interesting simplification
that happens here is that $\mu$ becomes a multiple of $\nu$. Another
interesting fact is that in this case the eigenvalues of $W$ coincide
with the notion of curvature of $\partial U$ with respect to some
Finsler structure. For the details see \citep{MR2336304}. Let us
also mention that in this case we have 
\begin{align*}
D\rho(x) & =\mu(y)=\frac{1}{\gamma^{\circ}(\nu)}\nu(y)=\frac{1}{\gamma^{\circ}(Dd)}Dd(y)\\
 & =\frac{1}{\gamma^{\circ}(Dd)}Dd\big(x-\rho(x)D\gamma^{\circ}(\mu(y))\big)=\frac{1}{\gamma^{\circ}(Dd)}Dd\big(x-\rho(x)D\gamma^{\circ}(D\rho(x))\big).
\end{align*}
Here we have used the classical formula $Dd=\nu$ for the Euclidean
distance $d$. We have also used the relation (\ref{eq: parametrize by rho})
between $x,y$, and the formula $D\rho=\mu$, to eliminate $y$ from
the above equation; and this is not possible when $\varphi$ is nonzero.
Now we can differentiate the above equation to find $D^{2}\rho(x)$.
This approach is how we first found $D^{2}\rho(x)$. But when $\varphi$
is nonzero this approach does not work. However, we were able to find
a more direct way to compute $D^{2}\rho(x)$ by using the inverse
function theorem, and differentiating the relation $D\rho(x)=\mu(y)$. 

The formula (\ref{eq: D2 rho (x)}) for $D^{2}\rho$ is very crucial
in our analysis, and it has been used several times in this paper.
To the best of author's knowledge, formulas of this kind have not
appeared in the literature before, except for the simple case where
$\rho$ is the Euclidean distance to the boundary. (Although, some
special two dimensional cases also appeared in our earlier works \citep{Safdari20151,Paper-4}.)
One of the main applications of this formula is in Lemma \ref{lem: D2 rho decreas},
which implies that $D^{2}\rho$ attains its maximum on $\partial U$.
This interesting property is actually a consequence of a more general
property of the solutions to Hamilton-Jacobi equations (remember that
$\rho$ is the unique viscosity solution of the Hamilton-Jacobi equation
(\ref{eq: H-J eq})). Let $v$ be the solution to a Hamilton-Jacobi
equation with convex Hamiltonian. Suppose $v$ is smooth enough on
a neighborhood of one of its characteristic curves. Then for every
vector $\xi$, $D_{\xi\xi}^{2}v$ decreases along that characteristic
curve, as we move away from the boundary. We have derived this in
the next paragraph. Surprisingly, this monotonicity property, although
very simple to deduce, has not appeared in the literature before,
again to the best of author's knowledge. However, it has been known
to some experts, as conveyed to the author through personal communications.

\textbf{Monotonicity of the second derivative of the solutions to
Hamilton-Jacobi equations:} Suppose $v$ satisfies the equation $H(x,v,Dv)=0$,
where $H$ is a convex function in all of its arguments. Let $p$
be the variable in $H$ for which we substitute $Dv$. Let $x(s)$
be a characteristic curve of the equation. Then we have $\dot{x}=D_{p}H$.
Let us assume that $v$ is $C^{3}$ on a neighborhood of the image
of $x(s)$. Let 
\[
q(s):=D_{\xi\xi}^{2}v(x(s))=\xi_{i}\xi_{j}D_{ij}^{2}v,
\]
for some vector $\xi$. Then we have 
\[
\dot{q}=\xi_{i}\xi_{j}D_{ijk}^{3}v\,\dot{x}^{k}=\xi_{i}\xi_{j}D_{ijk}^{3}vD_{p_{k}}H.
\]
On the other hand, if we differentiate the equation we get $D_{x_{i}}H+D_{v}HD_{i}v+D_{p_{k}}HD_{ik}^{2}v=0$.
And if we differentiate one more time we get 
\begin{align*}
 & D_{x_{i}x_{j}}^{2}H+D_{x_{i}v}^{2}HD_{j}v+D_{x_{i}p_{k}}^{2}HD_{jk}^{2}v+D_{vx_{j}}^{2}HD_{i}v+D_{vv}^{2}HD_{i}vD_{j}v+D_{vp_{k}}^{2}HD_{i}vD_{jk}^{2}v\\
 & \qquad\qquad+D_{v}HD_{ij}^{2}v+D_{p_{k}x_{j}}^{2}HD_{ik}^{2}v+D_{p_{k}v}^{2}HD_{j}vD_{ik}^{2}v+D_{p_{k}p_{l}}^{2}HD_{jl}^{2}vD_{ik}^{2}v+D_{p_{k}}HD_{ijk}^{3}v=0.
\end{align*}
Now if we multiply the above expression by $\xi_{i}\xi_{j}$, and
sum over $i,j$, we obtain 
\begin{align}
\dot{q} & =-\begin{bmatrix}\xi^{T} & \langle\xi,Dv\rangle & \xi^{T}D^{2}v\end{bmatrix}\begin{bmatrix}D_{xx}^{2}H & D_{xv}^{2}H & D_{xp}^{2}H\\
D_{vx}^{2}H & D_{vv}^{2}H & D_{vp}^{2}H\\
D_{px}^{2}H & D_{pv}^{2}H & D_{pp}^{2}H
\end{bmatrix}\begin{bmatrix}\xi\\
\langle\xi,Dv\rangle\\
D^{2}v\xi
\end{bmatrix}-D_{v}Hq\label{eq: ODE D2 rho}\\
 & =-\eta^{T}D^{2}H\eta-D_{v}Hq,\nonumber 
\end{align}
where $\eta:=\begin{bmatrix}\xi^{T} & \langle\xi,Dv\rangle & \xi^{T}D^{2}v\end{bmatrix}^{T}$.
Hence we have $\dot{q}\le-D_{v}Hq$, since $H$ is convex. Thus by
Gronwall's inequality we obtain 
\[
q(s)\le q(0)e^{-\int_{0}^{s}D_{v}Hd\tau}.
\]
In particular when $D_{v}H\ge0$, i.e. when $H$ is increasing in
$v$, we have 
\[
D_{\xi\xi}^{2}v(x(s))=q(s)\le q(0)=D_{\xi\xi}^{2}v(x(0)).
\]

\begin{rem*}
Let $A$ be a symmetric positive semidefinite matrix. Then the more
general inequality 
\[
\mathrm{tr}[AD^{2}v(x(s))]\le\mathrm{tr}[AD^{2}v(x(0))]
\]
follows easily; because we have $\mathrm{tr}[AD^{2}v]=\sum a_{j}D_{\xi_{j}\xi_{j}}^{2}v$,
where $\xi_{1},\cdots,\xi_{n}$ is an orthonormal basis of eigenvectors
of $A$ corresponding to the nonnegative eigenvalues $a_{1},\dots,a_{n}$.
Note that in Lemma \ref{lem: D2 rho decreas} we do not need the $C^{3}$
regularity of $\rho$. We suspect that in the general case, it is
possible to weaken the regularity assumption on $v$ too.
\end{rem*}
\begin{rem*}
Also notice that the ODE system (\ref{eq: ODE D2 rho}) can be used
to compute an explicit formula for $D^{2}v$. For example when $H$
does not depend on $x,v$, we have the following Riccati type equation
for $D^{2}v$: 
\[
\frac{d}{ds}D^{2}v(x(s))=-D^{2}vD_{pp}^{2}HD^{2}v.
\]
Then it is easy to show that $D^{2}v$ must be given by a formula
similar to (\ref{eq: D2 rho (x)}). However, as we said above, this
approach requires imposing extra regularity on $v$. Also, solving
the system (\ref{eq: ODE D2 rho}) in general seems to be a daunting
task.
\end{rem*}
In the rest of Section \ref{sec: Reg Opstacles}, we obtain several
other interesting facts about $\rho$ and the $\rho$-ridge. See Theorem
\ref{thm: det Q > 0}, Proposition \ref{prop: R =00003D R0 bar},
and the remarks after them. Here, among other things, we show that
the eigenvalues of $Q$ are positive; and we have $R_{\rho}=\overline{R}_{\rho,0}$.
We also make some remarks about the Hausdorff dimension of $R_{\rho}$.
Finally in Section \ref{sec: Global-Optimal-Reg} we prove the main
results of this paper, aka Theorems \ref{thm: Reg NonConv dom} and
\ref{thm: Reg Conv dom}. First, we employ our detailed knowledge
about the $\rho$-ridge, which we obtained in Theorem \ref{thm: ridge},
to show that $u$ cannot touch the obstacles at any of their points
of singularity, i.e. 
\begin{thm}
\label{thm: ridge is elastic}Suppose the Assumptions \ref{assu: 1},\ref{assu: 2}
hold. Then we have 
\begin{eqnarray*}
R_{\rho}\cap P^{+}=\emptyset, & \hspace{2cm} & R_{\bar{\rho}}\cap P^{-}=\emptyset.
\end{eqnarray*}
\end{thm}
Before stating the next theorem, let us review some well-known facts
from convex analysis. Consider a compact convex set $K$. Let $x\in\partial K$,
and $\mathrm{v}\in\R^{n}-\{0\}$. We say the hyperplane 
\begin{equation}
H_{x,\mathrm{v}}:=\{y\in\R^{n}:\langle y-x,\mathrm{v}\rangle=0\}\label{eq: hyperplane}
\end{equation}
is a \textit{supporting hyperplane} of $K$ at $x$ if $K\subset\{y:\langle y-x,\mathrm{v}\rangle\le0\}$.
In this case we say $\mathrm{v}$ is an \textit{outer normal vector}
of $K$ at $x$. The \textit{normal cone} of $K$ at $x$ is the closed
convex cone 
\begin{equation}
N(K,x):=\{0\}\cup\{\mathrm{v}\in\R^{n}-\{0\}:\mathrm{v}\textrm{ is an outer normal vector of }K\textrm{ at }x\}.\label{eq: normal cone}
\end{equation}
It is easy to see that when $\partial K$ is $C^{1}$ we have 
\[
N(K,x)=\{tD\gamma(x):t\ge0\}.
\]
For more details see {[}\citealp{MR3155183}, Sections 1.3 and 2.2{]}.
\begin{thm}
\label{thm: Reg NonConv dom}Suppose the Assumption \ref{assu: 1}
holds. Also suppose that 
\begin{enumerate}
\item[\upshape{(a)}] $K\subset\R^{n}$ is a compact convex set whose interior contains
the origin.
\item[\upshape{(b)}] $U\subset\R^{n}$ is a bounded open set, and $\partial U$ is $C^{2,\alpha}$,
for some $\alpha>0$.
\item[\upshape{(c)}] $\varphi:\R^{n}\to\R$ is a $C^{2,\alpha}$ function, such that $\gamma^{\circ}(D\varphi)\le1$.
And if for some $y\in\partial U$ we have $\gamma^{\circ}(D\varphi(y))=1$
then we must have 
\begin{equation}
\langle\mathrm{v},\nu(y)\rangle\ne0,\label{eq: 0 in reg Nonconv dom}
\end{equation}
for every nonzero $\mathrm{v}\in N(K^{\circ},D\varphi(y))$. %
\end{enumerate}
Let $u$ be the minimizer of $J$ over $W_{K^{\circ},\varphi}(U)$.
Then we have 
\[
u\in W^{2,\infty}(U)=C^{1,1}(\overline{U}).
\]
\end{thm}
\begin{rem*}
Note that we are not assuming any regularity about $\partial K$ or
$\partial K^{\circ}$. In particular, $K^{\circ}$, which defines
the gradient constraint, need not be strictly convex. 
\end{rem*}
\begin{rem*}
Also note that if $\gamma^{\circ}(D\varphi)<1$ then we do not need
to impose any other restriction on $\varphi$. It is also obvious
that if $\gamma^{\circ}(D\varphi)\le1$ then we can approximate $\varphi$
with functions that satisfy $\gamma^{\circ}(D\,\cdot)<1$. So, intuitively,
most admissible boundary conditions $\varphi$ satisfy the conditions
of the theorem.
\end{rem*}
\begin{rem*}
Let us further elaborate on the restrictions imposed on $D\varphi$,
and present a geometric interpretation for it. Remember that $\mu=D\varphi+\lambda\nu$,
and we have $\gamma^{\circ}(\mu)=1$. We will show that (Lemma \ref{lem: K-normal})
for a point $y\in\partial U$, $D\gamma^{\circ}(\mu)$ is the direction
along which lie the points in $U$ that have $y$ as their $\rho$-closest
point. Note that we have $D\gamma^{\circ}(\mu)\in N(K^{\circ},\mu)$.
Now when $\gamma^{\circ}(D\varphi)=1$, $D\varphi$ plays the role
of $\mu$. And $\mathrm{v}\in N(K^{\circ},D\varphi)$ plays the role
of $D\gamma^{\circ}(\mu)$. Hence we need to impose the conditions
of the theorem in order to be sure that there is a direction along
which we can enter $U$ and hit the points whose $\rho$-closest point
is $y$. (We should mention that when $\gamma^{\circ}(D\varphi)=1$,
then $\mu$ is not necessarily equal to $D\varphi$. However, this
case can also be dealt with using our assumptions in the theorem;
and we do not need new assumptions for it.)
\end{rem*}
The idea of the proof of the above theorem is to approximate $K^{\circ}$
with smoother convex sets. Then, as it is common in the study of the
regularity of PDEs, we have to find uniform bounds for the various
norms of the approximations to $u$. Here, among other estimations,
we will use the fact that the second derivative of the approximations
to $\rho$ attain their maximums on $\partial U$. Let us also mention
that in order to get the regularity up to the boundary, we need to
use the result of \citet{indrei2016nontransversal}, which is a generalization
of the work of \citet{MR3198649}. However, their result is only about
flat boundaries. So we have to modify it with standard techniques,
to be able to handle arbitrary smooth boundaries. Next, we will use
similar ideas in Theorem \ref{thm: Reg Conv dom}, to obtain a local
regularity result. We will show that when $U$ is convex, $u$ belongs
to $W_{\mathrm{loc}}^{2,\infty}(U)$, without assuming any regularity
of the gradient constraint, nor of the $\partial U$. The idea of
the proof is to approximate both $K^{\circ},U$ with smoother convex
sets. 
\begin{thm}
\label{thm: Reg Conv dom}Suppose the Assumption \ref{assu: 1} holds.
Also suppose that 
\begin{enumerate}
\item[\upshape{(a)}] $K\subset\R^{n}$ is a compact convex set whose interior contains
the origin.
\item[\upshape{(b)}] $U\subset\R^{n}$ is a bounded convex open set. 
\item[\upshape{(c)}] $\varphi:\R^{n}\to\R$ is a $C^{2}$ function, such that $\gamma^{\circ}(D\varphi)\le1$.
And either $\varphi$ is linear, or the following condition holds:
If for some $y\in\partial U$ we have $\gamma^{\circ}(D\varphi(y))=1$
then we must have 
\begin{equation}
\langle\mathrm{v},\mathrm{w}\rangle\ne0,\label{eq: 0 in reg Conv dom}
\end{equation}
for every nonzero $\mathrm{v}\in N(K^{\circ},D\varphi(y))$ and $\mathrm{w}\in N(U,y)$.
\end{enumerate}
Let $u$ be the minimizer of $J$ over $W_{K^{\circ},\varphi}(U)$.
Then we have 
\[
u\in W_{\mathrm{loc}}^{2,\infty}(U)=C_{\mathrm{loc}}^{1,1}(U).
\]
\end{thm}
\begin{rem*}
Note that if $\gamma^{\circ}(D\varphi)<1$ then we do not need to
impose any other restriction on $\varphi$. And as we mentioned above,
intuitively, most admissible boundary conditions $\varphi$ satisfy
the conditions of the theorem. Also, the restrictions on $D\varphi$
has the same geometric interpretation as before. The only difference
is that here $\mathrm{w}$ plays the role of normal $\nu$ to $\partial U$.
\end{rem*}
At the end, in Appendix \ref{sec: Local-Optimal-Reg}, we obtain a
standard local regularity result which we have used in the article.
We prove that the minimizer $u$ is in $W_{\mathrm{loc}}^{2,\infty}$,
when $\partial U$ is Lipschitz, and we have an upper bound on the
weak second derivative of $\gamma$. To this end, we mollify the obstacles,
and we solve the double obstacle problems with these smooth mollified
obstacles. This gives us functions $u_{\varepsilon}$ which approximate
$u$. By using penalization technique, we show that $u_{\varepsilon}$
is in $W_{\mathrm{loc}}^{2,p}$. Then we show that $u_{\varepsilon}$'s
have a uniform bound in $W_{\mathrm{loc}}^{2,p}$. Hence we can conclude
that $u$ is also in $W_{\mathrm{loc}}^{2,p}$. The methods employed
here are classical, but to the best of author\textquoteright s knowledge
the results have not appeared elsewhere. Nevertheless, we include
the proofs here for completeness. Then in Theorem \ref{thm: Reg u},
using the uniform bound on $D^{2}u_{\varepsilon}$, we can show that
$u$ belongs to $W_{\mathrm{loc}}^{2,\infty}$. This last step was
made possible by the work of \citet{MR3198649}, and its generalization
by \citet{Indrei-Minne}. Finally, in Appendix \ref{sec: Ridge, Elast, Plast},
we collect some elementary results about the ridge, and the elastic
and plastic regions.

\section{\label{sec: Notation-and-Prelim}Notation and Preliminaries}

First let us introduce some notation. 
\begin{enumerate}
\item $d(x):=\min_{y\in\partial U}|x-y|$ : the Euclidean distance to $\partial U$. 
\item $[x,y],\,]x,y[,\,[x,y[,\,]x,y]$ : the closed, open, and half-open
line segments with endpoints $x,y$.
\item We denote by $C^{\omega}$ the space of analytic functions (or submanifolds);
so in the following when we talk about $C^{k,\alpha}$ regularity
with $k$ greater than some fixed integer, we are also including $C^{\infty}$
and $C^{\omega}$.%
\item We will use the convention of summing over repeated indices.
\end{enumerate}

Recall that the gauge function $\gamma$ satisfies 
\begin{align*}
 & \gamma(rx)=r\gamma(x),\\
 & \gamma(x+y)\le\gamma(x)+\gamma(y),
\end{align*}
for all $x,y\in\mathbb{R}^{n}$ and $r\ge0$. Also, note that as $B_{c}(0)\subseteq K\subseteq B_{C}(0)$
for some $C\ge c>0$, we have 
\[
\frac{1}{C}|x|\le\gamma(x)\le\frac{1}{c}|x|,
\]
for all $x\in\mathbb{R}^{n}$. 

It is well known that for all $x,y\in\mathbb{R}^{n}$, we have 
\begin{equation}
\langle x,y\rangle\leq\gamma(x)\gamma^{\circ}(y).\label{eq: gen Cauchy-Schwartz}
\end{equation}
In fact, more is true and we have 
\begin{equation}
\gamma^{\circ}(y)=\underset{x\ne0}{\max}\frac{\langle x,y\rangle}{\gamma(x)}.\label{eq: gen Cauchy-Schwartz 2}
\end{equation}
For a proof of this, see page 54 of \citep{MR3155183}.

It is easy to see that the the strict convexity of $K$ (which means
that $\partial K$ does not contain any line segment) is equivalent
to the strict convexity of $\gamma$. By homogeneity of $\gamma$,
the latter is equivalent to 
\[
\gamma(x+y)<\gamma(x)+\gamma(y)
\]
when $x\ne cy$ and $y\ne cx$ for any $c\ge0$. 

Moreover, from the definition of $\rho$ we easily obtain (see the
proof of Proposition \ref{prop: equiv})
\begin{equation}
-\gamma(x-y)\le\rho(y)-\rho(x)\le\gamma(y-x),\label{eq: rho Lip}
\end{equation}
for all $x,y\in\R^{n}$. The above inequality also holds if we replace
$\rho,\gamma$ with $\bar{\rho},\bar{\gamma}$. Thus in particular,
$\rho,\bar{\rho}$ are Lipschitz continuous.
\begin{lem}
\label{lem: cont of y}Suppose $x_{i}\in\overline{U}$ converges to
$x\in\overline{U}$, and $y_{i}\in\partial U$ is a (not necessarily
unique) $\rho$-closest point to $x_{i}$.
\begin{enumerate}
\item[\upshape{(a)}] If $y_{i}$ converges to $\tilde{y}\in\partial U$, then $\tilde{y}$
is one of the $\rho$-closest points on $\partial U$ to $x$.
\item[\upshape{(b)}] If $y\in\partial U$ is the unique $\rho$-closest point to $x$,
then $y_{i}$ converges to $y$.
\end{enumerate}
\end{lem}
\begin{proof}
This lemma is a simple consequence of the continuity of $\gamma,\rho$,
and compactness of $\partial U$. For (a) we have 
\[
\gamma(x-\tilde{y})=\lim\gamma(x_{i}-y_{i})=\lim\rho(x_{i})=\rho(x).
\]
Hence $\tilde{y}$ is a $\rho$-closest point to $x$.

Now to prove (b) suppose to the contrary that $y_{i}\not\to y$. Then
as $\partial U$ is compact, there is a subsequence $y_{i_{k}}$ that
converges to $z\in\partial U$ where $z\ne y$. Then by (a) $z$ must
be a $\rho$-closest point to $x$, which is in contradiction with
our assumption.
\end{proof}
\begin{assumption}
\label{assu: 1}We assume that the functional $J$ is given by
\[
J[v]=J[v;U]:=\int_{U}F(Dv)+g(v)\,dx,
\]
where $F:\mathbb{R}^{n}\to\mathbb{R}$ and $g:\mathbb{R}\to\mathbb{R}$
 are $C^{2,\bar{\alpha}}$ convex functions satisfying 
\begin{align}
 & -c_{1}|z|^{q}\le g(z)\le c_{2}|z|^{2}, &  & c_{3}|Z|^{2}\le F(Z)\le c_{4}|Z|^{2},\nonumber \\
 & |g'(z)|\le c_{5}(|z|+1), &  & |DF(Z)|\le c_{6}|Z|,\label{eq: Bnds}\\
 & 0\le g''\le c_{7}, &  & c_{8}|\xi|^{2}\le D_{ij}^{2}F(Z)\,\xi_{i}\xi_{j}\le c_{9}|\xi|^{2},\nonumber 
\end{align}
for all $z\in\mathbb{R}$ and $Z,\xi\in\mathbb{R}^{n}$. Here, $c_{i}>0$,
$0<\bar{\alpha}\le1$, and $1\le q<2$. 
\end{assumption}

\begin{rem*}
Note that by our assumption, $F$ is strictly convex, and $F(0)=0$
is its unique global minimum.
\end{rem*}
\begin{rem*}
The constants $c_{1},\dots,c_{9}$ will only be used in the estimates
of Appendix \ref{sec: Local-Optimal-Reg}.
\end{rem*}
\begin{rem*}
Since $W_{K^{\circ},\varphi}$ is a nonempty closed convex set, we
can apply the direct method of the calculus of variations, to conclude
the existence of a unique minimizer $u$. See, for example, the proof
of Theorem 3.30 in \citep{MR2361288}. 
\end{rem*}
\begin{prop}
\label{prop: equiv}Suppose the Assumption \ref{assu: 1} holds. Then,
$u$ is also the minimizer of $J$ over 
\[
W_{\bar{\rho},\rho}=W_{\bar{\rho},\rho}(U):=\{v\in H^{1}(U):-\bar{\rho}\le v\leq\rho\textrm{ a.e., }v=\varphi\textrm{ on }\partial U\}.
\]
\end{prop}
\begin{rem*}
In fact we can weaken the hypothesis of this theorem, and only assume
that $F,g$ are convex and at least one of them is strictly convex.
See \citep{MR1} for details. 
\end{rem*}
\begin{proof}
As shown in \citep{MR1797872}, \citep{MR1}, $u$ is also the minimizer
of $J$ over 
\[
\{v\in H^{1}(U):u^{-}\le v\leq u^{+}\textrm{ a.e., }v=\varphi\textrm{ on }\partial U\},
\]
where $u^{-},u^{+}\in W_{K^{\circ},\varphi}$ satisfy $u^{-}\le v\le u^{+}$
for all $v\in W_{K^{\circ},\varphi}$. We claim that 
\begin{align*}
 & u^{+}(x)=\rho(x)=\underset{y\in\partial U}{\min}[\gamma(x-y)+\varphi(y)],\\
 & u^{-}(x)=-\bar{\rho}(x)=-\underset{y\in\partial U}{\min}[\gamma(y-x)-\varphi(y)].
\end{align*}
First, let us show that $\rho,-\bar{\rho}\in W_{K^{\circ},\varphi}$.
By (\ref{eq: phi Lip}), for a given $x\in\partial U$ and every $y\in\partial U$
we have 
\begin{align*}
\gamma(x-x)+\varphi(x) & =\varphi(x)\le\gamma(x-y)+\varphi(y),\\
\gamma(x-x)-\varphi(x) & =-\varphi(x)\le\gamma(y-x)-\varphi(y).
\end{align*}
Hence $\rho(x)=\varphi(x)$, and $-\bar{\rho}(x)=\varphi(x)$. Next
let $x,z\in\R^{n}$. Then due to the compactness of $\partial U$,
there is $y\in\partial U$ such that $\rho(z)=\gamma(z-y)+\varphi(y)$.
Therefore we have 
\[
\rho(x)-\rho(z)\le\gamma(x-y)+\varphi(y)-\gamma(z-y)-\varphi(y)\le\gamma(x-z).
\]
Then by Lemma 2.1 of \citep{MR1797872}, $\rho$ is Lipschitz and
$D\rho\in K^{\circ}$ a.e.. The case of $-\bar{\rho}$ is similar. 

Finally, let us show that for an arbitrary $v\in W_{K^{\circ},\varphi}$
we have $-\bar{\rho}\le v\le\rho$ a.e.. Note that $v$ is Lipschitz,
since $K^{\circ}$ is bounded and $\varphi,\partial U$ are Lipschitz.
Then similarly to the proof of Lemma 2.2 of \citep{MR1}, we can mollify
$v$ and use the mean value theorem together with the inequality (\ref{eq: gen Cauchy-Schwartz})
to obtain 
\[
v(x)-v(y)\le\gamma(x-y),
\]
when $[x,y]\subset U$. Using continuity we can allow $y$ to be
on $\partial U$ too. Hence $v(x)\le\gamma(x-y)+\varphi(y)$ whenever
$[x,y[\subset U$. Now let $z\in\partial U$, and consider the line
segment $[x,z]$. This line segment might not be entirely in $U$,
but if we let $y$ to be the closest point to $x$ on $[x,z]\cap\partial U$,
then we must have $[x,y[\subset U$. Then 
\begin{align*}
v(x) & \le\gamma(x-y)+\varphi(y)\\
 & \le\gamma(x-y)+\gamma(y-z)+\varphi(z)=\gamma(x-z)+\varphi(z).
\end{align*}
Hence $v\le\rho$. Similarly we have $-v=-\varphi$ on $\partial U$,
and $D(-v)\in-K^{\circ}=(-K)^{\circ}$ a.e.. Therefore $-v\le d_{-K,-\varphi}=\bar{\rho}$,
or $v\ge-\bar{\rho}$.
\end{proof}
\begin{rem*}
The above proof also shows that $W_{K^{\circ},\varphi}\subset W_{\bar{\rho},\rho}$,
and $-\bar{\rho},\rho\in W_{K^{\circ},\varphi}$. In addition it shows
that $-\bar{\rho}(x)\le\rho(x)$ for all $x$. 
\end{rem*}
\begin{rem*}
Note that as $u$ has bounded gradient, it is Lipschitz continuous.
Thus, for every $x\in U$ (not just for a.e. $x$) we have 
\[
-\bar{\rho}(x)\leq u(x)\leq\rho(x).
\]
As a result for every $x\in\partial U$ we have $u(x)=\varphi(x)$.
\end{rem*}

\subsection{\label{subsec: Reg gaug}Regularity of the gauge function}

Suppose that $\partial K$ is $C^{k,\alpha}$ $(k\ge1\,,\,0\le\alpha\le1)$.
Let us show that as a result, $\gamma$ is $C^{k,\alpha}$ on $\mathbb{R}^{n}-\{0\}$.
Let $r=\sigma(\theta)$ for $\theta\in\mathbb{S}^{n-1}$, be the equation
of $\partial K$ in polar coordinates. Then $\sigma$ is positive
and $C^{k,\alpha}$. To see this note that locally, $\partial K$
is given by a $C^{k,\alpha}$ equation $f(x)=0$. On the other hand
we have $x=rX(\theta)$, for some smooth function $X$. Hence we have
$f(rX(\theta))=0$; and the derivative of this expression with respect
to $r$ is 
\[
\langle X(\theta),Df(rX(\theta))\rangle=\frac{1}{r}\langle x,Df(x)\rangle.
\]
But this is nonzero since $Df$ is orthogonal to $\partial K$, and
$x$ cannot be tangent to $\partial K$ (otherwise $0$ cannot be
in the interior of $K$, as $K$ lies on one side of its supporting
hyperplane at $x$). Thus we get the desired by the Implicit Function
Theorem. Now, it is straightforward to check that for a nonzero point
in $\mathbb{R}^{n}$ with polar coordinates $(s,\phi)$ we have 
\[
\gamma((s,\phi))=\frac{s}{\sigma(\phi)}.
\]
This formula easily gives the smoothness of $\gamma$. On the other
hand, note that as $\partial K=\{\gamma=1\}$ and $D\gamma\ne0$ by
(\ref{eq: g0 (Dg)=00003D1}), $\partial K$ is as smooth as $\gamma$.

Now, suppose in addition that $K$ is strictly convex. Then $\gamma$
is strictly convex too. By Remark 1.7.14 and Theorem 2.2.4 of \citep{MR3155183},
$K^{\circ}$ is also strictly convex and its boundary is $C^{1}$.
Therefore $\gamma^{\circ}$ is strictly convex, and it is $C^{1}$
on $\mathbb{R}^{n}-\{0\}$. Hence by Corollary 1.7.3 of \citep{MR3155183},
for $x\ne0$ we have (Notice that the strict convexity of $K$ is
only needed for the existence of $D\gamma^{\circ}$.) 
\begin{eqnarray}
D\gamma(x)\in\partial K^{\circ}, &  & D\gamma^{\circ}(x)\in\partial K,\label{eq: g0 (Dg)=00003D1}
\end{eqnarray}
or equivalently 
\[
\gamma^{\circ}(D\gamma)=1,\qquad\gamma(D\gamma^{\circ})=1.
\]
In particular $D\gamma,D\gamma^{\circ}$ are nonzero on $\mathbb{R}^{n}-\{0\}$.

Let us also suppose that $k\ge2$, and the principal curvatures of
$\partial K$ are positive everywhere. Then $K$ is strictly convex.
We can also show that $\gamma^{\circ}$ is $C^{k,\alpha}$ on $\mathbb{R}^{n}-\{0\}$.
To see this, let $n_{K}:\partial K\to\mathbb{S}^{n-1}$ be the Gauss
map, i.e. $n_{K}(y)$ is the outward unit normal to $\partial K$
at $y$. Then $n_{K}$ is $C^{k-1,\alpha}$ and its derivative is
an isomorphism at the points with positive principal curvatures, i.e.
everywhere. Hence $n_{K}$ is locally invertible with a $C^{k-1,\alpha}$
inverse $n_{K}^{-1}$, around any point of $\mathbb{S}^{n-1}$. Now
note that as it is well known, $\gamma^{\circ}$ equals the support
function of $K$, i.e. 
\[
\gamma^{\circ}(x)=\sup\{\langle x,y\rangle:y\in K\}.
\]
Thus as shown in page 115 of \citep{MR3155183}, for $x\ne0$ we have
\[
D\gamma^{\circ}(x)=n_{K}^{-1}(\frac{x}{|x|}).
\]
Which gives the desired result. As a consequence, since $\partial K^{\circ}=\{\gamma^{\circ}=1\}$
and $D\gamma^{\circ}\ne0$ by (\ref{eq: g0 (Dg)=00003D1}), $\partial K^{\circ}$
is $C^{k,\alpha}$ too. Furthermore, as shown on page 120 of \citep{MR3155183},
the principal curvatures of $\partial K^{\circ}$ are also positive
everywhere.

Let us recall a few more properties of $\gamma,\gamma^{\circ}$. Since
they are positively 1-homogeneous, $D\gamma,D\gamma^{\circ}$ are
positively 0-homogeneous, and $D^{2}\gamma,D^{2}\gamma^{\circ}$ are
positively $(-1)$-homogeneous, i.e. 
\begin{eqnarray}
\gamma(tx)=t\gamma(x), & D\gamma(tx)=D\gamma(x), & D^{2}\gamma(tx)=\frac{1}{t}D^{2}\gamma(x),\nonumber \\
\gamma^{\circ}(tx)=t\gamma^{\circ}(x), & D\gamma^{\circ}(tx)=D\gamma^{\circ}(x), & D^{2}\gamma^{\circ}(tx)=\frac{1}{t}D^{2}\gamma^{\circ}(x),\label{eq: homog}
\end{eqnarray}
for $x\ne0$ and $t>0$. As a result, using Euler's theorem on homogeneous
functions we get 
\begin{eqnarray}
\langle D\gamma(x),x\rangle=\gamma(x), &  & D^{2}\gamma(x)\,x=0,\nonumber \\
\langle D\gamma^{\circ}(x),x\rangle=\gamma^{\circ}(x), &  & D^{2}\gamma^{\circ}(x)\,x=0,\label{eq: Euler formula}
\end{eqnarray}
for $x\ne0$. Here $D^{2}\gamma(x)\,x$ is the action of the matrix
$D^{2}\gamma(x)$ on the vector $x$. We also recall the following
fact from \citep{MR2336304}, that for $x\ne0$ 
\begin{eqnarray}
D\gamma^{\circ}(D\gamma(x))=\frac{x}{\gamma(x)}, &  & D\gamma(D\gamma^{\circ}(x))=\frac{x}{\gamma^{\circ}(x)}.\label{eq: Dg(Dg0)}
\end{eqnarray}
Finally let us mention that by Corollary 2.5.2 of \citep{MR3155183},
when $x\ne0$ the eigenvalues of $D^{2}\gamma(x)$ are $0$ with the
corresponding eigenvector $x$, and $\frac{1}{|x|}$ times the principal
radii of curvature of $\partial K^{\circ}$ at the unique point that
has $x$ as an outward normal vector. Remember that the principal
radii of curvature are the reciprocals of the principal curvatures.
Thus by our assumption, the eigenvalues of $D^{2}\gamma(x)$ are all
positive except for one $0$. We have a similar characterization of
the eigenvalues of $D^{2}\gamma^{\circ}(x)$.\textcolor{red}{}

\section{\label{sec: Reg Opstacles}Regularity of the Obstacles}

In this section, we are going to study the singularities of the functions
$\rho,\bar{\rho}$. Since $\rho,\bar{\rho}$ are viscosity solutions
of Hamilton-Jacobi equations of type (\ref{eq: H-J eq}), we can regard
this section as the study of the regularity of solutions to some Hamilton-Jacobi
equations. We know that $\rho,\bar{\rho}$ are Lipschitz functions.
We want to characterize the set over which they are more regular.
In order to do that, we need to impose some additional restrictions
on $K,U$ and $\varphi$.
\begin{assumption}
\label{assu: 2}Suppose that $k\ge2$ is an integer, and $0\le\alpha\le1$.
We assume that 
\begin{enumerate}
\item[\upshape{(a)}] $K\subset\R^{n}$ is a compact convex set whose interior contains
the origin. In addition, $\partial K$ is $C^{k,\alpha}$, and its
principal curvatures are positive at every point.
\item[\upshape{(b)}] $U\subset\R^{n}$ is a bounded open set, and $\partial U$ is $C^{k,\alpha}$.
\item[\upshape{(c)}] $\varphi:\R^{n}\to\R$ is a $C^{k,\alpha}$ function, such that $\gamma^{\circ}(D\varphi)<1$.
\end{enumerate}
Finally, in this section, we restrict the domain of $\rho,\bar{\rho}$
to $\overline{U}$. So that we can extend them to have different values
beyond $\partial U$.
\end{assumption}
\begin{rem*}
As shown in Subsection \ref{subsec: Reg gaug}, the above assumption
implies that $K,\gamma$ are strictly convex. In addition, $K^{\circ},\gamma^{\circ}$
are strictly convex, and $\partial K^{\circ},\gamma^{\circ}$ are
also $C^{k,\alpha}$. Furthermore, the principal curvatures of $\partial K^{\circ}$
are also positive at every point. Similar conclusions obviously hold
for $-K,-\varphi$ and $(-K)^{\circ}=-K^{\circ}$ too. Hence in the
sequel, whenever we prove a property for $\rho$, it holds for $\bar{\rho}$
too.
\end{rem*}
\begin{rem*}
It is obvious that under the above assumption, $\varphi$ will satisfy
the Lipschitz property (\ref{eq: phi Lip}) and its strict version
given in (\ref{eq: phi strct Lip}), as shown for example in Lemma
2.2 of \citep{MR1}. Furthermore, the above assumption implies the
bound (\ref{eq: bd D2 g}) on the weak second derivative of $\gamma$,
as we have shown in Lemma \ref{lem: C2 -> ass 3}. Therefore the above
assumption implies the Assumption \ref{assu: 3} of Appendix \ref{sec: Local-Optimal-Reg}.
Hence under the above assumption, all the results of the appendices
are true.
\end{rem*}
Suppose $\nu$ is the inward unit normal to $\partial U$. By the
above assumption, $\nu$ is a $C^{k-1,\alpha}$ function of $y\in\partial U$.
It is well known (see {[}\citealp{MR1814364}, Section 14.6{]}) that
we have $\nu=Dd$, where $d$ is the Euclidean distance to $\partial U$.
So $Dd$ is a $C^{k-1,\alpha}$ extension of $\nu$ to a neighborhood
of $\partial U$. We always assume that we are working with this extension
of $\nu$.
\begin{lem}
Suppose the Assumption \ref{assu: 2} holds. Then for every $y\in\partial U$
there is a unique scalar $\lambda(y)>0$ such that 
\begin{equation}
\gamma^{\circ}\big(D\varphi(y)+\lambda(y)\nu(y)\big)=1.\label{eq: lambda}
\end{equation}
Furthermore, $\lambda:\partial U\to(0,\infty)$ is a $C^{k-1,\alpha}$
function.
\end{lem}
\begin{proof}
Since $\gamma^{\circ}(D\varphi)<1$, the vector $D\varphi(y)$ is
in the interior of $K^{\circ}$, and therefore the ray emanating from
it in the direction of $\nu(y)$ must intersect $\partial K^{\circ}$.
Also, this ray cannot intersect $\partial K^{\circ}$ in more than
one point due to the convexity of $K^{\circ}$. So there is a unique
$\lambda(y)>0$ such that $D\varphi(y)+\lambda(y)\nu(y)\in\partial K^{\circ}$,
as desired.

By the implicit function theorem we can show that $\lambda(y)$ is
a $C^{k-1,\alpha}$ function of $y$. Consider the $C^{k-1,\alpha}$
function 
\[
(y,\lambda)\mapsto\gamma^{\circ}\big(D\varphi(y)+\lambda\nu(y)\big).
\]
We need the derivative with respect to $\lambda$ to be nonzero, i.e.
we need $\langle D\gamma^{\circ}(D\varphi+\lambda\nu),\nu\rangle\ne0$
when $\gamma^{\circ}(D\varphi+\lambda\nu)=1$. But $D\gamma^{\circ}(D\varphi+\lambda\nu)$
is normal to the surface of $\partial K^{\circ}$ at the point $D\varphi+\lambda\nu$.
On the other hand, the line $t\mapsto D\varphi+t\nu$ that passes
through $D\varphi+\lambda\nu$ cannot be tangent to $\partial K^{\circ}$,
since it intersects the interior of the convex set $K^{\circ}$ at
$D\varphi$. Hence $D\gamma^{\circ}(D\varphi+\lambda\nu)$ cannot
be orthogonal to the direction of the line, i.e to $\nu$. Thus we
get the desired.
\end{proof}
Remember that for $y\in\partial U$ we have 
\[
\mu(y)=D\varphi(y)+\lambda(y)\nu(y).
\]
By the above lemma, we know that $\mu$ is a $C^{k-1,\alpha}$ function,
and $\gamma^{\circ}(\mu)=1$. In particular, $\mu$ is always nonzero.
In addition, it is easy to compute $D\mu$. First we need to compute
$D\lambda$ using implicit function theorem. We have 
\begin{align*}
0=D(\gamma^{\circ}(D\varphi+\lambda\nu)) & =D\gamma^{\circ}(\mu)\big(D^{2}\varphi+\lambda D\nu+\nu\otimes D\lambda\big)\\
 & =D\gamma^{\circ}(\mu)D^{2}\varphi+\lambda D\gamma^{\circ}(\mu)D\nu+\langle D\gamma^{\circ}(\mu),\nu\rangle D\lambda,
\end{align*}
which gives us $D\lambda$. Thus we get 
\begin{align*}
D\mu & =D^{2}\varphi+\lambda D\nu+\nu\otimes D\lambda\\
 & =D^{2}\varphi+\lambda D\nu-\frac{1}{\langle D\gamma^{\circ}(\mu),\nu\rangle}\nu\otimes\big(D\gamma^{\circ}(\mu)D^{2}\varphi+\lambda D\gamma^{\circ}(\mu)D\nu\big).
\end{align*}
Note that $\langle D\gamma^{\circ}(\mu),\nu\rangle\ne0$ as shown
in the above proof. To simplify the above expression, recall that
we have 
\[
X=\frac{1}{\langle D\gamma^{\circ}(\mu),\nu\rangle}D\gamma^{\circ}(\mu)\otimes\nu.
\]
Note that if $w$ is orthogonal to $\nu$, i.e. if it is tangent to
$\partial U$, then $(I-X)w=w-0=w$. In addition we have $(I-X)D\gamma^{\circ}(\mu)=D\gamma^{\circ}(\mu)-D\gamma^{\circ}(\mu)=0$.
So $I-X$ is the projection on the tangent space to $\partial U$
parallel to $D\gamma^{\circ}(\mu)$. (Note that $D\gamma^{\circ}(\mu)$
is not tangent to $\partial U$ due to $\langle D\gamma^{\circ}(\mu),\nu\rangle\ne0$.)
Now it is easy to check that 
\begin{equation}
D\mu=(I-X^{T})(D^{2}\varphi+\lambda D^{2}d).\label{eq: D mu}
\end{equation}
Here we also used the fact that $\nu=Dd$. Let us recall that the
eigenvalues of $D^{2}d(y)=D\nu(y)$ are minus the principal curvatures
of $\partial U$ at $y$, and $0$. For the details see {[}\citealp{MR1814364},
Section 14.6{]}.
\begin{lem}
\label{lem: K-normal}Suppose the Assumption \ref{assu: 2} holds.
Then for every $y\in\partial U$ we have 
\begin{equation}
\big\langle D\gamma^{\circ}(\mu(y)),\nu(y)\big\rangle>0.\label{eq: Dg0 . n > 0}
\end{equation}
Furthermore, let $x\in U$, and suppose $y$ is one of the $\rho$-closest
points to $x$ on $\partial U$. Then we have 
\begin{equation}
\frac{x-y}{\gamma(x-y)}=D\gamma^{\circ}(\mu(y)).\label{eq: K-normal}
\end{equation}
Or equivalently 
\begin{equation}
x=y+\big(\rho(x)-\varphi(y)\big)\,D\gamma^{\circ}(\mu(y)).\label{eq: parametrize by rho}
\end{equation}
\end{lem}
\begin{proof}
First note that $\mu\in\partial K^{\circ}$, since $\gamma^{\circ}(\mu)=1$.
Hence $D\gamma^{\circ}(\mu)$ is the outward normal to $\partial K^{\circ}$
at $\mu$. On the other hand, $D\varphi=\mu-\lambda\nu$ belongs to
the ray passing through $\mu$ in the direction $-\nu$, because $\lambda>0$.
But we know that $D\varphi$ is in the interior of $K^{\circ}$, since
$\gamma^{\circ}(D\varphi)<1$. Thus the ray $t\mapsto\mu-t\nu$ for
$t>0$, passes through the interior of $K^{\circ}$. Therefore this
ray and $K^{\circ}$ must lie on the same side of the tangent space
to $\partial K^{\circ}$ at $\mu$, because $K^{\circ}$ is strictly
convex. Hence we must have $\langle D\gamma^{\circ}(\mu),\nu\rangle=-\langle D\gamma^{\circ}(\mu),-\nu\rangle>0$,
as desired. 

Now let $\mathrm{z}\mapsto Y(\mathrm{z})$ be a smooth parametrization
of $\partial U$ around $Y(0)=y$. Then due to $\rho$'s definition,
the function $\varphi(Y(\mathrm{z}))+\gamma(x-Y(\mathrm{z}))$ has
a minimum at $\mathrm{z}=0$. Hence 
\[
0=D_{\mathrm{z}_{j}}\big[\varphi(Y(\mathrm{z}))+\gamma(x-Y(\mathrm{z}))\big]=\sum_{i\le n}D_{\mathrm{z}_{j}}Y^{i}\big[D_{i}\varphi(y)-D_{i}\gamma(x-y)\big].
\]
Thus $D\varphi(y)-D\gamma(x-y)$ is orthogonal to every $D_{\mathrm{z}_{j}}Y$
for every $j$, and therefore it is orthogonal to $\partial U$. Hence
we have $D\varphi(y)-D\gamma(x-y)=c\nu(y)$ for some scalar $c$.
First, let us show that $c<0$. Suppose to the contrary that $c\ge0$.
Note that $\langle\nu(y),x-y\rangle\ge0$, because by Lemma \ref{lem: segment to the closest pt}
we know that $[x,y[\subset U$. Therefore by (\ref{eq: Euler formula})
and (\ref{eq: gen Cauchy-Schwartz}) we have 
\begin{align*}
\gamma(x-y) & =\langle D\gamma(x-y),x-y\rangle=\langle D\varphi(y)-c\nu(y),x-y\rangle\\
 & =\langle D\varphi(y),x-y\rangle-c\langle\nu(y),x-y\rangle\\
 & \le\langle D\varphi(y),x-y\rangle\le\gamma^{\circ}(D\varphi)\gamma(x-y)<\gamma(x-y),
\end{align*}
which is a contradiction. So $c<0$.

On the other hand we know that $\gamma^{\circ}(D\gamma)=1$. Thus
we must have 
\[
\gamma^{\circ}(D\varphi(y)-c\nu(y))=1.
\]
Therefore by the previous lemma $-c=\lambda(y)$, since $-c>0$. Hence
we obtain 
\[
D\gamma(x-y)=D\varphi(y)-c\nu(y)=D\varphi(y)+\lambda(y)\nu(y)=\mu(y).
\]
Now if we apply $D\gamma^{\circ}$ to both sides of the above equation,
then by (\ref{eq: Dg(Dg0)}) we obtain 
\[
\frac{x-y}{\gamma(x-y)}=D\gamma^{\circ}(D\gamma(x-y))=D\gamma^{\circ}(\mu(y)),
\]
as desired. The other equation follows immediately, since by (\ref{eq: rho})
we know that $\gamma(x-y)=\rho(x)-\varphi(y)$.
\end{proof}
\begin{prop}
\label{prop: D rho}Suppose the Assumption \ref{assu: 2} holds. Let
$x\in U$. Then $\rho$ is differentiable at $x$ if and only if $x\in U-R_{\rho,0}$.
And in that case we have 
\[
D\rho(x)=D\gamma(x-y)=\mu(y),
\]
where $y$ is the unique $\rho$-closest point to $x$ on $\partial U$.
In particular we have $D\rho(x)\ne0$.
\end{prop}
\begin{rem*}
This theorem is actually true if we merely assume that $\varphi$
and $\partial K$ are $C^{1}$. In addition, we only need $\partial U$
to be compact, and the boundedness of $U$ is not actually needed
in the following proof. So for example, the same result holds on the
domain $\R^{n}-\overline{U}$, with its corresponding $\rho,\mu$.
\end{rem*}
\begin{proof}
This is a trivial consequence of Theorem 3.4.4 of \citep{cannarsa2004semiconcave},
since $\rho$ is the minimum of a family of smooth functions. (Such
functions are called marginal functions.) Just note that we have to
restrict ourselves to a neighborhood of $x$, so that $D\gamma(\cdot-\tilde{y})$
exists and is continuous for every $\tilde{y}\in\partial U$. Also
note that if $x$ has more than one $\rho$-closest point on $\partial U$,
then $\rho$ is not differentiable at $x$ by Lemma \ref{lem: rho not C1}.
Finally note that $D\gamma(x-y)=\mu(y)$, as we have shown in the
proof of the previous lemma. 
\end{proof}
\begin{proof}[\textbf{Proof of Theorem \ref{thm: rho is C2 at y}}]
 We will show that $\rho$ has a $C^{k,\alpha}$ extension to an
open neighborhood of $y$. Note that if we consider $\rho$ as a function
on all of $\R^{n}$, then it is not differentiable on $\partial U$.
However, we will show that the following extension of $\rho$, which
can be considered a signed version of $\rho$ on $\R^{n}$, is $C^{k,\alpha}$
on a neighborhood of $\partial U$: 
\[
\rho_{s}(x):=\begin{cases}
\rho(x) & \textrm{if }x\in\overline{U},\\
-\bar{\rho}(x) & \textrm{if }x\in\R^{n}-\overline{U}.
\end{cases}
\]
Note that for $x\in\partial U$ we have $-\bar{\rho}(x)=-(-\varphi(x))=\varphi(x)=\rho(x)$.
So in particular, $\rho_{s}$ is a continuous function. In addition,
note that $\partial(\R^{n}-\overline{U})=\partial U$, but the inward
unit normal to $\partial(\R^{n}-\overline{U})$ is $-\nu$. Let $\bar{\mu}$
be the function constructed from $-\varphi,-\nu$, as $\mu$ is constructed
from $\varphi,\nu$. Then we have $\bar{\mu}=-D\varphi+\bar{\lambda}(-\nu)=-(D\varphi+\bar{\lambda}\nu)$,
for some uniquely determined $\bar{\lambda}>0$. But the gauge function
of $(-K)^{\circ}=-K^{\circ}$ is $\gamma^{\circ}(-\,\cdot)$, so we
must have $1=\gamma^{\circ}(-\bar{\mu})=\gamma^{\circ}(D\varphi+\bar{\lambda}\nu)$.
Hence $\bar{\lambda}=\lambda$, and therefore $\bar{\mu}=-\mu$. Thus
if we incorporate this in (\ref{eq: parametrize by rho}), we obtain
\begin{equation}
x=y+\big(\bar{\rho}(x)-(-\varphi(y))\big)\,\big(-D\gamma^{\circ}(-\bar{\mu}(y))\big)=y+\big(-\bar{\rho}(x)-\varphi(y)\big)\,D\gamma^{\circ}(\mu(y)),\label{eq: parametrize by -rho}
\end{equation}
where $x\in\R^{n}-\overline{U}$ has $y$ as its $\bar{\rho}$-closest
point on $\partial U$. Note that the derivative of the gauge function
of $(-K)^{\circ}$ is $-D\gamma^{\circ}(-\,\cdot)$.

Let $\mathrm{z}\mapsto Y(\mathrm{z})$ be a $C^{k,\alpha}$ parametrization
of $\partial U$ around $Y(0)=y$, where $\mathrm{z}$ varies in an
open set $V\subset\R^{n-1}$. Consider the map $G:V\times\R\to\R^{n}$
defined by 
\[
G(\mathrm{z},t):=Y(\mathrm{z})+\big(t-\varphi(Y(\mathrm{z}))\big)\,D\gamma^{\circ}\big(\mu(Y(\mathrm{z}))\big).
\]
Note that $G$ is a $C^{k-1,\alpha}$ function. Since $\mu$ is a
$C^{k-1,\alpha}$ function on the $C^{k,\alpha}$ manifold $\partial U$,
we can extend it to a $C^{k-1,\alpha}$ function on a neighborhood
of $\partial U$, by Lemma 6.38 of \citep{MR1814364}. Also note that
we have $G(0,\varphi(y))=y$. Now we have 
\[
\begin{cases}
D_{\mathrm{z}_{j}}G=D_{\mathrm{z}_{j}}Y-D_{\mathrm{z}_{j}}Y^{i}D_{i}\varphi D\gamma^{\circ}(\mu)+(t-\varphi)D_{\mathrm{z}_{j}}Y^{i}D^{2}\gamma^{\circ}(\mu)D_{i}\mu,\\
D_{t}G=D\gamma^{\circ}(\mu).
\end{cases}
\]
Note that $DY$ is evaluated at $\mathrm{z}$, and $\mu,\varphi,D\mu,D\varphi$
are evaluated at $Y(\mathrm{z})$. Also note that in the last term
of $D_{\mathrm{z}_{j}}G$, we evaluate the action of the matrix $D^{2}\gamma^{\circ}(\mu)$
on the vector $D_{i}\mu$. In addition, remember that we are using
the convention of summing over repeated indices. Hence we have 
\[
\begin{cases}
D_{\mathrm{z}_{j}}G(0,\varphi(y))=D_{\mathrm{z}_{j}}Y-D_{\mathrm{z}_{j}}Y^{i}D_{i}\varphi D\gamma^{\circ}(\mu),\\
D_{t}G(0,\varphi(y))=D\gamma^{\circ}(\mu),
\end{cases}
\]
where $\mu,D\varphi$ are evaluated at $y$. Let $w_{j}:=D_{\mathrm{z}_{j}}Y$.
Note that $w_{1},\dots,w_{n-1}$ is a basis for the tangent space
to $\partial U$ at $y$. Let $w$ be the orthogonal projection of
$D\gamma^{\circ}(\mu)$ on this tangent space. Then we have (we represent
a matrix by its columns) 
\begin{align*}
 & \hspace{-1cm}\det DG(0,\varphi(y))\\
 & =\det\begin{bmatrix}w_{1}-\langle w_{1},D\varphi\rangle D\gamma^{\circ}(\mu) & \cdots & w_{n-1}-\langle w_{n-1},D\varphi\rangle D\gamma^{\circ}(\mu) & D\gamma^{\circ}(\mu)\end{bmatrix}\\
 & =\det\begin{bmatrix}w_{1} & \cdots & w_{n-1} & D\gamma^{\circ}(\mu)\end{bmatrix}=\det\begin{bmatrix}w_{1} & \cdots & w_{n-1} & w+\langle D\gamma^{\circ}(\mu),\nu\rangle\nu\end{bmatrix}\\
 & =\langle D\gamma^{\circ}(\mu),\nu\rangle\det\begin{bmatrix}w_{1} & \cdots & w_{n-1} & \nu\end{bmatrix}\ne0.
\end{align*}
Note that in the last line we have used (\ref{eq: Dg0 . n > 0}),
and the fact that $w_{1},\dots,w_{n-1},\nu$ are linearly independent.

Therefore by the inverse function theorem, $G$ is invertible on an
open set of the form $W\times(\varphi(y)-h,\varphi(y)+h)$, and it
has a $C^{k-1,\alpha}$ inverse on a neighborhood of $y$. Let $B_{r}(y)$
be contained in that neighborhood, and suppose $r$ is small enough
so that for every $x\in\overline{U}\cap B_{r}(y)$, the $\rho$-closest
points on $\partial U$ to $x$ belong to $\partial U\cap Y(W)$.
This is possible due to Lemma \ref{lem: cont of y}, and the fact
that $y$ is the unique $\rho$-closest point to $y$ because of (\ref{eq: phi strct Lip}).
Similarly, suppose that $r$ is small enough so that for every $x\in(\R^{n}-\overline{U})\cap B_{r}(y)$,
the $\bar{\rho}$-closest points on $\partial U$ to $x$ belong to
$\partial U\cap Y(W)$. Also suppose that $r$ is small enough so
that for every $x\in B_{r}(y)$ we have $\rho_{s}(x)\in(\varphi(y)-h,\varphi(y)+h)$,
which is possible due to the continuity if $\rho_{s}$. Now we know
that $G:(\mathrm{z},t)\mapsto x$ has an inverse, denoted by $\mathrm{z}(x),t(x)$,
where $\mathrm{z}(\cdot),t(\cdot)$ are $C^{k-1,\alpha}$ functions
of $x$. Let $\tilde{y}:=Y(\mathrm{z}(x))$. Then we have 
\[
x=G(\mathrm{z}(x),t(x)):=\tilde{y}+\big(t(x)-\varphi(\tilde{y})\big)\,D\gamma^{\circ}\big(\mu(\tilde{y})\big).
\]
On the other hand, (\ref{eq: parametrize by rho}) and (\ref{eq: parametrize by -rho})
imply that 
\[
x=\hat{y}+\big(\rho_{s}(x)-\varphi(\hat{y})\big)\,D\gamma^{\circ}(\mu(\hat{y})),
\]
where $\hat{y}$ is one of the $\rho$-closest or $\bar{\rho}$-closest
points on $\partial U$ to $x$, depending on whether $x\in\overline{U}$
or $x\in\R^{n}-\overline{U}$. (Note that when $x=\hat{y}\in\partial U$,
the equation holds trivially.) But by our assumption about $B_{r}(y)$,
there is $\hat{\mathrm{z}}\in W$ such that $\hat{y}=Y(\hat{\mathrm{z}})$.
Hence $(\hat{\mathrm{z}},\rho_{s}(x))\in W\times(\varphi(y)-h,\varphi(y)+h)$,
and we have $G(\hat{\mathrm{z}},\rho_{s}(x))=x$. Therefore due to
the invertibility of $G$ we must have 
\[
\hat{y}=Y(\hat{\mathrm{z}})=Y(\mathrm{z}(x)),\qquad\rho_{s}(x)=t(x).
\]
Thus in particular, $\rho_{s}$ is a $C^{k-1,\alpha}$ function of
$x$. Hence $\rho$ is a $C^{k-1,\alpha}$ function on $\overline{U}\cap B_{r}(y)$.

Next, consider the line segment $t\mapsto y+(t-\varphi(y))D\gamma^{\circ}(\mu(y))$,
where $t\in(\varphi(y),\varphi(y)+\tilde{h})$. If $\tilde{h}>0$
is small enough, then this segment lies inside $U\cap B_{r}(y)$,
since we know that $\big\langle D\gamma^{\circ}(\mu(y)),\nu(y)\big\rangle>0$.
Now similarly to the last paragraph, we can show that if $x$ belongs
to this segment, then $y$ is the $\rho$-closest point on $\partial U$
to $x$. Note that $x\notin R_{\rho,0}$, since $\rho$ is differentiable
at $x$. Thus we have $D\rho(x)=\mu(y)$. Hence if we let $x$ approaches
$y$ along this segment, we get 
\[
D\rho(y)=\lim_{x\to y}D\rho(x)=\lim_{x\to y}\mu(y)=\mu(y),
\]
because $D\rho$ is continuous. 

Finally note that for every $x\in\overline{U}\cap B_{r}(y)$ we have
\[
D\rho_{s}(x)=D\rho(x)=\mu\big(Y(\mathrm{z}(x))\big).
\]
Similarly we have $D\rho_{s}(x)=-D\bar{\rho}(x)=-\bar{\mu}\big(Y(\mathrm{z}(x))\big)=\mu\big(Y(\mathrm{z}(x))\big)$,
when $x\in(\R^{n}-\overline{U})\cap B_{r}(y)$. Furthermore we know
that $\mathrm{z}(\cdot),Y,\mu$ are $C^{k-1,\alpha}$ functions. Therefore
$\rho_{s}$ is a $C^{k,\alpha}$ function. Consequently, $\rho$ is
a $C^{k,\alpha}$ function on $\overline{U}\cap B_{r}(y)$, as desired.

Now let us compute $D^{2}\rho(y)$. We know that $D\rho=\mu$ on $\partial U$.
So for every vector $w$ which is tangent to $\partial U$ we have
$D_{w}D\rho=D_{w}\mu$. Remember that $I-X$ is the projection on
the tangent space to $\partial U$ parallel to $D\gamma^{\circ}(\mu)$.
Hence by (\ref{eq: D mu}), for every vector $\tilde{w}$ we have
\begin{align*}
\tilde{w}(D^{2}\rho)w=\tilde{w}(D\mu)w & =\tilde{w}(I-X^{T})(D^{2}\varphi+\lambda D^{2}d)w\\
 & =\tilde{w}(I-X^{T})(D^{2}\varphi+\lambda D^{2}d)(I-X)w.
\end{align*}
Next let us show that $D^{2}\rho(y)D\gamma^{\circ}(\mu)=0$. The reason
is that as we have seen before, when $s\ge0$ is small, the point
$y+sD\gamma^{\circ}(\mu)\in U$ has $y$ as its unique $\rho$-closest
point on $\partial U$. Thus by (\ref{eq: D rho (x)}) we have $D\rho(y+sD\gamma^{\circ}(\mu))=\mu$.
Hence if we differentiate with respect to $s$ we get 
\[
D^{2}\rho\big(y+sD\gamma^{\circ}(\mu)\big)D\gamma^{\circ}(\mu)=0.
\]
So if we let $s\to0$ we get the desired. Therefore we have 
\begin{align*}
\tilde{w}(D^{2}\rho)D\gamma^{\circ}(\mu)=0 & =\tilde{w}(I-X^{T})(D^{2}\varphi+\lambda D^{2}d)0\\
 & =\tilde{w}(I-X^{T})(D^{2}\varphi+\lambda D^{2}d)(I-X)D\gamma^{\circ}(\mu).
\end{align*}
Thus we get the desired formula for $D^{2}\rho(y)$.
\end{proof}
\begin{proof}[\textbf{Proof of Theorem \ref{thm: rho is C^2}}]
 Let $\mathrm{z}\mapsto Y(\mathrm{z})$ be a $C^{k,\alpha}$ parametrization
of $\partial U$ around $Y(0)=y$, where $\mathrm{z}$ varies in an
open set $V\subset\R^{n-1}$. Consider the map $G:V\times\R\to\R^{n}$
defined by 
\[
G(\mathrm{z},t):=Y(\mathrm{z})+\big(t-\varphi(Y(\mathrm{z}))\big)\,D\gamma^{\circ}\big(D\rho(Y(\mathrm{z}))\big).
\]
Note that $G$ is a $C^{k-1,\alpha}$ function. Also note that by
(\ref{eq: parametrize by rho}) and (\ref{eq: D rho (y)}) we have
\[
G(0,\rho(x))=y+\big(\rho(x)-\varphi(y)\big)\,D\gamma^{\circ}(D\rho(y))=y+\big(\rho(x)-\varphi(y)\big)\,D\gamma^{\circ}(\mu(y))=x.
\]
We wish to compute $DG$ around the point $(0,\rho(x))$. We have
\[
\begin{cases}
D_{\mathrm{z}_{j}}G=D_{\mathrm{z}_{j}}Y-\langle D\varphi,D_{\mathrm{z}_{j}}Y\rangle D\gamma^{\circ}(\mu)+(\rho(x)-\varphi)D^{2}\gamma^{\circ}(\mu)D^{2}\rho(y)D_{\mathrm{z}_{j}}Y,\\
D_{t}G=D\gamma^{\circ}(\mu).
\end{cases}
\]
Note that $DY$ is evaluated at $\mathrm{z}=0$; $\mu,\varphi,D\varphi$
are evaluated at $y=Y(0)$; and we used the fact that $D\rho(y)=\mu(y)$.

Next note that we have $QD\gamma^{\circ}(\mu)=D\gamma^{\circ}(\mu)$,
since by (\ref{eq: D2 rho Dg0 =00003D 0}) we have $D^{2}\rho(y)D\gamma^{\circ}(\mu)=0$.
Now let $w_{j}:=D_{\mathrm{z}_{j}}Y(0)$. Note that $w_{1},\dots,w_{n-1}$
is a basis for the tangent space to $\partial U$ at $y$. Let $w$
be the orthogonal projection of $D\gamma^{\circ}(\mu)$ on this tangent
space. Then we have (we represent a matrix by its columns) 
\begin{align*}
 & \hspace{-1cm}\det DG(0,\rho(x))\\
 & =\det\begin{bmatrix}Qw_{1}-\langle D\varphi,w_{1}\rangle D\gamma^{\circ}(\mu) & \cdots & Qw_{n-1}-\langle D\varphi,w_{n-1}\rangle D\gamma^{\circ}(\mu) & D\gamma^{\circ}(\mu)\end{bmatrix}\\
 & =\det\begin{bmatrix}Qw_{1} & \cdots & Qw_{n-1} & D\gamma^{\circ}(\mu)\end{bmatrix}=\det\begin{bmatrix}Qw_{1} & \cdots & Qw_{n-1} & QD\gamma^{\circ}(\mu)\end{bmatrix}\\
 & =\det Q\,\det\begin{bmatrix}w_{1} & \cdots & w_{n-1} & D\gamma^{\circ}(\mu)\end{bmatrix}\\
 & =\det Q\,\det\begin{bmatrix}w_{1} & \cdots & w_{n-1} & w+\langle D\gamma^{\circ}(\mu),\nu\rangle\nu\end{bmatrix}\\
 & =\langle D\gamma^{\circ}(\mu),\nu\rangle\det Q\,\det\begin{bmatrix}w_{1} & \cdots & w_{n-1} & \nu\end{bmatrix}\ne0.
\end{align*}
Note that in the last line we have used (\ref{eq: Dg0 . n > 0}),
and the fact that $w_{1},\dots,w_{n-1},\nu$ are linearly independent.
Therefore by the inverse function theorem, $G$ is invertible on a
neighborhood of $(0,\rho(x))$, and it has a $C^{k-1,\alpha}$ inverse
on a neighborhood of $x$. Then as in the proof of the previous theorem,
we can show that the inverse of $G$ is of the form $G^{-1}(\cdot)=(\mathrm{z}(\cdot),\rho(\cdot))$,
where $\mathrm{z}(\cdot)$ is a $C^{k-1,\alpha}$ function of $x$.
Note that here we need the fact that $x\notin R_{\rho,0}$. In addition,
we can similarly conclude that $\rho$ is $C^{k,\alpha}$ on a neighborhood
of $x$. 

Now let us compute $DG^{-1}$ at $x$. To simplify the notation set
$s_{j}:=\langle D\varphi,w_{j}\rangle$. Then we have 
\begin{align*}
DG(0,\rho(x)) & =\begin{bmatrix}Qw_{1}-s_{1}D\gamma^{\circ}(\mu) & \cdots & Qw_{n-1}-s_{n-1}D\gamma^{\circ}(\mu) & D\gamma^{\circ}(\mu)\end{bmatrix}\\
 & =\begin{bmatrix}Qw_{1}-s_{1}QD\gamma^{\circ}(\mu) & \cdots & Qw_{n-1}-s_{n-1}QD\gamma^{\circ}(\mu) & QD\gamma^{\circ}(\mu)\end{bmatrix}\\
 & =Q\begin{bmatrix}w_{1}-s_{1}D\gamma^{\circ}(\mu) & \cdots & w_{n-1}-s_{n-1}D\gamma^{\circ}(\mu) & D\gamma^{\circ}(\mu)\end{bmatrix}\\
 & =QA\begin{bmatrix}e_{1}-s_{1}e_{n} & \cdots & e_{n-1}-s_{n-1}e_{n} & e_{n}\end{bmatrix}\negthickspace,
\end{align*}
where $e_{j}$ is the $j$-th column vector in the standard basis
of $\R^{n}$, and 
\[
A:=\begin{bmatrix}w_{1} & \cdots & w_{n-1} & D\gamma^{\circ}(\mu)\end{bmatrix}\negthickspace.
\]
Note that the $j$-th column of every matrix is equal to the action
of the matrix on $e_{j}$. Also note that $w_{1},\dots,w_{n-1},D\gamma^{\circ}(\mu)$
are linearly independent due to (\ref{eq: Dg0 . n > 0}). Thus $A$
is invertible. Therefore we have 
\begin{align*}
DG^{-1}(x) & =\big(DG(0,\rho(x))\big)^{-1}\\
 & =\big(QA\begin{bmatrix}e_{1}-s_{1}e_{n} & \cdots & e_{n-1}-s_{n-1}e_{n} & e_{n}\end{bmatrix}\big)^{-1}\\
 & =\begin{bmatrix}e_{1}-s_{1}e_{n} & \cdots & e_{n-1}-s_{n-1}e_{n} & e_{n}\end{bmatrix}^{-1}A^{-1}Q^{-1}\\
 & =\begin{bmatrix}e_{1}+s_{1}e_{n} & \cdots & e_{n-1}+s_{n-1}e_{n} & e_{n}\end{bmatrix}A^{-1}Q^{-1}.
\end{align*}
Note that when $i<n-1$, the $i$-th component of $G^{-1}$ is $\mathrm{z}_{i}$.
Hence the $i$-th row of $DG^{-1}$ is $D\mathrm{z}_{i}$. On the
other hand, the $i$-th row of $DG^{-1}$ is equal to the $i$-th
row of 
\[
\begin{bmatrix}e_{1}+s_{1}e_{n} & \cdots & e_{n-1}+s_{n-1}e_{n} & e_{n}\end{bmatrix}
\]
times $A^{-1}Q^{-1}$; which equals to $e_{i}^{T}A^{-1}Q^{-1}$, where
$e_{i}^{T}$ is the transpose of $e_{i}$. So we have $D\mathrm{z}=\tilde{I}A^{-1}Q^{-1}$,
where $\tilde{I}$ is the $(n-1)\times n$ matrix whose $i$-th row
is $e_{i}^{T}$. 

Now note that $D\rho(x)=\mu(y)=D\rho(y)=D\rho(Y(0))=D\rho(Y(\mathrm{z}(x)))$.
Thus we have 
\[
D^{2}\rho(x)=D^{2}\rho(y)DY(0)D\mathrm{z}(x)=D^{2}\rho(y)DY(0)\tilde{I}A^{-1}Q^{-1}.
\]
On the other hand we know that $DY(0)=\begin{bmatrix}w_{1} & \cdots & w_{n-1}\end{bmatrix}$,
i.e. the $j$-th column of $DY(0)$ is the $j$-th column of $A$,
for $j<n$. Then it is easy to check that $DY(0)\tilde{I}=A\hat{I}$,
where $\hat{I}$ is the $n\times n$ matrix whose first $n-1$ columns
are the same as $I$, and its $n$-th column is $0$. Next note that
the $n$-th row of $A^{-1}$ is $\frac{1}{\langle D\gamma^{\circ}(\mu),\nu\rangle}\nu$.
Because it must be orthogonal to the first $n-1$ columns of $A$,
i.e. $w_{j}$'s; so it is a multiple of $\nu$. In addition, its inner
product with the $n$-th column of $A$, i.e. $D\gamma^{\circ}(\mu)$,
must be $1$; thus we get the desired. Hence we have 
\begin{align*}
A\hat{I}A^{-1} & =A\big(I-\begin{bmatrix}0 & \cdots & 0 & e_{n}\end{bmatrix}\big)A^{-1}\\
 & =I-A\begin{bmatrix}0 & \cdots & 0 & e_{n}\end{bmatrix}A^{-1}=I-\begin{bmatrix}0 & \cdots & 0 & Ae_{n}\end{bmatrix}A^{-1}\\
 & =I-\begin{bmatrix}0 & \cdots & 0 & D\gamma^{\circ}(\mu)\end{bmatrix}A^{-1}=I-\frac{1}{\langle D\gamma^{\circ}(\mu),\nu\rangle}D\gamma^{\circ}(\mu)\otimes\nu=I-X.
\end{align*}
Therefore we get 
\[
D^{2}\rho(x)=D^{2}\rho(y)A\hat{I}A^{-1}Q^{-1}=D^{2}\rho(y)\big(I-X\big)Q^{-1}.
\]
However, by (\ref{eq: D2 rho (y)}) we know that $D^{2}\rho(y)$ has
a factor of $I-X$. Also, as explained before Lemma \ref{lem: K-normal},
we know that $I-X$ is a projection. Thus $(I-X)^{2}=I-X$, and we
get the desired formula for $D^{2}\rho(x)$.
\end{proof}

At this point we have the tools to completely characterize the $\rho$-ridge
$R_{\rho}$. Note that under the Assumption \ref{assu: 2} we have
$R_{\rho,0}\subset R_{\rho}$, due to Lemma \ref{lem: rho not C1}.
Theorem \ref{thm: ridge} specifies those points which are in $R_{\rho}-R_{\rho,0}$.
\begin{proof}[\textbf{Proof of Theorem \ref{thm: ridge}}]
 In Theorem \ref{thm: rho is C^2} we have shown that if $\det Q\ne0$
then $x\notin R_{\rho}$. So we only need to show that if $\det Q=0$
then $x\in R_{\rho}$. Suppose $\det Q(x)=0$. Then the definition
of $Q$ implies that $\tilde{\kappa}:=\frac{1}{\rho(x)-\varphi(y)}=\frac{1}{\gamma(x-y)}>0$
is an eigenvalue of $W(y)$. Suppose $\zeta$ is the corresponding
eigenvector of $W$. Let $z\in]x,y[$. Then by Lemma \ref{lem: segment to the closest pt}
we have $z\in U$, and $y$ is the unique $\rho$-closest point on
$\partial U$ to $z$. In addition we have $z-y=t(x-y)$ for some
$t\in(0,1)$. Thus $\gamma(z-y)=t\gamma(x-y)$. Therefore if $z$
is close enough to $x$, then we must have $\det Q(z)\ne0$. Because
otherwise $\frac{1}{\gamma(z-y)}=\frac{\tilde{\kappa}}{t}$ must be
an eigenvalue of $W$ for infinitely many $t$'s, which is a contradiction.
Hence by Theorem \ref{thm: rho is C^2}, $\rho$ is $C^{k,\alpha}$
on a neighborhood of $z$. 

Now suppose to the contrary that $x\notin R_{\rho}$. Then $\rho$
is $C^{1,1}$ on a neighborhood of $x$. Thus $\rho$ belongs to $W^{2,\infty}$
on a neighborhood of $x$. Consequently, $D^{2}\rho$ belongs to $L^{\infty}$
on a neighborhood of $x$. Therefore $D^{2}\rho$ is bounded on a
neighborhood of an open line segment $]x,z_{0}[$ for some $z_{0}\in]x,y[$,
because $\rho$ is $C^{k,\alpha}$ there. However, for $z\in]x,z_{0}[$
we have 
\[
Q(z)=I-\gamma(z-y)W=\frac{1}{s\tilde{\kappa}}\big(s\tilde{\kappa}I-W\big),
\]
where $s>1$ is such that $s(z-y)=x-y$. Hence we have $Q(z)\zeta=\frac{1}{s\tilde{\kappa}}(s-1)\tilde{\kappa}\zeta=\frac{s-1}{s}\zeta$.
Let $\xi:=\zeta D^{2}\gamma^{\circ}(\mu)$. Then by (\ref{eq: D2 rho (x)})
we have 
\begin{align*}
D_{\xi\zeta}^{2}\rho(z) & =\xi D^{2}\rho(z)\zeta=\xi D^{2}\rho(y)Q(z)^{-1}\zeta\\
 & =\xi D^{2}\rho(y)\frac{s}{s-1}\zeta=\frac{s}{s-1}\zeta D^{2}\gamma^{\circ}(\mu)D^{2}\rho(y)\zeta=\frac{-s\tilde{\kappa}}{s-1}|\zeta|^{2}.
\end{align*}
Therefore as $z\to x$ we have $D_{\xi\zeta}^{2}\rho(z)\to-\infty$,
since $s\to1^{+}$. This contradiction shows that $\rho$ cannot be
$C^{1,1}$ on a neighborhood of $x$, and therefore $x\in R_{\rho}$.
\end{proof}
The proofs of the following theorem and the proposition after it are
variants of the proofs of similar results in \citep{MR2336304}.
\begin{thm}
\label{thm: det Q > 0}Suppose the Assumption \ref{assu: 2} holds.
Suppose $x\in U$, and $y$ is one of the $\rho$-closest points to
$x$ on $\partial U$. Suppose $\tilde{\kappa}$ is an eigenvalue
of $W(y)$, where $W$ is defined by (\ref{eq: W,Q}). Then we have
\begin{equation}
\tilde{\kappa}\big(\rho(z)-\varphi(y)\big)=\tilde{\kappa}\gamma(z-y)<1,\label{eq: kd < 1}
\end{equation}
for every $z\in]x,y[$. As a consequence we have 
\[
\det Q(z)=\det\big(I-\gamma(z-y)W(y)\big)>0.
\]
\end{thm}
\begin{rem*}
As a result, due to the continuity we have 
\[
\tilde{\kappa}\big(\rho(x)-\varphi(y)\big)=\tilde{\kappa}\gamma(x-y)\le1,
\]
and $\det\big(I-\gamma(x-y)W(y)\big)\ge0$. Note that $x$ can have
$\rho$-closest points on $\partial U$ other than $y$. 
\end{rem*}
\begin{proof}
Note that by Lemma \ref{lem: segment to the closest pt}, $y$ is
the unique $\rho$-closest point to $z$ on $\partial U$. Also note
that by (\ref{eq: rho}) we have $\gamma(z-y)=\rho(z)-\varphi(y)$.
So $Q(z)=I-\gamma(z-y)W$. Now let $\mathrm{z}\mapsto Y(\mathrm{z})$
be a smooth parametrization of $\partial U$ around $Y(0)=y$, where
$\mathrm{z}$ varies in an open set $V\subset\R^{n-1}$. Then due
to $\rho$'s definition, the function $\varphi(Y(\mathrm{z}))+\gamma(x-Y(\mathrm{z}))$
has a minimum at $\mathrm{z}=0$. Hence its second derivative must
be positive semidefinite at $\mathrm{z}=0$. We have 
\[
D_{\mathrm{z}_{j}}\big[\varphi(Y(\mathrm{z}))+\gamma(x-Y(\mathrm{z}))\big]=\langle D_{\mathrm{z}_{j}}Y,D\varphi(Y(\mathrm{z}))-D\gamma(x-Y(\mathrm{z}))\rangle.
\]
Therefore at $\mathrm{z}=0$ we have 
\begin{align}
D_{\mathrm{z}_{j}\mathrm{z}_{j}}^{2}\big[\varphi(Y(\mathrm{z}))+\gamma(x-Y(\mathrm{z}))\big] & =\langle D_{\mathrm{z}_{j}\mathrm{z}_{j}}^{2}Y,D\varphi(y)-D\gamma(x-y)\rangle\nonumber \\
 & \qquad\qquad+D_{\mathrm{z}_{j}}Y\big(D^{2}\varphi(y)+D^{2}\gamma(x-y)\big)D_{\mathrm{z}_{j}}Y\ge0.\label{eq: 1 in det Q > 0}
\end{align}
However by (\ref{eq: K-normal}) we know that $x-y=\gamma(x-y)D\gamma^{\circ}(\mu(y))$.
Hence by (\ref{eq: homog}),(\ref{eq: Dg(Dg0)}) we have 
\[
D\gamma(x-y)=D\gamma(D\gamma^{\circ}(\mu))=\mu,
\]
since $\gamma^{\circ}(\mu)=1$. Thus by (\ref{eq: mu}) we have 
\[
D\varphi(y)-D\gamma(x-y)=D\varphi(y)-\mu(y)=-\lambda(y)\nu(y).
\]
On the other hand we have $\langle D_{\mathrm{z}_{j}}Y,\nu(Y(\mathrm{z}))\rangle=0$.
So if we differentiate this equality we get 
\[
\langle D_{\mathrm{z}_{j}\mathrm{z}_{j}}^{2}Y,\nu\rangle+D_{\mathrm{z}_{j}}Y(D\nu)D_{\mathrm{z}_{j}}Y=0.
\]
However we know that $D\nu=D^{2}d$, where $d$ is the Euclidean distance
to $\partial U$. Therefore we obtain 
\[
\langle D_{\mathrm{z}_{j}\mathrm{z}_{j}}^{2}Y,D\varphi(y)-D\gamma(x-y)\rangle=\langle D_{\mathrm{z}_{j}\mathrm{z}_{j}}^{2}Y,-\lambda\nu\rangle=D_{\mathrm{z}_{j}}Y(\lambda D^{2}d)D_{\mathrm{z}_{j}}Y.
\]
Thus by inserting this equality in (\ref{eq: 1 in det Q > 0}) we
get 
\[
D_{\mathrm{z}_{j}}Y\big(\lambda D^{2}d(y)+D^{2}\varphi(y)+D^{2}\gamma(x-y)\big)D_{\mathrm{z}_{j}}Y\ge0.
\]
To simplify the notation set $w_{j}:=D_{\mathrm{z}_{j}}Y$. Remember
that we have $(I-X)w_{j}=w_{j}$, since $I-X$ is the projection on
the tangent space to $\partial U$ parallel to $D\gamma^{\circ}(\mu)$,
as explained after equation (\ref{eq: X}). Hence we have 
\[
w_{j}(I-X^{T})\big(\lambda D^{2}d(y)+D^{2}\varphi(y)\big)(I-X)w_{j}+w_{j}D^{2}\gamma(x-y)w_{j}\ge0.
\]
Then by (\ref{eq: D2 rho (y)}) we get 
\[
w_{j}\big(D^{2}\rho(y)+D^{2}\gamma(x-y)\big)w_{j}\ge0.
\]
In addition, by (\ref{eq: homog}) we have 
\[
D^{2}\gamma(x-y)=D^{2}\gamma\big(\gamma(x-y)D\gamma^{\circ}(\mu)\big)=\tfrac{1}{\gamma(x-y)}D^{2}\gamma(D\gamma^{\circ}(\mu)).
\]
Thus the symmetric matrix $D^{2}\rho(y)+\frac{1}{\gamma(x-y)}D^{2}\gamma(D\gamma^{\circ}(\mu))$
is positive semidefinite, since its action on $D\gamma^{\circ}(\mu)$
is zero due to (\ref{eq: D2 rho Dg0 =00003D 0}),(\ref{eq: Euler formula});
and $w_{1},\dots,w_{n-1},D\gamma^{\circ}(\mu)$ form a basis for $\R^{n}$
due to (\ref{eq: Dg0 . n > 0}).

Now note that the eigenvalues of $D^{2}\gamma(D\gamma^{\circ}(\mu))$
are all positive except for one $0$ which corresponds to the eigenvector
$D\gamma^{\circ}(\mu)$, as explained in Subsection \ref{subsec: Reg gaug}.
We also know that $\langle D\gamma^{\circ}(\mu),\mu\rangle=\gamma^{\circ}(\mu)=1\ne0$.
Let $\{\mu\}^{\perp}$ be the subspace orthogonal to $\mu$. Then
for every nonzero $w\in\{\mu\}^{\perp}$ we must have $wD^{2}\gamma(D\gamma^{\circ}(\mu))w>0$.
Hence for every $t>\frac{1}{\gamma(x-y)}$ we have 
\begin{align}
w\big(D^{2}\rho(y)+tD^{2}\gamma(D\gamma^{\circ}(\mu))\big)w & =w\big(D^{2}\rho(y)+\tfrac{1}{\gamma(x-y)}D^{2}\gamma(D\gamma^{\circ}(\mu))\big)w\nonumber \\
 & \qquad\qquad+\big(t-\tfrac{1}{\gamma(x-y)}\big)wD^{2}\gamma(D\gamma^{\circ}(\mu))w>0.\label{eq: 2 in det Q > 0}
\end{align}
On the other hand, if we differentiate (\ref{eq: Dg(Dg0)}) we get
\[
D^{2}\gamma(D\gamma^{\circ}(\cdot))D^{2}\gamma^{\circ}(\cdot)=\frac{1}{\gamma^{\circ}(\cdot)}I-\frac{1}{\gamma^{\circ}(\cdot)^{2}}(\cdot)\otimes D\gamma^{\circ}(\cdot).
\]
Hence we have $D^{2}\gamma(D\gamma^{\circ}(\mu))D^{2}\gamma^{\circ}(\mu)=I-\mu\otimes D\gamma^{\circ}(\mu)$.
By taking the transpose of this equation we get 
\begin{equation}
D^{2}\gamma^{\circ}(\mu)D^{2}\gamma(D\gamma^{\circ}(\mu))=I-D\gamma^{\circ}(\mu)\otimes\mu.\label{eq: D2g0 . D2g(Dg0)}
\end{equation}
Next suppose that $\tilde{\kappa}$ is an eigenvalue of $W$ corresponding
to the eigenvector $w$. If $\tilde{\kappa}=0$ then (\ref{eq: kd < 1})
holds trivially. So suppose $\tilde{\kappa}\ne0$. Then we have 
\begin{align*}
\langle w,\mu\rangle & =\frac{1}{\tilde{\kappa}}\langle Ww,\mu\rangle=-\frac{1}{\tilde{\kappa}}\langle D^{2}\gamma^{\circ}(\mu)D^{2}\rho(y)w,\mu\rangle\\
 & =-\frac{1}{\tilde{\kappa}}\langle D^{2}\rho(y)w,D^{2}\gamma^{\circ}(\mu)\mu\rangle=-\frac{1}{\tilde{\kappa}}\langle D^{2}\rho(y)w,0\rangle=0.
\end{align*}
Hence $w\in\{\mu\}^{\perp}$. On the other hand, note that the eigenvalues
of $D^{2}\gamma^{\circ}(\mu)$ are all positive except for one $0$
which corresponds to the eigenvector $\mu$. In addition we know that
the other eigenvectors of $D^{2}\gamma^{\circ}(\mu)$ are orthogonal
to $\mu$, since it is a symmetric matrix. Thus the image of $D^{2}\gamma^{\circ}(\mu)$
is $\{\mu\}^{\perp}$, and its restriction to $\{\mu\}^{\perp}$ is
positive definite and invertible. Let $A$ be the symmetric matrix
whose action on $\{\mu\}^{\perp}$ is the inverse of the action of
$D^{2}\gamma^{\circ}(\mu)$ restricted to $\{\mu\}^{\perp}$; and
$A\mu=0$. Then the restriction of $A$ to $\{\mu\}^{\perp}$ is also
positive definite. It is easy to check that $AD^{2}\gamma^{\circ}(\mu)=I-\frac{1}{|\mu|^{2}}\mu\otimes\mu$.

Now for $z\in]x,y[$ let $t:=\frac{1}{\gamma(z-y)}$. Then $t>\frac{1}{\gamma(x-y)}$.
Therefore by (\ref{eq: 2 in det Q > 0}) and (\ref{eq: D2g0 . D2g(Dg0)})
we have 
\begin{align*}
0 & <w\big(D^{2}\rho(y)+tD^{2}\gamma(D\gamma^{\circ}(\mu))\big)w=w(I-\tfrac{1}{|\mu|^{2}}\mu\otimes\mu)\big(D^{2}\rho(y)+tD^{2}\gamma(D\gamma^{\circ}(\mu))\big)w\\
 & =wAD^{2}\gamma^{\circ}(\mu)\big(D^{2}\rho(y)+tD^{2}\gamma(D\gamma^{\circ}(\mu))\big)w=wA\big(-W+t(I-D\gamma^{\circ}(\mu)\otimes\mu)\big)w\\
 & =wA(-\tilde{\kappa}+t)w=\tfrac{1}{\gamma(z-y)}(-\tilde{\kappa}\gamma(z-y)+1)wAw.
\end{align*}
Note that since $w$ is orthogonal to $\mu$, we have $w(\mu\otimes\mu)=(D\gamma^{\circ}(\mu)\otimes\mu)w=0$.
Hence we get the desired relation (\ref{eq: kd < 1}), because $wAw>0$
due to the positive definiteness of $A$ restricted to $\{\mu\}^{\perp}$. 

Finally note that every eigenvalue of $W$ must be real. Because $W$
is the product of two symmetric matrices $D^{2}\gamma^{\circ}(\mu),-D^{2}\rho(y)$;
and $D^{2}\gamma^{\circ}(\mu)$ is positive semidefinite.  Thus in
particular, $W$ is similar to a real triangular matrix. Hence the
determinant of $Q(z)=I-\gamma(z-y)W$ is the product of $1-\tilde{\kappa}\gamma(z-y)$,
where $\tilde{\kappa}$ varies among the eigenvalues of $W$. Therefore
$\det Q>0$ as desired.
\end{proof}

\begin{rem*}
Suppose the Assumption \ref{assu: 2} holds. Suppose $x\in U$, and
$y$ is one of the $\rho$-closest points to $x$ on $\partial U$.
It can be shown that the segment $]x,y[$ is the characteristic curve
associated to the first order PDE (\ref{eq: H-J eq}). We do not use
this fact in this article, but let us summarize what we have proved
so far about the segment $]x,y[$.
\begin{enumerate}
\item In Lemma \ref{lem: segment to the closest pt} we have shown that
$]x,y[\subset U$, and $y$ is the unique $\rho$-closest point to
every point of $]x,y[$. In particular we have $]x,y[\cap R_{\rho,0}=\emptyset$.
\item We have also seen that $\rho$ varies linearly along $]x,y[$, and
by (\ref{eq: D rho (x)}) we have $D\rho(z)=\mu(y)$ for every $z\in]x,y[$.
\item By (\ref{eq: K-normal}) we know that $]x,y[$ is parallel to $D\gamma^{\circ}(\mu(y))$.
We also have 
\[
D^{2}\rho(z)D\gamma^{\circ}(\mu(y))=0,
\]
for every $z\in]x,y[$. Because we have $W(y)D\gamma^{\circ}(\mu)=0$
by (\ref{eq: W,Q}),(\ref{eq: D2 rho Dg0 =00003D 0}). So we have
$Q(z)D\gamma^{\circ}(\mu)=D\gamma^{\circ}(\mu)$ due to (\ref{eq: W,Q}).
Hence we get the desired by (\ref{eq: D2 rho (x)}) and (\ref{eq: D2 rho Dg0 =00003D 0}).
Furthermore due to (\ref{eq: D2 rho (x)}), $D^{2}\gamma^{\circ}(\mu)D^{2}\rho(z)$
can be triangulated, and its triangular form is 
\[
\begin{bmatrix}\frac{-\tilde{\kappa}_{1}}{1-\gamma(z-y)\tilde{\kappa}_{1}} &  & \qquad\qquad\;* & *\\
 & \ddots &  & \begin{array}{c}
*\\
\\
\end{array}\\
0\qquad\;\: &  & \frac{-\tilde{\kappa}_{n-1}}{1-\gamma(z-y)\tilde{\kappa}_{n-1}}\\
0\qquad0 &  &  & 0
\end{bmatrix}\negthickspace,
\]
where $\tilde{\kappa}_{1},\dots,\tilde{\kappa}_{n-1},0$ are the eigenvalues
of $W(y)$. Remember that the eigenvalues of $W$ are all real, and
hence it can be triangulated, as we explained in the previous proof.
Also remember that $1-\gamma(z-y)\tilde{\kappa}_{j}>0$ by the previous
theorem. Notice the similarity between the above form and the classical
formula for $D^{2}d$ derived in {[}\citealp{MR1814364}, Section
14.6{]}. Although, we cannot necessarily find a similar form for $D^{2}\rho(z)$
itself.%
\item By the above theorem we know that $\det Q(z)\ne0$ for every $z\in]x,y[$.
Thus in particular we have $]x,y[\cap R_{\rho}=\emptyset$ due to
Theorem \ref{thm: ridge}. Hence by Theorem \ref{thm: rho is C^2},
$\rho$ is $C^{k,\alpha}$ on a neighborhood of $]x,y[$.
\end{enumerate}
Finally let us mention that for the function $G$ defined in the proof
of Theorem \ref{thm: rho is C^2} we have $\det G>0$. Because as
we have shown there we have 
\[
\det G=\langle D\gamma^{\circ}(\mu),\nu\rangle\det Q\,\det\begin{bmatrix}w_{1} & \cdots & w_{n-1} & \nu\end{bmatrix}\negthickspace,
\]
and we know that the first two factors are positive. The third factor
is also positive, since it can be shown that $\det\begin{bmatrix}w_{1} & \cdots & w_{n-1} & \nu\end{bmatrix}=\sqrt{\det(g_{ij})}$,
where $g_{ij}:=\langle w_{i},w_{j}\rangle$ are the components of
the Riemannian metric on $\partial U$. For the proof see for example
Lemma 4.10 of \citep{MR2336304}.
\end{rem*}
\begin{prop}
\label{prop: R =00003D R0 bar}Suppose the Assumption \ref{assu: 2}
holds. Then we have $R_{\rho}=\overline{R}_{\rho,0}$.
\end{prop}
\begin{proof}
Note that $R_{\rho}\subset U$ is closed in $U$ by definition, and
it has a positive distance from $\partial U$ due to Theorem \ref{thm: rho is C2 at y}.
Hence $R_{\rho}$ is closed. We also know that $R_{\rho,0}\subset R_{\rho}$.
So we have $\overline{R}_{\rho,0}\subset R_{\rho}$. Thus we only
need to show that $R_{\rho}\subset\overline{R}_{\rho,0}$. Let $x\in U-\overline{R}_{\rho,0}$.
We will show that $x\notin R_{\rho}$. There is $r>0$ such that $B_{r}(x)\subset U-\overline{R}_{\rho,0}$.
For every $z\in B_{r}(x)$ let $p(z)\in\partial U$ be the unique
$\rho$-closest point to $z$ on $\partial U$. Then Lemma \ref{lem: cont of y}
implies that $p$ is continuous on $B_{r}(x)$. 

Let $\mathrm{z}\mapsto Y(\mathrm{z})$ be a smooth parametrization
of $\partial U$ around $Y(0)=p(x)$, where $\mathrm{z}$ varies in
an open set $V\subset\R^{n-1}$. Then $Y^{-1}:Y(V)\to V$ is a continuous
bijection. There is $0<s<r$ such that $p(B_{s}(x))\subset Y(V)$,
since $p$ is continuous. Now $Y^{-1}\circ p:\partial B_{s}(x)\to V\subset\R^{n-1}$
is a continuous map. Hence by Borsuk\textendash Ulam theorem (see
Corollary 2B.7 of \citep{hatcher2002algebraic}) there are $x^{-}:=x-sw$
and $x^{+}:=x+sw$, for some $w\in\partial B_{1}(0)$, such that 
\[
Y^{-1}\circ p(x^{-})=Y^{-1}\circ p(x^{+}).
\]
But $Y^{-1}$ is one to one, so $p(x^{-})=p(x^{+})$. Let $y:=p(x^{\pm})$.
Then by (\ref{eq: parametrize by rho}) we have 
\[
x^{\pm}=y+\big(\rho(x^{\pm})-\varphi(y)\big)\,D\gamma^{\circ}(\mu(y)).
\]
Thus $x^{-},x^{+}$ are on the ray emanating from $y$ in the direction
of $D\gamma^{\circ}(\mu)$. Suppose for example $x^{-}\in]y,x^{+}[$.
On the other hand, it is obvious that $x\in]x^{-},x^{+}[$. So we
get $x\in]y,x^{+}[$. Hence by Theorem \ref{thm: det Q > 0} we have
$\det Q(x)\ne0$, and therefore $x\notin R_{\rho}$ due to Theorem
\ref{thm: rho is C^2}, as desired. Note that by Lemma \ref{lem: segment to the closest pt}
it also follows that $p(x)=y$. This fact is not needed in this proof,
but we will use it in the next remark.
\end{proof}
\begin{rem*}
Suppose the Assumption \ref{assu: 2} holds. Let $y\in\partial U$.
Then by Theorem \ref{thm: rho is C2 at y} we know that $y$ is the
$\rho$-closest point to some points in $U$. By (\ref{eq: parametrize by rho}),
these points must lie on the ray emanating from $y$ in the direction
of $D\gamma^{\circ}(\mu(y))$. By Lemma \ref{lem: segment to the closest pt}
we know that if $x$ on this ray has $y$ as a $\rho$-closest point
on $\partial U$, then every point between $x,y$ belongs to $U$,
and has $y$ as its unique $\rho$-closest point on $\partial U$.
On the other hand, this ray will intersect $\partial U$ at a point
other than $y$, because $U$ is bounded. Let $\tilde{y}$ be the
closest of such intersection points to $y$. Then by Lemma \ref{lem: cont of y},
the points on the ray near $\tilde{y}$ (which are inside $U$) cannot
have $y$ as their $\rho$-closest point on $\partial U$, since their
$\rho$-closest points must converge to $\tilde{y}$ as they approach
$\tilde{y}$.

Let $x$ be the supremum of the points on the ray that are inside
$U$, and have $y$ as a $\rho$-closest point on $\partial U$. Note
that by Lemma \ref{lem: cont of y}, $y$ is one of the $\rho$-closest
points to $x$ on $\partial U$, since we can approach $x$ with points
having $y$ as their $\rho$-closest point on $\partial U$. On the
other hand, $x$ must belong to $R_{\rho}=\overline{R}_{\rho,0}$.
Because as shown in the above proof, if $x\notin\overline{R}_{\rho,0}$
then there is another point $x^{+}$ on the ray that has $y$ as a
$\rho$-closest point on $\partial U$, and $x\in]y,x^{+}[$; which
contradicts our choice of $x$. Hence to summarize, along the ray
emanating from $y$ in the direction of $D\gamma^{\circ}(\mu(y))$,
the only points that have $y$ as their  $\rho$-closest point on
$\partial U$ are those points which lie between $y$ and the closest
point to $y$ on the intersection of $R_{\rho}$ and the ray.
\end{rem*}
The next lemma is needed in the next section, when we deal with the
regularity of $u$. It states that the pure second order partial
derivatives of $\rho$ satisfy a monotonicity property. 
\begin{lem}
\label{lem: D2 rho decreas}Suppose the Assumption \ref{assu: 2}
holds. Let $x\in U-R_{\rho}$, and let $y$ be the unique $\rho$-closest
point to $x$ on $\partial U$. Let $A$ be a symmetric positive semidefinite
matrix. Then we have 
\[
\mathrm{tr}[AD^{2}\rho(x)]\le\mathrm{tr}[AD^{2}\rho(y)],
\]
In particular we have $D_{\xi\xi}^{2}\rho(x)\le D_{\xi\xi}^{2}\rho(y)$,
for every $\xi\in\R^{n}$.
\end{lem}
\begin{rem*}
As we mentioned in the introduction, the above monotonicity property
is true because $\rho$ satisfies the Hamilton-Jacobi equation (\ref{eq: H-J eq}),
and  the segment $]x,y[$ is the characteristic curve associated
to it.
\end{rem*}
\begin{proof}
By (\ref{eq: parametrize by rho}) we know that $t\mapsto y+tD\gamma^{\circ}(\mu)$
parametrizes the segment $[y,x]$, when $t$ goes from $0$ to $\gamma(x-y)$.
We also know that $\rho$ is $C^{2}$ on a neighborhood of $[y,x]$,
due to the Theorems \ref{thm: det Q > 0} and \ref{thm: rho is C^2}.
Thus to prove the lemma it suffices to show that 
\[
\frac{d}{dt}\mathrm{tr}[AD^{2}\rho(y+tD\gamma^{\circ}(\mu))]\le0.
\]
To simplify the notation set $\tilde{\mu}:=D\gamma^{\circ}(\mu)$.
Also let $\sqrt{A}$ be the unique symmetric positive semidefinite
matrix whose square is $A$. Then by (\ref{eq: W,Q}),(\ref{eq: D2 rho (x)})
we have 
\begin{align*}
\frac{d}{dt}\mathrm{tr}[AD^{2}\rho(y+t\tilde{\mu})] & =\mathrm{tr}\big[A\frac{d}{dt}D^{2}\rho(y+t\tilde{\mu})\big]=\mathrm{tr}\big[A\frac{d}{dt}\big(D^{2}\rho(y)(I-tW)^{-1}\big)\big]\\
 & =\mathrm{tr}\big[AD^{2}\rho(y)(I-tW)^{-1}W(I-tW)^{-1}\big]\\
 & =-\mathrm{tr}\big[AD^{2}\rho(y)(I-tW)^{-1}D^{2}\gamma^{\circ}(\mu)D^{2}\rho(y)(I-tW)^{-1}\big]\\
 & =-\mathrm{tr}\big[\sqrt{A}\sqrt{A}D^{2}\rho(y+t\tilde{\mu})D^{2}\gamma^{\circ}(\mu)D^{2}\rho(y+t\tilde{\mu})\big]\\
 & =-\mathrm{tr}\big[\sqrt{A}D^{2}\rho(y+t\tilde{\mu})D^{2}\gamma^{\circ}(\mu)D^{2}\rho(y+t\tilde{\mu})\sqrt{A}\big]\le0.
\end{align*}
Note that the matrix in the last line of the above formula is positive
semidefinite, since $D^{2}\gamma^{\circ}(\mu)$ is positive semidefinite,
and the other factors are symmetric. Finally, to get the last statement
of lemma we just need to set $A=\xi\otimes\xi$, which is trivially
a symmetric positive semidefinite matrix. Then it is easy to check
that $\mathrm{tr}[AD^{2}\rho]=D_{\xi\xi}^{2}\rho$.
\end{proof}
\begin{rem*}
At the end of this section, we would like to comment on the structure
of the set $R_{\rho}$, particularly its Hausdorff dimension. Suppose
additionally that $\gamma^{\circ},\varphi$ are $C^{\infty}$. Moreover
suppose that $U$ is connected, and $\partial U$ is $C^{2,1}$. Let
$H(x,p):=\gamma^{\circ}(p+D\varphi(x))$. Then $v(x):=\rho(x)-\varphi(x)$
satisfies the following Hamilton-Jacobi equation 
\[
\begin{cases}
H(x,Dv)=\gamma^{\circ}(D\rho)=1 & \textrm{in }U,\\
v=0 & \textrm{on }\partial U.
\end{cases}
\]
In addition, the sets $\{p\in\R^{n}:H(x,p)<1\}=-D\varphi(x)+\mathrm{int}(K^{\circ})$
are smooth strictly convex sets containing $0$; because we have $\gamma^{\circ}(D\varphi)<1$.
Furthermore, by (\ref{eq: phi strct Lip}), for every $x\in U$ we
have 
\[
v(x)=\rho(x)-\varphi(x)=\gamma(x-y)+\varphi(y)-\varphi(x)>0,
\]
where $y$ is the $\rho$-closest point on $\partial U$ to $x$.
Also note that the set of singularities of $v$ equals $R_{\rho}$,
since $\varphi$ is smooth. Therefore we can apply the result of \citet{MR2094267},
and conclude that 
\[
\mathcal{H}^{n-1}(R_{\rho})<\infty,
\]
where $\mathcal{H}^{n-1}$ is the $n-1$-dimensional Hausdorff measure.
It also follows that $R_{\rho}$ can be covered by countably many
$n-1$-dimensional $C^{1}$ submanifolds, except for a set whose $\mathcal{H}^{n-1}$
measure is zero. In other words, $R_{\rho}$ is $C^{1}$-rectifiable.
It is also shown in \citep{MR2094267} that $R_{\rho}$ is path connected.
In addition, if we merely assume that $\gamma^{\circ},\varphi,\partial U$
are $C^{2,\alpha}$ for some $\alpha>0$, then we can repeat the arguments
in \citep{MR2336304}, and conclude that the Hausdorff dimension of
$R_{\rho}$ is at most $n-\alpha$. However, we will not use these
properties of $R_{\rho}$; so we will not provide the details here.
Finally let us also mention the interesting technique of using projections
in convex geometry to estimate the Hausdorff dimension of the singular
set. See \citep{indrei2013free,indrei2016regularity} for applications
of this technique to the study of the singular set in a double obstacle
problem related to the Monge-Ampere equation.%
\end{rem*}

\section{\label{sec: Global-Optimal-Reg}Global Optimal Regularity}

In this section we prove our main results, aka Theorems \ref{thm: Reg NonConv dom}
and \ref{thm: Reg Conv dom}. We will prove that $u$ belongs to
$C^{1,1}(\overline{U})$, without assuming any regularity about $K$.
To this end, first we prove Theorem \ref{thm: ridge is elastic},
which says that when $\partial K$ is smooth enough, $u$ does not
touch the obstacles $\rho,-\bar{\rho}$ at their singularities.
\begin{proof}[\textbf{Proof of Theorem \ref{thm: ridge is elastic}}]
 We have already shown in Proposition \ref{prop: ridge0 elastic}
that $R_{\rho,0},R_{\bar{\rho},0}$ do not intersect $P^{+},P^{-}$
respectively. Note that as we discussed in the beginning of the last
section, the Assumptions of Theorem \ref{thm: ridge is elastic} imply
the assumptions of Proposition \ref{prop: ridge0 elastic}, since
by Theorem \ref{thm: Reg u} we know that $u\in C_{\mathrm{loc}}^{1,1}(U)$.
So we only need to show that $R_{\rho}-R_{\rho,0},R_{\bar{\rho}}-R_{\bar{\rho},0}$
do not intersect $P^{+},P^{-}$ respectively. Suppose to the contrary
that there is a point $x\in U$ which belongs to $(R_{\rho}-R_{\rho,0})\cap P^{+}$;
the other case is similar. Let $y$ be the unique $\rho$-closest
point to $x$ on $\partial U$. Then by Theorem \ref{thm: ridge}
we must have $\det Q(x)=0$. Now the definition of $Q$ implies that
$\tilde{\kappa}:=\frac{1}{\rho(x)-\varphi(y)}=\frac{1}{\gamma(x-y)}>0$
is an eigenvalue of $W(y)$. Suppose $\zeta$ is the corresponding
eigenvector of $W$, and $|\zeta|=1$. Note that $\zeta$ is not parallel
to the segment $]x,y[$, i.e. to the vector $D\gamma^{\circ}(\mu)$;
because we have $WD\gamma^{\circ}(\mu)=0$.

Let $z\in]x,y[$. Then by Lemma \ref{lem: segment to the closest pt}
we have $z\in U$, and $y$ is the unique $\rho$-closest point on
$\partial U$ to $z$. In addition, by Theorem \ref{thm: det Q > 0}
we have $\det Q(z)\ne0$. Hence by Theorem \ref{thm: rho is C^2},
$\rho$ is $C^{k,\alpha}$ on a neighborhood of the line segment $]x,y[$.
We call this neighborhood $V$. Furthermore, for $z\in]x,y[$ we
have 
\[
Q(z)=I-\gamma(z-y)W=\frac{1}{s\tilde{\kappa}}\big(s\tilde{\kappa}I-W\big),
\]
where $s>1$ is such that $s(z-y)=x-y$. Hence we have $Q(z)\zeta=\frac{1}{s\tilde{\kappa}}(s-1)\tilde{\kappa}\zeta=\frac{s-1}{s}\zeta$.
Then by (\ref{eq: D2 rho (x)}) we have 
\[
D_{\zeta\zeta}^{2}\rho(z)=\zeta D^{2}\rho(z)\zeta=\zeta D^{2}\rho(y)Q(z)^{-1}\zeta=\frac{s}{s-1}\zeta D^{2}\rho(y)\zeta.
\]
We claim that $\zeta D^{2}\rho(y)\zeta<0$. The reason is that for
$\xi:=D^{2}\rho(y)\zeta$ we have $D^{2}\gamma^{\circ}(\mu)\xi=-W\zeta=-\tilde{\kappa}\zeta$.
On the other hand, we know that the eigenvalues of $D^{2}\gamma^{\circ}(\mu)$
are all positive except for one $0$ which corresponds to the eigenvector
$\mu$. We also know that the other eigenvectors of $D^{2}\gamma^{\circ}(\mu)$
are orthogonal to $\mu$, since it is a symmetric matrix. In addition,
as shown in the proof of Theorem \ref{thm: det Q > 0}, $\zeta$ is
orthogonal to $\mu$, since it is an eigenvector of $W$ corresponding
to a nonzero eigenvalue. Let $(\zeta_{1},\dots,\zeta_{n-1},0)$ and
$(\xi_{1},\dots,\xi_{n})$ be the coordinates of $\zeta,\xi$ in the
orthonormal basis consisting of the eigenvectors of $D^{2}\gamma^{\circ}(\mu)$.
Let $\tau_{1},\dots,\tau_{n-1},0$ be the corresponding eigenvalues
of $D^{2}\gamma^{\circ}(\mu)$. Then we have 
\[
(\tau_{1}\xi_{1},\dots,\tau_{n-1}\xi_{n-1},0)=D^{2}\gamma^{\circ}(\mu)\xi=-\tilde{\kappa}\zeta=(-\tilde{\kappa}\zeta_{1},\dots,-\tilde{\kappa}\zeta_{n-1},0).
\]
Hence we have $\xi_{j}=-\frac{\tilde{\kappa}}{\tau_{j}}\zeta_{j}$
for $j<n$. Therefore 
\[
\zeta D^{2}\rho(y)\zeta=\langle\zeta,\xi\rangle=-\tilde{\kappa}\sum_{j<n}\tfrac{1}{\tau_{j}}|\zeta_{j}|^{2}<0,
\]
as desired. As a consequence we have 
\begin{equation}
D_{\zeta\zeta}^{2}\rho(z)=\frac{s}{s-1}\zeta D^{2}\rho(y)\zeta\to-\infty\qquad\textrm{ as }\;z\to x,\label{eq: D2 rho to - infty}
\end{equation}
since $s\to1^{+}$. 

Now since $x\in P^{+}$ we have $u(x)=\rho(x)$. Hence by lemma \ref{lem: segment is plastic}
we have $[x,y[\subset P^{+}$. Thus $u(z)=\rho(z)$ for every $z\in]x,y[$.
Also remember that $u\le\rho$ everywhere, since $u\in W_{\bar{\rho},\rho}$.
Hence $\rho-u$ is a $C^{1}$ function on $V$, which attains its
maximum, $0$, on $]x,y[$. Thus $Du=D\rho$ on the segment $]x,y[$. 

Now we claim that for any $z\in]x,y[$ there are points $z_{i}:=z+\varepsilon_{i}\zeta$
in $V$ converging to $z$, at which we have 
\[
D_{\zeta}u(z_{i})\le D_{\zeta}\rho(z_{i}).
\]
Since otherwise we would have $D_{\zeta}u>D_{\zeta}\rho$ on a segment
of the form $]z,z+r\zeta[$, for some small $r>0$. But as $u(z)=\rho(z)$
and $Du(z)=D\rho(z)$, this implies that $u>\rho$ on $]z,z+r\zeta[$;
which is a contradiction. Thus we get the desired. As a consequence
we have 
\[
D_{\zeta}u(z_{i})-D_{\zeta}u(z)\le D_{\zeta}\rho(z_{i})-D_{\zeta}\rho(z).
\]
By applying the mean value theorem to the restriction of $\rho$ to
the segment $[z,z_{i}]$, we get 
\begin{equation}
D_{\zeta}u(z_{i})-D_{\zeta}u(z)\le|z_{i}-z|D_{\zeta\zeta}^{2}\rho(w_{i}),\label{eq: upper bd on D2u}
\end{equation}
for some $w_{i}\in]z,z_{i}[$. 

On the other hand, $u$ is a $C^{1,1}$ function on a neighborhood
of $x$, due to Theorem \ref{thm: Reg u}. Consequently there is $M>0$
such that 
\begin{equation}
-M\le\frac{D_{\zeta}u(z_{i})-D_{\zeta}u(z)}{|z_{i}-z|},\label{eq: lower bd on D2u}
\end{equation}
for distinct $z,z_{i}$ sufficiently close to $x$. Now let $z\in]x,y[$
be close enough to $x$ so that $D_{\zeta\zeta}^{2}\rho(z)<-3M$,
which is possible due to (\ref{eq: D2 rho to - infty}). Then let
$z_{i}=z+\varepsilon_{i}\zeta$ be close enough to $z$ so that we
have $D_{\zeta\zeta}^{2}\rho(w_{i})<-2M$, which is possible due to
the continuity of $D^{2}\rho$ on $V$. But this is in contradiction
with (\ref{eq: upper bd on D2u}) and (\ref{eq: lower bd on D2u}).
\end{proof}
Before presenting the proof of Theorem \ref{thm: Reg NonConv dom},
let us state a few remarks about it.
\begin{rem*}
Note that the assumptions of the theorem also hold when we replace
$K,\varphi,K^{\circ}$ by $-K,-\varphi$ and $(-K)^{\circ}=-K^{\circ}$.
In particular note that if $D\varphi\in\partial K^{\circ}$, i.e.
if $\gamma^{\circ}(D\varphi(y))=1$, then we have $-D\varphi\in-\partial K^{\circ}=\partial(-K^{\circ})$;
and vice versa. In addition, it is easy to see that 
\[
\mathrm{v}\in N(K^{\circ},D\varphi(y))\iff-\mathrm{v}\in N(-K^{\circ},-D\varphi(y)).
\]
So as a result, $\rho,\bar{\rho}$ will have the same properties.
\end{rem*}
\begin{rem*}
When $\gamma^{\circ}(D\varphi(y))=1$, and $\mathrm{v}\in N(K^{\circ},D\varphi(y))$
is nonzero, then we do not necessarily have $\langle\mathrm{v},\nu(y)\rangle>0$,
in contrast to (\ref{eq: Dg0 . n > 0}). Because $D\varphi(y)$ can
be different from $\mu(y)$, if we define $\mu$ in a continuous way.%
\end{rem*}
\begin{rem*}
As we will see in the following proof, in order to show that $u\in W^{2,p}(U)\cap W_{\mathrm{loc}}^{2,\infty}(U)$
for every $p<\infty$, we only need $\partial U,\varphi$ to be $C^{2}$.
But we need their $C^{2,\alpha}$ regularity to be able to apply the
result of \citep{indrei2016nontransversal}, and conclude the optimal
regularity of $u$ up to the boundary.
\end{rem*}
\begin{proof}[\textbf{Proof of Theorem \ref{thm: Reg NonConv dom}}]
 As shown in \citep{MR0467623}, a compact convex set with nonempty
interior can be approximated, in the Hausdorff metric, by a sequence
of compact convex sets with nonempty interior which have $C^{2}$
boundaries with positive curvature. We can scale each element of such
approximating sequence, to make the sequence a shrinking one. We apply
this result to $K^{\circ}$. Thus there is a sequence $K_{k}^{\circ}$
of compact convex sets, that have $C^{2}$ boundaries with positive
curvature, and 
\begin{eqnarray*}
K_{k+1}^{\circ}\subset\mathrm{int}(K_{k}^{\circ}), & \qquad & K^{\circ}={\textstyle \bigcap}K_{k}^{\circ}.
\end{eqnarray*}
Notice that we can take the approximations of $K^{\circ}$ to be
the polar of other convex sets, because the double polar of a compact
convex set with $0$ in its interior is itself. Also note that $K_{k}$'s
are strictly convex compact sets with $0$ in their interior, which
have $C^{2}$ boundaries with positive curvature. Furthermore we have
$K=(K^{\circ})^{\circ}\supset K_{k+1}\supset K_{k}$. For the proof
of these facts see {[}\citealp{MR3155183}, Sections 1.6, 1.7 and
2.5{]}. 

To simplify the notation we use $\gamma_{k},\gamma_{k}^{\circ},\rho_{k},\bar{\rho}_{k}$
instead of $\gamma_{K_{k}},\gamma_{K_{k}^{\circ}},\rho_{K_{k},\varphi},\bar{\rho}_{K_{k},\varphi}$,
respectively. Note that $K_{k},U,\varphi$ satisfy the Assumption
\ref{assu: 2}. In particular we have $\gamma_{k}^{\circ}(D\varphi)<1$,
since $D\varphi\in K^{\circ}\subset\mathrm{int}(K_{k}^{\circ})$.
Let $u_{k}$ be the minimizer of $J$ over $W_{K_{k}^{\circ},\varphi}(U)$.
Then by Theorem \ref{thm: Reg u} we have $u_{k}\in W_{\textrm{loc}}^{2,\infty}(U)$.
We also have 
\begin{eqnarray*}
-\bar{\rho}_{1}\le-\bar{\rho}_{k}\le u_{k}\le\rho_{k}\le\rho_{1}, & \qquad & Du_{k}\in K_{k}^{\circ}\subset K_{1}^{\circ}\quad\textrm{ a.e.}.
\end{eqnarray*}
Note that $\rho_{k}\le\rho_{1}$ and $\bar{\rho}_{k}\le\bar{\rho}_{1}$,
since $\gamma_{k}\le\gamma_{1}$ due to $K_{k}\supset K_{1}$. Therefore
$u_{k}$ is a bounded sequence in $W^{1,\infty}(U)=C^{0,1}(\overline{U})$.
Hence by the Arzela-Ascoli Theorem a subsequence of $u_{k}$, which
we still denote by $u_{k}$, uniformly converges to a continuous function
$\tilde{u}\in C^{0}(\overline{U})$. Note that $\tilde{u}|_{\partial U}=\varphi$,
because $u_{k}|_{\partial U}=\varphi$ for every $k$. 

We break the rest of this proof into two parts. In Part I we derive
the uniform bound (\ref{eq: 1 in reg Nonconv dom}). And in Part II
we obtain the regularity of $u$ by using this bound.\medskip{}

PART I:\medskip{}

Let $R_{k}$ be the $\rho_{k}$-ridge, and let $E_{k},P_{k}^{\pm}$
be the elastic and plastic regions of $u_{k}$. Let us show that 
\begin{equation}
\|D_{i}(D_{i}F(Du_{k}))\|_{L^{\infty}(U)}\le C,\label{eq: 1 in reg Nonconv dom}
\end{equation}
for some $C$ independent of $k$. To see this, note that by (\ref{eq: diff ineq})
on $E_{k}$ we have 
\[
D_{i}(D_{i}F(Du_{k}))=g'(u_{k}).
\]
But $u_{k}$ is uniformly bounded independently of $k$. So by (\ref{eq: Bnds})
we get the desired bound on $E_{k}$. Next consider $P_{k}^{+}$.
Again by (\ref{eq: diff ineq}) we have 
\[
D_{i}(D_{i}F(Du_{k}))\ge g'(u_{k})\qquad\textrm{ a.e. on }P_{k}^{+}.
\]
Thus similarly we have a lower bound for $D_{i}(D_{i}F(Du_{k}))$
on $P_{k}^{+}$, which is independent of $k$. 

On the other hand, since $P_{k}^{+}$ does not intersect $R_{k}$
due to Theorem \ref{thm: ridge is elastic}, $\rho_{k}$ is at least
$C^{2}$ on $P_{k}^{+}$. Let $\lambda_{k},\mu_{k}$ be defined by
(\ref{eq: lambda}),(\ref{eq: mu}) respectively, when we use $\gamma_{k}^{\circ}$
instead of $\gamma^{\circ}$. Now as $u_{k}=\rho_{k}$ on $P_{k}^{+}$,
Theorem 4.4 of \citep{MR3409135} implies that for a.e. $x\in P_{k}^{+}$
we have $D_{i}(D_{i}F(Du_{k}(x)))=D_{i}(D_{i}F(D\rho_{k}(x)))$. Hence
by Lemma \ref{lem: D2 rho decreas} we have 
\begin{align*}
D_{i}(D_{i}F(Du_{k}(x))) & =D_{i}(D_{i}F(D\rho_{k}(x)))=D_{ij}^{2}F(D\rho_{k}(x))D_{ij}^{2}\rho_{k}(x)\\
 & =\mathrm{tr}[D^{2}F(\mu_{k})D^{2}\rho_{k}(x)]\le\mathrm{tr}[D^{2}F(\mu_{k})D^{2}\rho_{k}(y)],
\end{align*}
where $y$ is the $\rho_{k}$-closest point on $\partial U$ to $x$.
But by (\ref{eq: Bnds}) we know that $D^{2}F$ is bounded. So we
only need to show that $D^{2}\rho_{k}$ is bounded on $\partial U$,
independently of $k$. However by (\ref{eq: D2 rho (y)}), for every
$y\in\partial U$ we have 
\[
D^{2}\rho_{k}(y)=(I-X_{k}^{T})\big(D^{2}\varphi(y)+\lambda_{k}(y)D^{2}d(y)\big)(I-X_{k}),
\]
where $X_{k}:=\frac{1}{\langle D\gamma_{k}^{\circ}(\mu_{k}),\nu\rangle}D\gamma_{k}^{\circ}(\mu_{k})\otimes\nu$.
It is obvious that $D^{2}\varphi,D^{2}d$ are bounded on $\partial U$,
independently of $k$. In addition, note that $\gamma_{k}^{\circ}\ge\gamma_{1}^{\circ}$,
since $K_{k}^{\circ}\subset K_{1}^{\circ}$. Thus by (\ref{eq: lambda})
we have 
\[
\gamma_{1}^{\circ}(D\varphi+\lambda_{k}\nu)\le\gamma_{k}^{\circ}(D\varphi+\lambda_{k}\nu)=1.
\]
Hence by (\ref{eq: bd gama}) applied to $\gamma_{1}^{\circ}$, we
have $|D\varphi+\lambda_{k}\nu|\le C$, for some $C>0$. Therefore
we get $|\lambda_{k}|=|\lambda_{k}\nu|\le C+|D\varphi|$. Thus $\lambda_{k}$
is bounded on $\partial U$ independently of $k$.

Hence we only need to show that the entries of $I-X_{k}$ are bounded
on $\partial U$ independently of $k$. Note that we have $\gamma_{k}(D\gamma_{k}^{\circ}(\mu_{k}))=1$
due to (\ref{eq: g0 (Dg)=00003D1}). Thus $\gamma(D\gamma_{k}^{\circ}(\mu_{k}))\le1$
for every $k$, since $\gamma\le\gamma_{k}$ due to $K\supset K_{k}$.
So $D\gamma_{k}^{\circ}(\mu_{k})$ is bounded independently of $k$.
Therefore it only remains to show that $\langle D\gamma_{k}^{\circ}(\mu_{k}),\nu\rangle$
has a positive lower bound on $\partial U$ independently of $k$.%
{} Note that for every $k$, $\langle D\gamma_{k}^{\circ}(\mu_{k}),\nu\rangle$
is a continuous positive function on the compact set $\partial U$,
due to (\ref{eq: Dg0 . n > 0}). Hence there is $c_{k}>0$ such that
$\langle D\gamma_{k}^{\circ}(\mu_{k}),\nu\rangle\ge c_{k}$. Suppose
to the contrary that $c_{k}$ has a subsequence $c_{k_{j}}\to0$.
Let us denote $k_{j}$ by $j$ for simplicity. Then there is a sequence
$y_{j}\in\partial U$ such that 
\begin{equation}
\langle D\gamma_{j}^{\circ}(\mu_{j}(y_{j})),\nu(y_{j})\rangle\to0.\label{eq: 2 in reg Nonconv dom}
\end{equation}
By passing to another subsequence, we can assume that $y_{j}\to y\in\partial U$,
since $\partial U$ is compact. Now remember that 
\[
\mu_{j}(y_{j})=D\varphi(y_{j})+\lambda_{j}(y_{j})\nu(y_{j}),
\]
where $\lambda_{j}>0$. As we have shown in the last paragraph, $\lambda_{j}$
is bounded on $\partial U$ independently of $j$. Hence by passing
to another subsequence, we can assume that $\lambda_{j}\to\lambda^{*}\ge0$.
Therefore we have 
\[
\mu_{j}(y_{j})\to\mu^{*}:=D\varphi(y)+\lambda^{*}\nu(y).
\]
On the other hand we have $\gamma_{j}^{\circ}(\mu_{j}(y_{j}))=1$.
Hence $\gamma^{\circ}(\mu_{j}(y_{j}))\ge1$, since $\gamma_{j}^{\circ}\le\gamma^{\circ}$
due to $K^{\circ}\subset K_{j}^{\circ}$. Thus we get $\gamma^{\circ}(\mu^{*})\ge1$.
However we cannot have $\gamma^{\circ}(\mu^{*})>1$. Because then
$\mu^{*}$ will have a positive distance from $K^{\circ}$, and therefore
it will have a positive distance from $K_{j}^{\circ}$ for large enough
$j$. But this contradicts the facts that $\mu_{j}(y_{j})\to\mu^{*}$
and $\mu_{j}(y_{j})\in K_{j}^{\circ}$. Thus we must have $\gamma^{\circ}(\mu^{*})=1$,
i.e. $\mu^{*}\in\partial K^{\circ}$.

Now note that $\mathrm{v}_{j}:=D\gamma_{j}^{\circ}(\mu_{j}(y_{j}))\in N(K_{j}^{\circ},\mu_{j}(y_{j}))$.
In addition we have $\gamma_{j}(\mathrm{v}_{j})=1$ due to (\ref{eq: g0 (Dg)=00003D1}).
Hence we have $\mathrm{v}_{j}\in K_{j}\subset K$. Thus by passing
to yet another subsequence, we can assume that $\mathrm{v}_{j}\to\mathrm{v}\in K$.
We also have $\gamma_{1}(\mathrm{v}_{j})\ge1$, since $\gamma_{j}\le\gamma_{1}$
due to $K_{j}\supset K_{1}$. So we get $\gamma_{1}(\mathrm{v})\ge1$.
In particular $\mathrm{v}\ne0$. We claim that $\mathrm{v}\in N(K^{\circ},\mu^{*})$.
To see this, note that we have 
\[
K^{\circ}\subset K_{j}^{\circ}\subset\{z:\langle z-\mu_{j}(y_{j}),\mathrm{v}_{j}\rangle\le0\}.
\]
Hence for every $x\in K^{\circ}$ we have $\langle x-\mu_{j}(y_{j}),\mathrm{v}_{j}\rangle\le0$.
Thus as $j\to\infty$ we get $x-\mu_{j}(y_{j})\to x-\mu^{*}$. So
we have $\langle x-\mu^{*},\mathrm{v}\rangle\le0$. Therefore 
\[
K^{\circ}\subset\{z:\langle z-\mu^{*},\mathrm{v}\rangle\le0\},
\]
as desired. Now by (\ref{eq: 2 in reg Nonconv dom}) we obtain 
\begin{equation}
\langle\mathrm{v},\nu(y)\rangle=\lim\langle\mathrm{v}_{j},\nu(y_{j})\rangle=0.\label{eq: 3 in reg Nonconv dom}
\end{equation}
But if $\gamma^{\circ}(D\varphi(y))<1$ then we must have $\lambda^{*}>0$.
So $D\varphi=\mu^{*}-\lambda^{*}\nu$ belongs to the ray passing through
$\mu^{*}\in\partial K^{\circ}$ in the direction $-\nu$. However,
we know that $D\varphi$ is in the interior of $K^{\circ}$, since
$\gamma^{\circ}(D\varphi)<1$. Thus the ray $t\mapsto\mu^{*}-t\nu$
for $t>0$, passes through the interior of $K^{\circ}$. Therefore
this ray and $K^{\circ}$ must lie on the same side of the supporting
hyperplane $H_{\mu^{*},\mathrm{v}}$. In addition, the ray cannot
lie on the hyperplane, since it intersects the interior of $K^{\circ}$.
Hence we must have $\langle\mathrm{v},\nu\rangle=-\langle\mathrm{v},-\nu\rangle>0$,
which contradicts (\ref{eq: 3 in reg Nonconv dom}). 

Thus we must have $\gamma^{\circ}(D\varphi(y))=1$, i.e. $D\varphi\in\partial K^{\circ}$.
If $\lambda^{*}=0$ then $\mu^{*}=D\varphi$. Hence $\mathrm{v}\in N(K^{\circ},D\varphi)$.%
{} Then (\ref{eq: 3 in reg Nonconv dom}) is in contradiction with our
assumption (\ref{eq: 0 in reg Nonconv dom}), since we showed that
$\mathrm{v}\ne0$. So suppose $\lambda^{*}>0$. Then the ray $t\mapsto\mu^{*}-t\nu$
for $t>0$, passes through the two points $D\varphi,\mu^{*}\in\partial K^{\circ}$.
Furthermore, (\ref{eq: 3 in reg Nonconv dom}) implies that the ray
lies on the supporting hyperplane $H_{\mu^{*},\mathrm{v}}$. Therefore
$D\varphi(y)\in H_{\mu^{*},\mathrm{v}}$. Hence $H_{\mu^{*},\mathrm{v}}$
is also a supporting hyperplane of $K^{\circ}$ at $D\varphi(y)$.
So $\mathrm{v}\in N(K^{\circ},D\varphi)$, and again we arrive at
a contradiction with (\ref{eq: 0 in reg Nonconv dom}). 

Thus $\langle D\gamma_{k}^{\circ}(\mu_{k}),\nu\rangle$ must have
a positive lower bound on $\partial U$ independently of $k$, as
desired. Therefore $D^{2}\rho_{k}$ is bounded on $\partial U$ independently
of $k$; and consequently we have an upper bound for $D_{i}(D_{i}F(Du_{k}))$
on $P_{k}^{+}$, which is independent of $k$. Similarly, we can show
that $D_{i}(D_{i}F(Du_{k}))$ is bounded on $P_{k}^{-}$, independently
of $k$. Hence we obtain the desired bound (\ref{eq: 1 in reg Nonconv dom}).\medskip{}

PART II:\medskip{}

Now let $g_{k}:=D_{i}(D_{i}F(Du_{k}))$. Then $u_{k}$ is a weak
solution to the quasilinear elliptic equation 
\[
D_{i}(D_{i}F(Du_{k}))=g_{k},\qquad\qquad u_{k}|_{\partial U}=\varphi.
\]
Thus as shown in \citep{MR749677} we have $u_{k}\in C^{1,\alpha_{0}}(\overline{U})$
for some $\alpha_{0}>0$. In addition, $\|u_{k}\|_{C^{1,\alpha_{0}}(\overline{U})}$
is bounded independently of $k$ as shown for example in \citep{fan2007global}.
Then we get 
\[
a_{ij,k}D_{ij}^{2}u_{k}=g_{k},\qquad\qquad u_{k}|_{\partial U}=\varphi,
\]
where $a_{ij,k}:=D_{ij}^{2}F(Du_{k})$ is continuous on $\overline{U}$.
Hence by Theorems 9.13 and 9.15 of \citep{MR1814364} we have $u_{k}\in W^{2,p}(U)$
for every $p<\infty$, and there exists $C_{p}>0$ independent of
$k$ such that 
\begin{equation}
\|u_{k}\|_{W^{2,p}(U)}\le C_{p}\big(\|g_{k}\|_{L^{p}(U)}+\|u_{k}\|_{L^{p}(U)}\big).\label{eq: u k in W2,p}
\end{equation}
Here we used the fact that $a_{ij,k}$'s have a uniform modulus of
continuity independently of $k$, due to the uniform boundedness of
the $C^{\alpha_{0}}$ norm of $Du_{k}$. In fact we can replace the
$\|u_{k}\|_{L^{p}}$ on the right hand side of the above inequality
with $\|u_{k}\|_{L^{1}}$; see Theorem 2.2 of \citep{andersson2013optimal}
and the references therein. Therefore $u_{k}$ is a bounded sequence
in $W^{2,p}(U)$ due to (\ref{eq: 1 in reg Nonconv dom}).

Hence there is a subsequence of $u_{k}$, which we still denote by
$u_{k}$, that is weakly convergent in $W^{2,p}(U)$, and strongly
convergent in $C^{1}(\overline{U})$. Also remember that $u_{k}$
uniformly converges to a continuous function $\tilde{u}\in C^{0}(\overline{U})$
that satisfies $\tilde{u}|_{\partial U}=\varphi$. Thus all the above
limits must be equal to $\tilde{u}$. As a result, $\tilde{u}$ belongs
to $W^{2,p}(U)$ for every $p<\infty$. Furthermore we have $D\tilde{u}\in K^{\circ}$;
because $Du_{k}\in K_{k}^{\circ}$, and thus $D\tilde{u}\in K_{k}^{\circ}$
for every $k$. So we have $\tilde{u}\in W_{K^{\circ},\varphi}(U)$,
since $\tilde{u}|_{\partial U}=\varphi$. Now we will show that $\tilde{u}$
is the minimizer of $J$ over $W_{K^{\circ},\varphi}(U)$. Let $v\in W_{K^{\circ},\varphi}(U)$.
Then $v\in W_{K_{k}^{\circ},\varphi}(U)$ for every $k$. Thus we
get $J[u_{k}]\le J[v]$. But by the Dominated Convergence Theorem
we have $J[u_{k}]\to J[\tilde{u}]$. Hence $\tilde{u}$ is the minimizer,
and we must have $\tilde{u}=u$. Therefore $u\in W^{2,p}(U)$ for
every $p<\infty$.

Finally let us show that $u$ belongs to $W^{2,\infty}(U)$. First
we show that $D^{2}u_{k}$ is bounded on $P_{k}$ independently of
$k$. To see this, consider $P_{k}^{+}$; the other case is similar.
We have shown in the Part I of the proof that $D^{2}\rho_{k}$ is
bounded on $\partial U$,%
{} independently of $k$. Hence by Lemma \ref{lem: D2 rho decreas},
when $y$ is the $\rho_{k}$-closest point on $\partial U$ to $x\in P_{k}^{+}$,
and $\xi\in\R^{n}$, we have 
\begin{equation}
D_{\xi\xi}^{2}\rho_{k}(x)\le D_{\xi\xi}^{2}\rho_{k}(y)\le\tilde{C},\label{eq: D2 rho k is bdd}
\end{equation}
for some $\tilde{C}$ independent of $k$. In addition, for a.e. $x\in P_{k}^{+}$
we have 
\begin{align*}
D_{i}(D_{i}F(Du_{k}(x))) & =D_{i}(D_{i}F(D\rho_{k}(x)))=D_{ij}^{2}F(D\rho_{k}(x))D_{ij}^{2}\rho_{k}(x)\\
 & =\mathrm{tr}[D^{2}F(\mu_{k})D^{2}\rho_{k}(x)]=\underset{i\le n}{\sum}D_{\xi_{i}\xi_{i}}^{2}F(\mu_{k})D_{\xi_{i}\xi_{i}}^{2}\rho_{k}(x),
\end{align*}
where $\xi_{1},\cdots,\xi_{n}$ is an orthonormal basis of eigenvectors
of $D^{2}\rho_{k}(x)$. Hence by (\ref{eq: diff ineq}) we get 
\[
\underset{i\le n}{\sum}D_{\xi_{i}\xi_{i}}^{2}F(\mu_{k})D_{\xi_{i}\xi_{i}}^{2}\rho_{k}(x)=D_{i}(D_{i}F(Du_{k}(x)))\ge g'(u_{k}(x))\ge-C,
\]
for some $C$ independent of $k$. Note that here we are using (\ref{eq: Bnds}),
and the fact that $u_{k}$ is bounded independently of $k$. Therefore
by (\ref{eq: Bnds}),(\ref{eq: D2 rho k is bdd}) we obtain 
\[
D_{\xi_{j}\xi_{j}}^{2}F(\mu_{k})D_{\xi_{j}\xi_{j}}^{2}\rho_{k}(x)\ge-C-\underset{i\ne j}{\sum}D_{\xi_{i}\xi_{i}}^{2}F(\mu_{k})D_{\xi_{i}\xi_{i}}^{2}\rho_{k}(x)\ge-C-(n-1)c_{9}\tilde{C}.
\]
Hence again by (\ref{eq: Bnds}) we obtain 
\[
D_{\xi_{j}\xi_{j}}^{2}\rho_{k}(x)\ge-\frac{C}{c_{8}}-(n-1)\frac{c_{9}}{c_{8}}\tilde{C}.
\]
Thus by (\ref{eq: D2 rho k is bdd}) and the above inequality, $|D_{\xi_{j}\xi_{j}}^{2}\rho_{k}(x)|$
is bounded independently of $k$, for a.e. $x\in P_{k}^{+}$. Now
note that the numbers $D_{\xi_{j}\xi_{j}}^{2}\rho_{k}(x)$ are the
eigenvalues of $D^{2}\rho_{k}(x)$. So $D^{2}\rho_{k}$ is bounded
on $P_{k}^{+}$ independently of $k$. But on $P_{k}^{+}$ we have
$D^{2}u_{k}=D^{2}\rho_{k}$ a.e.. Therefore $D^{2}u_{k}$ is bounded
on $P_{k}^{+}$ independently of $k$. The case of $P_{k}^{-}$ can
be treated similarly.

Now let $x_{0}\in U$, and suppose that $B_{r}(x_{0})\subset U$.
Set $v_{k}(y):=u_{k}(x_{0}+ry)$ for $y\in B_{1}(0)$. Then by (\ref{eq: diff ineq}),
and the arguments of the above paragraph, we have 
\[
\begin{cases}
D_{ij}^{2}F(\frac{1}{r}Dv_{k})D_{ij}^{2}v_{k}=r^{2}g'(v_{k}) & \textrm{ a.e. in }B_{1}(0)\cap\Omega_{k},\\
|D^{2}v_{k}|\le C & \textrm{ a.e. in }B_{1}(0)-\Omega_{k},
\end{cases}
\]
for some $C$ independent of $k$. Here $\Omega_{k}:=\{y\in B_{1}(0):u_{k}(x_{0}+ry)\in E_{k}\}$.
(Note that on $B_{1}(0)-\Omega_{k}$ we have $D^{2}v_{k}=r^{2}D^{2}u_{k}$,
and $u_{k}\in P_{k}$; so $D^{2}v_{k}$ is bounded there.) Next recall
that $\|u_{k}\|_{W^{2,n}(B_{r}(x_{0}))}$, $\|g'(u_{k})\|_{L^{\infty}(B_{r}(x_{0}))}$
are bounded independently of $k$, due to (\ref{eq: u k in W2,p}),(\ref{eq: 1 in reg Nonconv dom}),
and the fact that $u_{k}$ is bounded independently of $k$. Therefore
$\|v_{k}\|_{W^{2,n}(B_{1}(0))}$ and $\|g'(v_{k})\|_{L^{\infty}(B_{1}(0))}$
are bounded independently of $k$ too. Also note that the Holder norms
of $D_{ij}^{2}F(\frac{1}{r}Dv_{k}),r^{2}g'(v_{k})$ are bounded independently
of $k$, since for every $\tilde{\alpha}<1$, $\|u_{k}\|_{C^{1,\tilde{\alpha}}(\overline{U})}$
is bounded independently of $k$. Thus we can apply the result of
\citep{Indrei-Minne} to deduce that 
\[
|D^{2}v_{k}|\le\bar{C}\qquad\textrm{ a.e. in }B_{\frac{1}{2}}(0),
\]
for some $\bar{C}$ independent of $k$. Therefore 
\[
|D^{2}u_{k}|\le\tilde{C}\qquad\textrm{ a.e. in }B_{\frac{r}{2}}(x_{0}),
\]
for some $\tilde{C}$ independent of $k$. Hence $u_{k}$ is a bounded
sequence in $W^{2,\infty}(B_{\frac{r}{2}}(x_{0}))$. Therefore a
subsequence of them converges weakly star in $W^{2,\infty}(B_{\frac{r}{2}}(x_{0}))$.
But the limit must be $u$; so we get $u\in W^{2,\infty}(B_{\frac{r}{2}}(x_{0}))$,
as desired.

Next let $x_{0}\in\partial U$. Let $\Phi$ be a $C^{2,\alpha}$ change
of coordinates on a neighborhood of $x_{0}$, that flattens $\partial U$
around $x_{0}$. More specifically, we assume that $\Phi:x\mapsto y$
maps a neighborhood of $x_{0}$ onto a neighborhood of $0$ that contains
$\overline{B}_{1}(0)$, and the $\Phi$-image of $U,\partial U$
lie respectively in the half-space $\{y_{n}>0\}$ and on the plane
$\{y_{n}=0\}$. Let $\Psi$ be the inverse of $\Phi$. Then we have
$y=\Phi(x)$ and $x=\Psi(y)$. Let $B_{1}^{+}:=B_{1}(0)\cap\{y_{n}>0\}$
and $B_{1}':=B_{1}(0)\cap\{y_{n}=0\}$. Now set 
\[
\hat{u}_{k}(y):=u_{k}(\Psi(y))-\varphi(\Psi(y))=u_{k}(x)-\varphi(x).
\]
It is obvious that $\hat{u}_{k}=0$ on $B_{1}'$. We also have $\hat{u}_{k}\in W^{2,n}(B_{1}^{+})\cap C^{1}(\overline{B}_{1}^{+})$
(see {[}\citealp{MR1814364}, Section 7.3{]}). In addition we have
\begin{align}
D_{\hat{i}}\hat{u}_{k}(y) & =D_{i}u_{k}(x)D_{\hat{i}}\Psi^{i}(y)-D_{i}\varphi(x)D_{\hat{i}}\Psi^{i}(y),\nonumber \\
D_{\hat{i}\hat{j}}^{2}\hat{u}_{k}(y) & =D_{ij}^{2}u_{k}(x)D_{\hat{i}}\Psi^{i}(y)D_{\hat{j}}\Psi^{j}(y)+D_{i}u_{k}(x)D_{\hat{i}\hat{j}}^{2}\Psi^{i}(y)\label{eq: D u hat}\\
 & \qquad\qquad-D_{ij}^{2}\varphi(x)D_{\hat{i}}\Psi^{i}(y)D_{\hat{j}}\Psi^{j}(y)-D_{i}\varphi(x)D_{\hat{i}\hat{j}}^{2}\Psi^{i}(y).\nonumber 
\end{align}
Therefore we get 
\[
\|\hat{u}_{k}\|_{W^{2,n}(B_{1}^{+})}\le C\big(\|u_{k}\|_{W^{2,n}(U)}+\|\varphi\|_{W^{2,\infty}(U)}\big),
\]
for some $C$ independent of $k$. Hence $\|\hat{u}_{k}\|_{W^{2,n}(B_{1}^{+})}$
is bounded independently of $k$, due to (\ref{eq: u k in W2,p}).

Now note that by (\ref{eq: diff ineq}) we have $D_{ij}^{2}F(Du_{k})D_{ij}^{2}u_{k}=D_{i}(D_{i}F(Du_{k}))=g'(u_{k})$
in $E_{k}$. Let $a_{ij,k}:=D_{ij}^{2}F(Du_{k})$, and 
\[
\hat{a}_{\hat{i}\hat{j},k}(y):=a_{lm,k}(x)D_{l}\Phi^{\hat{i}}(x)D_{m}\Phi^{\hat{j}}(x),
\]
where $x=\Psi(y)$. Then we have (Note that we sum over all repeated
indices except for $k$.)
\begin{align*}
\hat{a}_{\hat{i}\hat{j},k}D_{\hat{i}\hat{j}}^{2}\hat{u}_{k}(y) & =a_{lm,k}D_{l}\Phi^{\hat{i}}D_{m}\Phi^{\hat{j}}\big(D_{ij}^{2}u_{k}D_{\hat{i}}\Psi^{i}D_{\hat{j}}\Psi^{j}+D_{i}u_{k}D_{\hat{i}\hat{j}}^{2}\Psi^{i}\\
 & \qquad\quad\qquad\qquad\qquad-D_{ij}^{2}\varphi D_{\hat{i}}\Psi^{i}D_{\hat{j}}\Psi^{j}-D_{i}\varphi D_{\hat{i}\hat{j}}^{2}\Psi^{i}\big)\\
 & =a_{ij,k}D_{ij}^{2}u_{k}-a_{ij,k}D_{ij}^{2}\varphi+a_{lm,k}D_{l}\Phi^{\hat{i}}D_{m}\Phi^{\hat{j}}\big(D_{i}u_{k}D_{\hat{i}\hat{j}}^{2}\Psi^{i}-D_{i}\varphi D_{\hat{i}\hat{j}}^{2}\Psi^{i}\big)\\
 & =g'(u_{k})-a_{ij,k}D_{ij}^{2}\varphi+a_{lm,k}D_{l}\Phi^{\hat{i}}D_{m}\Phi^{\hat{j}}\big(D_{i}u_{k}D_{\hat{i}\hat{j}}^{2}\Psi^{i}-D_{i}\varphi D_{\hat{i}\hat{j}}^{2}\Psi^{i}\big)=:\hat{f}_{k}(y),
\end{align*}
for every $y$ in $\Omega_{k}:=\{y\in B_{1}^{+}:\Psi(y)\in E_{k}\}$.
Note that in the 2nd equality we used the fact that $D\Psi D\Phi=I$.
Also note that $\hat{f}_{k}\in C^{\alpha_{0}}(\overline{B}_{1}^{+})$
for some $\alpha_{0}>0$, and $\|\hat{f}_{k}\|_{C^{\alpha_{0}}(\overline{B}_{1}^{+})}$
is bounded independently of $k$; since $\|u_{k}\|_{C^{1,\tilde{\alpha}}(\overline{U})}$
is bounded independently of $k$, for every $\tilde{\alpha}<1$. Similarly,
$\|\hat{a}_{\hat{i}\hat{j},k}\|_{C^{\alpha_{0}}(\overline{B}_{1}^{+})}$
is bounded independently of $k$. 

On the other hand, $D^{2}\hat{u}_{k}$ is bounded on $B_{1}^{+}-\Omega_{k}:=\{y\in B_{1}^{+}:\Psi(y)\in P_{k}\}$
independently of $k$ due to (\ref{eq: D u hat}); because $D^{2}u_{k}$
is bounded on $P_{k}$ independently of $k$, and $Du_{k}$ is bounded
independently of $k$. Therefore we have 
\[
\begin{cases}
\hat{a}_{\hat{i}\hat{j},k}D_{\hat{i}\hat{j}}^{2}\hat{u}_{k}=\hat{f}_{k} & \textrm{ a.e. in }B_{1}^{+}\cap\Omega_{k},\\
|D^{2}\hat{u}_{k}|\le C & \textrm{ a.e. in }B_{1}^{+}-\Omega_{k},\\
\hat{u}_{k}=0 & \textrm{ on }B_{1}',
\end{cases}
\]
for some $C$ independent of $k$. Hence by the result of \citep{indrei2016nontransversal}
we get 
\[
|D^{2}\hat{u}_{k}|\le\bar{C}\qquad\textrm{ a.e. in }B_{\frac{1}{2}}(0)\cap\{y_{n}>0\},
\]
for some $\bar{C}$ independent of $k$. Thus 
\[
|D^{2}u_{k}|\le\tilde{C}\qquad\textrm{ a.e. in }B_{r}(x_{0})\cap U,
\]
for some $r>0$ and some $\tilde{C}$ independent of $k$. Because
we have 
\[
u_{k}(x)=\hat{u}_{k}(y)+\varphi(x)=\hat{u}_{k}(\Psi(x))+\varphi(x);
\]
so we can compute $D^{2}u_{k}$ in terms of $D^{2}\hat{u}_{k}$, similarly
to (\ref{eq: D u hat}). (Also note that $D\hat{u}_{k}$ is bounded
independently of $k$ due to (\ref{eq: D u hat}).) Hence $u_{k}$
is a bounded sequence in $W^{2,\infty}(B_{r}(x_{0})\cap U)$. Therefore
a subsequence of them converges weakly star in $W^{2,\infty}(B_{r}(x_{0})\cap U)$.
But the limit must be $u$; so we get $u\in W^{2,\infty}(B_{r}(x_{0})\cap U)$.
Finally note that we can cover $\partial U$ with finitely many open
balls of the form $B_{r}(x_{0})$ for $x_{0}\in\partial U$, over
which $u$ is $W^{2,\infty}$. Also, there is an open subset of $U$
whose union with these balls cover $U$, and over it $u$ is $W^{2,\infty}$
too. Thus we can conclude that $u\in W^{2,\infty}(U)$, as desired.
\end{proof}
Next we present the proof of Theorem \ref{thm: Reg Conv dom}. Note
that as we mentioned before the proof of Theorem \ref{thm: Reg NonConv dom},
the assumptions of Theorem \ref{thm: Reg Conv dom} also hold when
we replace $K,\varphi,K^{\circ}$ by $-K,-\varphi$ and $(-K)^{\circ}=-K^{\circ}$.
So as a result, $\rho,\bar{\rho}$ will have the same properties.%

\begin{proof}[\textbf{Proof of Theorem \ref{thm: Reg Conv dom}}]
 The idea of the proof is to approximate both $K^{\circ}$ and $U$
with smooth convex sets. As we explained in the proof of last theorem,
there is a sequence $K_{k}^{\circ}$ of compact convex sets, that
have $C^{2}$ boundaries with positive curvature, and 
\begin{eqnarray*}
K_{k+1}^{\circ}\subset\mathrm{int}(K_{k}^{\circ}), & \qquad & K^{\circ}={\textstyle \bigcap}K_{k}^{\circ}.
\end{eqnarray*}
Then it follows that $K_{k}$'s are strictly convex compact sets
with $0$ in their interior, which have $C^{2}$ boundaries with positive
curvature. Furthermore we have $K=(K^{\circ})^{\circ}\supset K_{k+1}\supset K_{k}$.
Similarly, there is a sequence $U_{k}$ of bounded convex open sets
with $C^{2}$ boundaries such that 
\begin{eqnarray*}
\overline{U}_{k+1}\subset U_{k}, & \qquad & \overline{U}={\textstyle \bigcap}U_{k}.
\end{eqnarray*}

To simplify the notation we use $\gamma_{k},\gamma_{k}^{\circ},\rho_{k},\bar{\rho}_{k}$
instead of $\gamma_{K_{k}},\gamma_{K_{k}^{\circ}},\rho_{K_{k},\varphi}(\cdot;U_{k}),\bar{\rho}_{K_{k},\varphi}(\cdot;U_{k})$,
respectively. Note that $K_{k},U_{k},\varphi$ satisfy the Assumption
\ref{assu: 2}. In particular we have $\gamma_{k}^{\circ}(D\varphi)<1$,
since $D\varphi\in K^{\circ}\subset\mathrm{int}(K_{k}^{\circ})$.
Let $u_{k}$ be the minimizer of $J[\cdot;U_{k}]$ over $W_{K_{k}^{\circ},\varphi}(U_{k})$.
Then by Theorem \ref{thm: Reg u} we have $u_{k}\in W_{\textrm{loc}}^{2,\infty}(U_{k})$.
We also have 
\begin{eqnarray*}
-\bar{\rho}_{k}\le u_{k}\le\rho_{k}, & \qquad & Du_{k}\in K_{k}^{\circ}\subset K_{1}^{\circ}\quad\textrm{ a.e.}.
\end{eqnarray*}
Now note that for every $x$ we have $\rho_{k}(x)=\gamma_{k}(x-y_{k})+\varphi(y_{k})$
for some $y_{k}\in\partial U_{k}$. Also note that $\gamma_{k}\le\gamma_{1}$,
since $K_{k}\supset K_{1}$. Hence by (\ref{eq: phi Lip}),(\ref{eq: bd gama}),
for every $y\in\partial U$ we have 
\begin{align*}
\rho_{k}(x) & \le\gamma_{1}(x-y_{k})+\varphi(y_{k})\\
 & \le\gamma_{1}(x-y)+\gamma_{1}(y-y_{k})+\varphi(y)+\varphi(y_{k})-\varphi(y)\\
 & \le\gamma_{1}(x-y)+\varphi(y)+\gamma_{1}(y-y_{k})+\gamma_{1}(y_{k}-y)\\
 & \le\gamma_{1}(x-y)+\varphi(y)+2C|y-y_{k}|,
\end{align*}
for some $C>0$. Therefore we get $\rho_{k}(x)\le\rho_{K_{1},\varphi}(x;U)+2C\mathrm{dist}(\partial U,\partial U_{k})$.
We have a similar bound for $\bar{\rho}_{k}$. So we obtain 
\begin{equation}
-\bar{\rho}_{K_{1},\varphi}(\cdot;U)-2C\mathrm{dist}(\partial U,\partial U_{k})\le u_{k}\le\rho_{K_{1},\varphi}(\cdot;U)+2C\mathrm{dist}(\partial U,\partial U_{k}).\label{eq: 1 in reg Conv dom}
\end{equation}
Thus $u_{k},Du_{k}$ are bounded independently of $k$. Therefore
$u_{k}$ is a bounded sequence in $W^{1,\infty}(U)=C^{0,1}(\overline{U})$.
Note that here we are using the fact that $\partial U$ is Lipschitz,
since locally it is the graph of a convex function. Hence by the
Arzela-Ascoli Theorem a subsequence of $u_{k}$, which we still denote
by $u_{k}$, uniformly converges to a continuous function $\tilde{u}\in C^{0}(\overline{U})$.
In addition we have $\tilde{u}|_{\partial U}=\varphi$, because in
the limit, the bounds (\ref{eq: 1 in reg Conv dom}) become $-\bar{\rho}_{K_{1},\varphi}(\cdot;U)\le\tilde{u}\le\rho_{K_{1},\varphi}(\cdot;U)$.
Thus we get the desired, since $-\bar{\rho}_{K_{1},\varphi}(\cdot;U)|_{\partial U}=\rho_{K_{1},\varphi}(\cdot;U)|_{\partial U}=\varphi$,
as shown in the proof of Proposition \ref{prop: equiv}.

Now we argue as we did in the proof of Theorem \ref{thm: Reg NonConv dom}.
Let $R_{k}$ be the $\rho_{k}$-ridge, and let $E_{k},P_{k}^{\pm}$
be the elastic and plastic regions of $u_{k}$. We want to show that
\begin{equation}
\|D_{i}(D_{i}F(Du_{k}))\|_{L^{\infty}(U_{k})}\le C,\label{eq: 2 in reg Conv dom}
\end{equation}
for some $C$ independent of $k$. As we have shown before by using
(\ref{eq: diff ineq}), it is easy to see that $D_{i}(D_{i}F(Du_{k}))$
is bounded on $E_{k}$, and it is bounded above on $P_{k}^{-}$,
and bounded below on $P_{k}^{+}$. The hard part is to obtain upper
bound for $D_{i}(D_{i}F(Du_{k}))$ on $P_{k}^{+}$, and lower bound
for it on $P_{k}^{-}$, which are independent of $k$. We will obtain
the upper bound on $P_{k}^{+}$; the other case is similar.

Since $P_{k}^{+}$ does not intersect $R_{k}$ due to Theorem \ref{thm: ridge is elastic},
$\rho_{k}$ is at least $C^{2}$ on $P_{k}^{+}$. Let $\nu_{k}$ be
the inward unit normal to $\partial U_{k}$. Also let $\lambda_{k},\mu_{k}$
be defined by (\ref{eq: lambda}),(\ref{eq: mu}) respectively, when
we use $\gamma_{k}^{\circ}$ instead of $\gamma^{\circ}$, and $\partial U_{k},\nu_{k}$
instead of $\partial U,\nu$. Hence as before, by Lemma \ref{lem: D2 rho decreas},
for a.e. $x\in P_{k}^{+}$ we have 
\begin{align*}
D_{i}(D_{i}F(Du_{k}(x))) & =D_{i}(D_{i}F(D\rho_{k}(x)))=D_{ij}^{2}F(D\rho_{k}(x))D_{ij}^{2}\rho_{k}(x)\\
 & =\mathrm{tr}[D^{2}F(\mu_{k})D^{2}\rho_{k}(x)]\le\mathrm{tr}[D^{2}F(\mu_{k})D^{2}\rho_{k}(y)],
\end{align*}
where $y$ is the $\rho_{k}$-closest point on $\partial U_{k}$ to
$x$. But by (\ref{eq: D2 rho (y)}), for every $y\in\partial U_{k}$
we have 
\[
D^{2}\rho_{k}(y)=(I-X_{k}^{T})\big(D^{2}\varphi(y)+\lambda_{k}(y)D^{2}d_{k}(y)\big)(I-X_{k}),
\]
where $X_{k}:=\frac{1}{\langle D\gamma_{k}^{\circ}(\mu_{k}),\nu_{k}\rangle}D\gamma_{k}^{\circ}(\mu_{k})\otimes\nu_{k}$,
and $d_{k}$ is the Euclidean distance to $\partial U_{k}$. However,
we know that the eigenvalues of $D^{2}d_{k}(y)$ are minus the principal
curvatures of $\partial U_{k}$ at $y$, and $0$; as shown in {[}\citealp{MR1814364},
Section 14.6{]}. So $D^{2}d_{k}(y)$ is negative semidefinite, since
$U_{k}$ is convex. Thus $\lambda_{k}(I-X_{k}^{T})D^{2}d_{k}(I-X_{k})$
is also negative semidefinite, since $\lambda_{k}>0$. On the other
hand by (\ref{eq: Bnds}), we know that $D^{2}F$ is positive definite.
Let $\sqrt{D^{2}F}$ be the unique symmetric positive definite matrix
whose square is $D^{2}F$. Then we get 
\begin{align*}
D_{i}(D_{i}F(Du_{k}(x))) & -\mathrm{tr}[D^{2}F(I-X_{k}^{T})D^{2}\varphi(I-X_{k})]\\
 & \le\mathrm{tr}[D^{2}F\lambda_{k}(I-X_{k}^{T})D^{2}d_{k}(I-X_{k})]\\
 & =\lambda_{k}\mathrm{tr}[\sqrt{D^{2}F}(I-X_{k}^{T})D^{2}d_{k}(I-X_{k})\sqrt{D^{2}F}]\le0,
\end{align*}
because $\sqrt{D^{2}F}(I-X_{k}^{T})D^{2}d_{k}(I-X_{k})\sqrt{D^{2}F}$
is negative semidefinite too. 

Hence we only need to find a bound for $\mathrm{tr}[D^{2}F(I-X_{k}^{T})D^{2}\varphi(I-X_{k})]$.
When $\varphi$ is linear we have $D^{2}\varphi=0$; so we have the
desired bound. Thus let us assume that $\varphi$ is not necessarily
linear. By (\ref{eq: Bnds}) we know that $D^{2}F$ is bounded. It
is also obvious that $D^{2}\varphi$ is bounded on $\partial U_{k}\subset\overline{U}_{1}$,
independently of $k$. Hence we only need to show that the entries
of $I-X_{k}$ are bounded on $\partial U_{k}$ independently of $k$.
Note that we have $\gamma_{k}(D\gamma_{k}^{\circ}(\mu_{k}))=1$ due
to (\ref{eq: g0 (Dg)=00003D1}). Thus $\gamma(D\gamma_{k}^{\circ}(\mu_{k}))\le1$
for every $k$, since $\gamma\le\gamma_{k}$ due to $K\supset K_{k}$.
So $D\gamma_{k}^{\circ}(\mu_{k})$ is bounded independently of $k$.
Also, $\nu_{k}$ is bounded independently of $k$, since $|\nu_{k}|=1$.
Therefore it only remains to show that $\langle D\gamma_{k}^{\circ}(\mu_{k}),\nu_{k}\rangle$
has a positive lower bound on $\partial U_{k}$ independently of $k$.%
{} 

Note that for every $k$, $\langle D\gamma_{k}^{\circ}(\mu_{k}),\nu_{k}\rangle$
is a continuous positive function on the compact set $\partial U_{k}$,
due to (\ref{eq: Dg0 . n > 0}). Hence there is $c_{k}>0$ such that
$\langle D\gamma_{k}^{\circ}(\mu_{k}),\nu_{k}\rangle\ge c_{k}$. Suppose
to the contrary that $c_{k}$ has a subsequence $c_{k_{j}}\to0$.
Let us denote $k_{j}$ by $j$ for simplicity. Then there is a sequence
$y_{j}\in\partial U_{j}\subset\overline{U}_{1}$ such that 
\begin{equation}
\langle D\gamma_{j}^{\circ}(\mu_{j}(y_{j})),\nu_{j}(y_{j})\rangle\to0.\label{eq: 22 in reg Conv dom}
\end{equation}
By passing to another subsequence, we can assume that $y_{j}\to y$,
since $\overline{U}_{1}$ is compact. We claim that $y\in\partial U$.
First note that $y\notin U$, since otherwise $y_{j}$ must be in
$U$ for large enough $j$, which contradicts the fact that $y_{j}\in\partial U_{j}$.
On the other hand, we have $y\in\overline{U}_{k}$ for every $k$;
because $y_{j}\in\overline{U}_{k}$ for large enough $j$. So we must
have $y\in\bigcap\overline{U}_{k}=\overline{U}$. 

Now remember that 
\[
\mu_{j}(y_{j})=D\varphi(y_{j})+\lambda_{j}(y_{j})\nu_{j}(y_{j}),
\]
where $\lambda_{j}>0$. In addition, note that $\gamma_{j}^{\circ}\ge\gamma_{1}^{\circ}$,
since $K_{j}^{\circ}\subset K_{1}^{\circ}$. Thus by (\ref{eq: lambda})
we have 
\[
\gamma_{1}^{\circ}(D\varphi+\lambda_{j}\nu_{j})\le\gamma_{j}^{\circ}(D\varphi+\lambda_{j}\nu_{j})=1.
\]
Hence by (\ref{eq: bd gama}) applied to $\gamma_{1}^{\circ}$, we
have $|D\varphi+\lambda_{j}\nu_{j}|\le C$, for some $C>0$. Therefore
we get $|\lambda_{j}|=|\lambda_{j}\nu_{j}|\le C+|D\varphi|$. Thus
$\lambda_{j}$ is bounded on $\partial U_{j}$ independently of $j$.
Hence by passing to another subsequence, we can assume that $\lambda_{j}\to\lambda^{*}\ge0$.
Also, note that $-\nu_{j}(y_{j})\in N(U_{j},y_{j})$. In addition
we have $|\nu_{j}(y_{j})|=1$. Thus by passing to yet another subsequence,
we can assume that $-\nu_{j}(y_{j})\to\mathrm{w}$, with $|\mathrm{w}|=1$.
Therefore similarly to the Part I of the previous proof, we can show
that $\mathrm{w}\in N(U,y)$. Thus we have 
\[
\mu_{j}(y_{j})\to\mu^{*}:=D\varphi(y)-\lambda^{*}\mathrm{w}.
\]
It also follows similarly that $\gamma^{\circ}(\mu^{*})=1$, i.e.
$\mu^{*}\in\partial K^{\circ}$. In addition, we can similarly conclude
that the sequence $\mathrm{v}_{j}:=D\gamma_{j}^{\circ}(\mu_{j}(y_{j}))\in N(K_{j}^{\circ},\mu_{j}(y_{j}))$
converges to a nonzero vector $\mathrm{v}\in N(K^{\circ},\mu^{*})$,
after we pass to one further subsequence.

Now by (\ref{eq: 22 in reg Conv dom}) we obtain 
\begin{equation}
\langle\mathrm{v},\mathrm{w}\rangle=\lim\langle\mathrm{v}_{j},-\nu_{j}(y_{j})\rangle=0.\label{eq: 23 in reg Conv dom}
\end{equation}
But if $\gamma^{\circ}(D\varphi(y))<1$ then we must have $\lambda^{*}>0$.
So $D\varphi=\mu^{*}+\lambda^{*}\mathrm{w}$ belongs to the ray passing
through $\mu^{*}\in\partial K^{\circ}$ in the direction $\mathrm{w}$.
However, we know that $D\varphi$ is in the interior of $K^{\circ}$,
since $\gamma^{\circ}(D\varphi)<1$. Thus the ray $t\mapsto\mu^{*}+t\mathrm{w}$
for $t>0$, passes through the interior of $K^{\circ}$. Therefore
this ray and $K^{\circ}$ must lie on the same side of the supporting
hyperplane $H_{\mu^{*},\mathrm{v}}$. In addition, the ray cannot
lie on the hyperplane, since it intersects the interior of $K^{\circ}$.
Hence we must have $\langle\mathrm{v},\mathrm{w}\rangle<0$, which
contradicts (\ref{eq: 23 in reg Conv dom}). Thus we must have $\gamma^{\circ}(D\varphi(y))=1$,
i.e. $D\varphi\in\partial K^{\circ}$. If $\lambda^{*}=0$ then $\mu^{*}=D\varphi$.
Hence $\mathrm{v}\in N(K^{\circ},D\varphi)$.%
{} Then (\ref{eq: 23 in reg Conv dom}) is in contradiction with our
assumption (\ref{eq: 0 in reg Conv dom}), since $\mathrm{v}\ne0$.
So suppose $\lambda^{*}>0$. Then the ray $t\mapsto\mu^{*}+t\mathrm{w}$
for $t>0$, passes through the two points $D\varphi,\mu^{*}\in\partial K^{\circ}$.
Furthermore, (\ref{eq: 23 in reg Conv dom}) implies that the ray
lies on the supporting hyperplane $H_{\mu^{*},\mathrm{v}}$. Therefore
$D\varphi(y)\in H_{\mu^{*},\mathrm{v}}$. Hence $H_{\mu^{*},\mathrm{v}}$
is also a supporting hyperplane of $K^{\circ}$ at $D\varphi(y)$.
So $\mathrm{v}\in N(K^{\circ},D\varphi)$, and again we arrive at
a contradiction with (\ref{eq: 0 in reg Conv dom}). 

Thus $\langle D\gamma_{k}^{\circ}(\mu_{k}),\nu_{k}\rangle$ must have
a positive lower bound on $\partial U_{k}$ independently of $k$,
as desired. Therefore $D^{2}\rho_{k}$ is bounded on $\partial U_{k}$
independently of $k$; and consequently we have an upper bound for
$D_{i}(D_{i}F(Du_{k}))$ on $P_{k}^{+}$, which is independent of
$k$. Similarly, we can show that $D_{i}(D_{i}F(Du_{k}))$ is bounded
on $P_{k}^{-}$, independently of $k$. Hence we obtain the desired
bound (\ref{eq: 2 in reg Conv dom}).

Now let $V_{l}\subset\overline{V}_{l}\subset U$ be an expanding sequence
of open sets with $C^{2}$ boundaries, such that $U=\bigcup V_{l}$.
Consider the sequence $u_{k}|_{V_{l+2}}$. Similarly to the proof
of Theorem \ref{thm: Reg u}, we can show that for every $p<\infty$
there is $C_{p,l}>0$, which is independent of $k$, such that 
\[
\|u_{k}\|_{W^{2,p}(V_{l})}\le C_{p,l}.
\]
Consequently, as $\partial V_{l}$ is $C^{2}$, for every $\tilde{\alpha}<1$,
$\|u_{k}\|_{C^{1,\tilde{\alpha}}(\overline{V}_{l})}$ is bounded independently
of $k$. Therefore, we can inductively construct subsequences $u_{k_{l}}$
of $u_{k}$, such that $u_{k_{l}}$ is a subsequence of $u_{k_{l-1}}$;
and $u_{k_{l}}$ is weakly convergent in $W^{2,p}(V_{l})$, and strongly
convergent in $C^{1}(\overline{V}_{l})$. Also remember that $u_{k}$
uniformly converges to a continuous function $\tilde{u}\in C^{0}(\overline{U})$
that satisfies $\tilde{u}|_{\partial U}=\varphi$. Thus all the limits
of the subsequences $u_{k_{l}}$ must be equal to $\tilde{u}$. As
a result, $\tilde{u}$ belongs to $W_{\mathrm{loc}}^{2,p}(U)$ for
every $p<\infty$. Furthermore we have $D\tilde{u}\in K^{\circ}$;
because $Du_{k}\in K_{k}^{\circ}$, and thus $Du\in K_{k}^{\circ}$
for every $k$. So we have $\tilde{u}\in W_{K^{\circ},\varphi}(U)$,
since $\tilde{u}|_{\partial U}=\varphi$.

Now we will show that $\tilde{u}$ is the minimizer of $J[\cdot;U]$
over $W_{K^{\circ},\varphi}(U)$. Let $v\in W_{K^{\circ},\varphi}(U)$.
Then for every $k$ we have 
\[
v_{k}:=\begin{cases}
v & \textrm{in }U\\
\varphi & \textrm{in }U_{k}-U
\end{cases}\in W_{K_{k}^{\circ},\varphi}(U_{k}).
\]
(Note that $v_{k}$ is Lipschitz, so it belongs to $H^{1}(U_{k})$.)
Thus we get 
\begin{align*}
J[u_{k};U_{k}] & \le J[v_{k};U_{k}]=\int_{U_{k}}F(Dv_{k})+g(v_{k})\,dx\\
 & =\int_{U}F(Dv)+g(v)\,dx+\int_{U_{k}-U}F(D\varphi)+g(\varphi)\,dx\le J[v;U]+C\mathcal{L}^{n}(U_{k}-U),
\end{align*}
where $C>0$ is an upper bound for $F(D\varphi)+g(\varphi)$ on $\overline{U}_{1}$,
and $\mathcal{L}^{n}(U_{k}-U)$ is the Lebesgue measure of $U_{k}-U$.
But since $u_{k},Du_{k}$ are bounded independently of $k$, by the
Dominated Convergence Theorem we have $J[u_{k};U_{k}]\to J[\tilde{u};U]$,
where the limit is taken through the diagonal subsequence $u_{l_{l}}$,
constructed in the previous paragraph. Also, $\mathcal{L}^{n}(U_{k}-U)\to0$
as $k\to\infty$. Hence $\tilde{u}$ is the minimizer of $J[\cdot;U]$
over $W_{K^{\circ},\varphi}(U)$, and therefore we must have $\tilde{u}=u$.
Thus $u\in W_{\mathrm{loc}}^{2,p}(U)$ for every $p<\infty$. Then
similarly to the proof of Theorem \ref{thm: Reg NonConv dom}, we
can conclude that $u\in W_{\mathrm{loc}}^{2,\infty}(U)$, as desired.%
\end{proof}

\appendix

\section{\label{sec: Local-Optimal-Reg}Local Optimal Regularity}

In this appendix we prove the local optimal regularity for variational
problems with gradient constraints. This result has been used in the
previous section to obtain the global optimal regularity. Most of
the methods employed in this section are classical and well known,
but to the best of author's knowledge the results have not appeared
elsewhere. Especially since the results are about the double obstacle
problem, and there are far fewer works on this problem compared to
the obstacle problem. Nevertheless, we include the proofs here for
completeness.%

First let us state our assumptions. We need a strict version of the
inequality (\ref{eq: phi Lip}), i.e. the Lipschitz property of $\varphi$.
We also need an upper bound on the weak second derivative of $\gamma$.
\begin{assumption}
\label{assu: 3} We assume that 
\begin{enumerate}
\item[\upshape{(a)}] $K\subset\R^{n}$ is a compact convex set whose interior contains
the origin. Furthermore we have 
\begin{equation}
\mathfrak{D}_{h,\xi}^{2}\gamma(x):=\frac{\gamma(x+h\xi)+\gamma(x-h\xi)-2\gamma(x)}{h^{2}}\leq\frac{C_{2}}{\gamma(x)-h},\label{eq: bd D2 g}
\end{equation}
for some $C_{2}>0$, and every nonzero $x,\xi\in\mathbb{R}^{n}$ with
$\gamma(\xi),\gamma(-\xi)\le1$, and every $0<h<\gamma(x)$. 
\item[\upshape{(b)}] $U\subset\R^{n}$ is a bounded open set with Lipschitz boundary.
\item[\upshape{(c)}] $\varphi:\R^{n}\to\R$ is a continuous function, and for all $x\ne y\in\R^{n}$
we have 
\begin{equation}
-\gamma(y-x)<\varphi(x)-\varphi(y)<\gamma(x-y).\label{eq: phi strct Lip}
\end{equation}
\end{enumerate}
\end{assumption}
\begin{rem*}
Note that $\gamma$ satisfies the inequality (\ref{eq: bd D2 g})
if and only if $\bar{\gamma}$ does.
\end{rem*}
\begin{lem}
\label{lem: C2 -> ass 3}The inequality (\ref{eq: bd D2 g}) holds
when $\gamma$ is $C^{2}$ on $\mathbb{R}^{n}-\{0\}$, or equivalently
when $\partial K$ is $C^{2}$.
\end{lem}
\begin{proof}
First note that $\gamma$ is nonzero on the segment $\{x+\tau\xi:-h\leq\tau\leq h\}$.
Because $\gamma(x)>h$ and $\gamma(\xi),\gamma(-\xi)\le1$, so by
the triangle inequality we get 
\begin{equation}
\gamma(x+\tau\xi)\ge\gamma(x)-\gamma(-\tau\xi)=\gamma(x)-|\tau|\gamma(\pm\xi)\ge\gamma(x)-h>0.\label{eq: 1 in bd D2g}
\end{equation}
Thus $\gamma$ is twice differentiable on this segment. Therefore,
we can apply the mean value theorem to the restriction of $\gamma$
and $D_{\xi}\gamma$ to the segment. Hence we get 
\begin{align*}
\mathfrak{D}_{h,\xi}^{2}\gamma(x) & =\frac{\gamma(x+h\xi)-\gamma(x)+\gamma(x-h\xi)-\gamma(x)}{h^{2}}\\
 & =\frac{hD_{\xi}\gamma(x+s\xi)-hD_{\xi}\gamma(x-t\xi)}{h^{2}}=\frac{(s+t)}{h}D_{\xi\xi}^{2}\gamma(x+r\xi)\le2D_{\xi\xi}^{2}\gamma(x+r\xi).
\end{align*}
Here, $0<s,t<h$ and $-t<r<s$; and we used the fact that $D_{\xi\xi}^{2}\gamma\ge0$,
due to the convexity of $\gamma$. Now, let $C_{2}>0$ be the maximum
of the continuous function 
\[
(\zeta,z)\mapsto2D_{\zeta\zeta}^{2}\gamma(z)=2\langle D^{2}\gamma(z)\,\zeta,\zeta\rangle
\]
over the compact set $(K\cap(-K))\times\partial K$. Then by $(-1)$-homogeneity
of $D^{2}\gamma$ we get 
\[
2D_{\xi\xi}^{2}\gamma(x+r\xi)=\frac{2}{\gamma(x+r\xi)}D_{\xi\xi}^{2}\gamma\Big(\frac{x+r\xi}{\gamma(x+r\xi)}\Big)\le\frac{C_{2}}{\gamma(x+r\xi)}\le\frac{C_{2}}{\gamma(x)-h}.
\]
Which is the desired result. Note that in the last inequality above,
we used (\ref{eq: 1 in bd D2g}).
\end{proof}
Remember that for some $C_{1}\ge C_{0}>0$, we have 
\begin{equation}
C_{0}|x|\le\gamma(x)\le C_{1}|x|,\label{eq: bd gama}
\end{equation}
for all $x\in\mathbb{R}^{n}$. Obviously, this inequality also holds
if we replace $\gamma$ with $\bar{\gamma}$.
\begin{lem}
\label{lem: rho + rho bar < 2d}Suppose that the strict Lipschitz
property (\ref{eq: phi strct Lip}) for $\varphi$ holds. Then for
all $x\notin\partial U$ we have 
\[
0<\rho(x)+\bar{\rho}(x)\le2C_{1}d(x),
\]
where $d$ is the Euclidean distance to $\partial U$.
\end{lem}
\begin{rem*}
Note that the above inequality implies that the two obstacles do not
touch inside $U$. Also note that for $y\in\partial U$ we have $\rho(y)+\bar{\rho}(y)=0$,
since we have seen that $\rho(y)=\varphi(y)$ and $\bar{\rho}(y)=-\varphi(y)$.%
\end{rem*}
\begin{proof}
We have $\rho(x)=\gamma(x-y)+\varphi(y)$, and $\bar{\rho}(x)=\gamma(z-x)-\varphi(z)$,
for some $y,z\in\partial U$. Thus when $z=y$ we have $\rho(x)+\bar{\rho}(x)=\gamma(x-y)+\gamma(y-x)>0$,
since $x\ne y$. And when $z\ne y$ we have 
\begin{align*}
\rho(x)+\bar{\rho}(x) & =\gamma(x-y)+\varphi(y)+\gamma(z-x)-\varphi(z)\\
 & \ge\gamma(z-y)+\varphi(y)-\varphi(z)>0,
\end{align*}
due to the triangle inequality for $\gamma$ and (\ref{eq: phi strct Lip}).
On the other hand, for $y\in\partial U$ we have 
\[
\rho(x)+\bar{\rho}(x)\le\gamma(x-y)+\varphi(y)+\gamma(y-x)-\varphi(y)\le2C_{1}|x-y|.
\]
Hence as $y$ is arbitrary we get $\rho(x)+\bar{\rho}(x)\le2C_{1}d(x)$.
\end{proof}
Let $\eta_{\varepsilon}$ be the standard mollifier. Then we define
\begin{align}
 & \psi_{\varepsilon}(x):=(\eta_{\varepsilon}*\rho)(x):=\int_{|y|\leq\varepsilon}\eta_{\varepsilon}(y)\rho(x-y)\,dy,\nonumber \\
 & \phi_{\varepsilon}(x):=-(\eta_{\varepsilon}*\bar{\rho})(x)+\delta_{\varepsilon},\label{eq: psi, phi}
\end{align}
where $3C_{1}\varepsilon<\delta_{\varepsilon}<4C_{1}\varepsilon$
is chosen such that $\partial\{\phi_{\varepsilon}<\psi_{\varepsilon}\}$
is $C^{\infty}$, which is possible by Sard's Theorem. Note that since
$\rho,\bar{\rho}$ are defined on all of $\mathbb{R}^{n}$, $\psi_{\varepsilon},\phi_{\varepsilon}$
are smooth functions on $\mathbb{R}^{n}$. Also 
\begin{align*}
|\psi_{\varepsilon}(x)-\rho(x)| & \leq\int_{|y|\leq\varepsilon}\eta_{\varepsilon}(y)|\rho(x-y)-\rho(x)|\,dy\\
 & \le\int_{|y|\leq\varepsilon}\eta_{\varepsilon}(y)\max\{\gamma(-y),\gamma(y)\}\,dy\leq\int_{|y|\leq\varepsilon}C_{1}|y|\eta_{\varepsilon}(y)\,dy\le C_{1}\varepsilon.
\end{align*}
Notice that we used (\ref{eq: rho Lip}) in the second inequality.
Similarly we have 
\[
2C_{1}\varepsilon<\phi_{\varepsilon}-(-\bar{\rho})<5C_{1}\varepsilon.
\]

Now, let 
\begin{equation}
U_{\varepsilon}:=\{x\in U:\phi_{\varepsilon}(x)<\psi_{\varepsilon}(x)\}.\label{eq: U_e}
\end{equation}
Then we have 
\begin{equation}
\{x\in U:\rho(x)+\bar{\rho}(x)>5C_{1}\varepsilon\}\subset U_{\varepsilon}\subset\{x\in\overline{U}:\phi_{\varepsilon}(x)\leq\psi_{\varepsilon}(x)\}\subset\{x\in U:d(x)>\varepsilon\}.\label{eq: inclusions}
\end{equation}
To see this note that $\phi_{\varepsilon}(x)\leq\psi_{\varepsilon}(x)$
implies that 
\[
3C_{1}\varepsilon<\delta_{\varepsilon}\le(\rho+\bar{\rho})\ast\eta_{\varepsilon}\le\rho+\bar{\rho}+C_{1}\varepsilon\le2C_{1}d(x)+C_{1}\varepsilon.
\]
Hence $d(x)>\varepsilon$. On the other hand, if $\phi_{\varepsilon}(x)\ge\psi_{\varepsilon}(x)$
then 
\[
4C_{1}\varepsilon>\delta_{\varepsilon}\ge(\rho+\bar{\rho})\ast\eta_{\varepsilon}\ge\rho+\bar{\rho}-C_{1}\varepsilon.
\]
Thus $\rho(x)+\bar{\rho}(x)<5C_{1}\varepsilon$. Hence $\rho(x)+\bar{\rho}(x)>5C_{1}\varepsilon$
implies $\phi_{\varepsilon}(x)<\psi_{\varepsilon}(x)$, as desired.
\begin{rem*}
The above inclusions show that $\overline{U}_{\varepsilon}\subset U$,
and 
\begin{equation}
U=\bigcup_{\varepsilon>0}U_{\varepsilon};\label{eq: U=00003DU U_e}
\end{equation}
since by Lemma \ref{lem: rho + rho bar < 2d} we know that $\rho+\bar{\rho}>0$
on $U$. In addition, remember that we have chosen $\delta_{\varepsilon}$
so that $\partial U_{\varepsilon}$ is $C^{\infty}$. Furthermore,
for every $\varepsilon$ there is $\tilde{\varepsilon}$ such that
\begin{equation}
U_{\varepsilon}\subset\{d>\varepsilon\}\subset\{\rho+\bar{\rho}>5C_{1}\tilde{\varepsilon}\}\subset U_{\tilde{\varepsilon}}.\label{eq: U_e subst U_e'}
\end{equation}
Because otherwise for every $j$ there is $x_{j}\in U$ such that
$d(x_{j})>\varepsilon$, while $\rho(x_{j})+\bar{\rho}(x_{j})\le\frac{1}{j}$.
But due to the compactness we can assume that $x_{j}\to x\in\overline{U}$.
Then by continuity we must have $\rho(x)+\bar{\rho}(x)=0$ and $d(x)\ge\varepsilon$.
Now by Lemma \ref{lem: rho + rho bar < 2d}, $\rho(x)+\bar{\rho}(x)=0$
implies that $x\in\partial U$, which contradicts the fact that $d(x)\ge\varepsilon$.
\end{rem*}

\begin{lem}
Suppose the Assumption \ref{assu: 3} holds. Then we have 
\[
D\phi_{\varepsilon},D\psi_{\varepsilon}\in K^{\circ}.
\]
Furthermore, for any unit vector $\xi$, and every $x\in U$ with
$d(x)>\varepsilon$ we have 
\begin{align}
 & D_{\xi\xi}^{2}\psi_{\varepsilon}(x)\leq\frac{C_{3}}{d(x)-\varepsilon},\nonumber \\
 & D_{\xi\xi}^{2}\phi_{\varepsilon}(x)\geq\frac{-C_{3}}{d(x)-\varepsilon},\label{eq: bd D2 phi}
\end{align}
where $C_{3}:=C_{0}^{-1}C_{1}^{2}C_{2}$, and $d$ is the Euclidean
distance to $\partial U$.
\end{lem}
\begin{proof}
To show the first part, note that $\rho,\bar{\rho}$ are Lipschitz
functions and $D\rho,-D\bar{\rho}\in K^{\circ}$ a.e., as shown in
\citep{MR1797872} using the property (\ref{eq: rho Lip}). Then because
of Jensen's inequality, and convexity and homogeneity of $\gamma^{\circ}$,
we have 
\begin{align*}
\gamma^{\circ}(D\psi_{\varepsilon}(x)) & \leq\int_{|y|\leq\varepsilon}\gamma^{\circ}(\eta_{\varepsilon}(y)D\rho(x-y))\,dy\\
 & =\int_{|y|\leq\varepsilon}\eta_{\varepsilon}(y)\gamma^{\circ}(D\rho(x-y))\,dy\leq\int_{|y|\leq\varepsilon}\eta_{\varepsilon}(y)\,dy\;=1.
\end{align*}
The case of $\phi_{\varepsilon}$ is similar.

Next, we assume initially that $\gamma(\xi),\gamma(-\xi)\le1$. Let
$x\in U$, then 
\[
\rho(x)=\gamma(x-y)+\varphi(y)
\]
for some $y\in\partial U$. We also have $\rho(\cdot)\leq\gamma(\cdot-y)+\varphi(y)$.
Thus by (\ref{eq: bd D2 g}) we get 
\begin{align}
\mathfrak{D}_{h,\xi}^{2}\rho(x) & :=\frac{\rho(x+h\xi)+\rho(x-h\xi)-2\rho(x)}{h^{2}}\nonumber \\
 & \le\frac{\gamma(x+h\xi-y)+\varphi(y)+\gamma(x-h\xi-y)+\varphi(y)-2(\gamma(x-y)+\varphi(y))}{h^{2}}\label{eq: bd D2 rho}\\
 & =\mathfrak{D}_{h,\xi}^{2}\gamma(x-y)\le\frac{C_{2}}{\gamma(x-y)-h}\le\frac{C_{2}}{C_{0}|x-y|-h}\le\frac{C_{2}}{C_{0}d(x)-h},\nonumber 
\end{align}
for $0<h<C_{0}d(x)$.

Now suppose $d(x)>C_{0}^{-1}h+\varepsilon$. Then due to the Lipschitz
continuity of $d$, for $|y|\le\varepsilon$ we have 
\[
C_{0}d(x-y)\geq C_{0}d(x)-C_{0}|y|\ge C_{0}d(x)-C_{0}\varepsilon>h.
\]
Hence by (\ref{eq: bd D2 rho}) we get 
\begin{align*}
\mathfrak{D}_{h,\xi}^{2}\psi_{\varepsilon}(x) & =\int_{|y|\le\varepsilon}\eta_{\varepsilon}(y)\mathfrak{D}_{h,\xi}^{2}\rho(x-y)\,dy\\
 & \leq\int_{|y|\le\varepsilon}\eta_{\varepsilon}(y)\frac{C_{2}}{C_{0}d(x-y)-h}\,dy\\
 & \leq\int_{|y|\le\varepsilon}\eta_{\varepsilon}(y)\frac{C_{2}}{C_{0}d(x)-C_{0}\varepsilon-h}\,dy=\frac{C_{2}}{C_{0}d(x)-C_{0}\varepsilon-h}.
\end{align*}
Let $h\rightarrow0^{+}$. Then for $x\in U$ with $d(x)>\varepsilon$
we get 
\[
D_{\xi\xi}^{2}\psi_{\varepsilon}(x)\leq\frac{C_{0}^{-1}C_{2}}{d(x)-\varepsilon}.
\]
Now assume that $|\xi|=1$. Then for $\hat{\xi}:=\frac{1}{C_{1}}\xi$
we have $\gamma(\hat{\xi}),\gamma(-\hat{\xi})\le1$. We can apply
the above inequality to $\hat{\xi}$ to get 
\[
D_{\xi\xi}^{2}\psi_{\varepsilon}(x)=C_{1}^{2}D_{\hat{\xi}\hat{\xi}}^{2}\psi_{\varepsilon}(x)\leq\frac{C_{0}^{-1}C_{1}^{2}C_{2}}{d(x)-\varepsilon}.
\]
The inequality for $\phi_{\varepsilon}$ follows similarly.
\end{proof}
Next, let $u_{\varepsilon}$ be the minimizer of 
\begin{equation}
J_{\varepsilon}[v]:=J[v;U_{\varepsilon}]=\int_{U_{\varepsilon}}F(Dv)+g(v)\,dx\label{eq: u_e}
\end{equation}
over $W_{\phi_{\varepsilon},\psi_{\varepsilon}}:=\{v\in H^{1}(U_{\varepsilon}):\phi_{\varepsilon}\leq v\leq\psi_{\varepsilon}\textrm{ a.e.}\}$.
Take an arbitrary $v$ in this space. Then $u_{\varepsilon}+t(v-u_{\varepsilon})$
is in this space for $0\le t\le1$. Thus 
\[
\frac{d}{dt}\bigg|_{t=0}J_{\varepsilon}[u_{\varepsilon}+t(v-u_{\varepsilon})]\ge0.
\]
By using the bounds (\ref{eq: Bnds}) on $F,g$, we arrive at the
\textit{variational inequality} 
\begin{equation}
\int_{U_{\varepsilon}}D_{i}F(Du_{\varepsilon})D_{i}(v-u_{\varepsilon})+g'(u_{\varepsilon})(v-u_{\varepsilon})\,dx\,\ge0.\label{eq: var ineq u_e}
\end{equation}
For the details see, for example, the proof of Theorem 3.37 in \citep{MR2361288}. 

\textcolor{blue}{}
\begin{lem}
Suppose the Assumptions \ref{assu: 1},\ref{assu: 3} hold. Then
we have 
\[
u_{\varepsilon}\in\underset{p<\infty}{\bigcap}W^{2,p}(U_{\varepsilon})\subset\underset{\tilde{\alpha}<1}{\bigcap}C^{1,\tilde{\alpha}}(\overline{U}_{\varepsilon}).
\]
\end{lem}
\begin{proof}
For $\delta>0$, let $\tilde{\beta}_{\delta}$ be a smooth increasing
convex function on $\mathbb{R}$, that vanishes on $(-\infty,0]$,
and equals $\frac{1}{2\delta}t^{2}$ for $t\ge\delta$. Set $\beta_{\delta}:=\tilde{\beta}_{\delta}'$.
Then $\beta_{\delta}$ is a smooth increasing function that vanishes
on $(-\infty,0]$, and equals $\frac{1}{\delta}t$ for $t\ge\delta$.
We further assume that $\beta_{\delta}$ is convex too. Let $u_{\varepsilon,\delta}$
be the minimizer of 
\begin{equation}
J_{\varepsilon,\delta}[v]:=\int_{U_{\varepsilon}}F(Dv)+g(v)+\tilde{\beta}_{\delta}(\phi_{\varepsilon}-v)+\tilde{\beta}_{\delta}(v-\psi_{\varepsilon})\,dx,\label{eq: I_e,d}
\end{equation}
over $\phi_{\varepsilon}+H_{0}^{1}(U_{\varepsilon})$. By Theorems
3.30, 3.37 in \citep{MR2361288}, $u_{\varepsilon,\delta}$ exists
and is the unique weak solution to the Euler-Lagrange equation 
\begin{align}
-D_{i}(D_{i}F(Du_{\varepsilon,\delta}))+g'(u_{\varepsilon,\delta})-\beta_{\delta}(\phi_{\varepsilon}-u_{\varepsilon,\delta})+\beta_{\delta}(u_{\varepsilon,\delta}-\psi_{\varepsilon})=0,\nonumber \\
u_{\varepsilon,\delta}=\phi_{\varepsilon}\textrm{ on }\partial U_{\varepsilon}.\label{eq: E-L u-e,d}
\end{align}

As proved in \citep{MR749677}, $u_{\varepsilon,\delta}\in C^{1,\alpha_{0}}(\overline{U}_{\varepsilon})$
for some $\alpha_{0}>0$. On the other hand, as shown in Chapter 2
of \citep{MR717034}, by using the difference quotient technique we
get $u_{\varepsilon,\delta}\in H_{\textrm{loc}}^{2}(U_{\varepsilon})$.
Hence we have 
\[
-a_{ij,\delta}(x)D_{ij}^{2}u_{\varepsilon,\delta}(x)=b_{\delta}(x),
\]
for a.e. $x\in U_{\varepsilon}$. Where $a_{ij,\delta}(x):=D_{ij}^{2}F(Du_{\varepsilon,\delta}(x))$,
and 
\[
b_{\delta}:=-g'(u_{\varepsilon,\delta})+\beta_{\delta}(\phi_{\varepsilon}-u_{\varepsilon,\delta})-\beta_{\delta}(u_{\varepsilon,\delta}-\psi_{\varepsilon}).
\]
Note that $a_{ij,\delta}\in C^{0,\alpha_{1}}(\overline{U}_{\varepsilon})$,
$b_{\delta}\in C^{1,\alpha_{1}}(\overline{U}_{\varepsilon})$, where
$\alpha_{1}=\min\{\bar{\alpha},\alpha_{0}\}$. Thus by using Schauder
estimates (see Theorem 6.14 of \citep{MR1814364}), we deduce that
$u_{\varepsilon,\delta}\in C^{2,\alpha_{1}}(\overline{U}_{\varepsilon})$.

We can easily show that $u_{\varepsilon,\delta}$ is uniformly bounded,
independently of $\delta$. Suppose $\delta\le\min\{1,\frac{1}{4c_{5}}\}$,
and $C^{+}\ge1+2\,\max_{x\in\overline{U}_{\varepsilon}}|\psi_{\varepsilon}(x)|$.
Then by the comparison principle (Theorem 10.1 of \citep{MR1814364})
to show that $u_{\varepsilon,\delta}\le C^{+}$, it is enough to show
that the constant function whose value is $C^{+}$ satisfies 
\[
-a_{ij}D_{ij}^{2}C^{+}+g'(C^{+})-\beta_{\delta}(\phi_{\varepsilon}-C^{+})+\beta_{\delta}(C^{+}-\psi_{\varepsilon})\ge0
\]
But this expression equals 
\begin{align*}
g'(C^{+})+\beta_{\delta}(C^{+}-\psi_{\varepsilon}) & \ge-c_{5}(C^{+}+1)+\frac{1}{\delta}(C^{+}-\psi_{\varepsilon})\\
 & \ge(\frac{1}{2\delta}-c_{5})C^{+}-c_{5}+\frac{1}{\delta}(\frac{1}{2}C^{+}-\psi_{\varepsilon})\ge c_{5}C^{+}-c_{5}\ge0.
\end{align*}
Similarly we can obtain a uniform lower bound for $u_{\varepsilon,\delta}$.

Now, add $D_{i}(D_{i}F(D\psi_{\varepsilon}))$ to the both sides of
(\ref{eq: E-L u-e,d}), and multiply the result by $(\beta_{\delta}(u_{\varepsilon,\delta}-\psi_{\varepsilon}))^{p-1}$
for some $p>2$, and integrate over $U_{\varepsilon}$ to obtain 
\begin{align}
\int_{U_{\varepsilon}}[-D_{i}(D_{i}F(Du_{\varepsilon,\delta}))+D_{i}(D_{i}F(D\psi_{\varepsilon}))](\beta_{\delta}(u_{\varepsilon,\delta}-\psi_{\varepsilon}))^{p-1}\,dx+\int_{U_{\varepsilon}}(\beta_{\delta}(u_{\varepsilon,\delta}-\psi_{\varepsilon}))^{p}\,dx\,\nonumber \\
=\int_{U_{\varepsilon}}[D_{i}(D_{i}F(D\psi_{\varepsilon}))-g'(u_{\varepsilon,\delta})](\beta_{\delta}(u_{\varepsilon,\delta}-\psi_{\varepsilon}))^{p-1}\,dx.\label{eq: b^p}
\end{align}
Note that $\beta_{\delta}(\phi_{\varepsilon}-u_{\varepsilon,\delta})\beta_{\delta}(u_{\varepsilon,\delta}-\psi_{\varepsilon})=0$.
After integration by parts, the first term becomes 
\[
(p-1)\int_{U_{\varepsilon}}[D_{i}F(Du_{\varepsilon,\delta})-D_{i}F(D\psi_{\varepsilon})][D_{i}u_{\varepsilon,\delta}-D_{i}\psi_{\varepsilon}]\beta_{\delta}'(u_{\varepsilon,\delta}-\psi_{\varepsilon})(\beta_{\delta}(u_{\varepsilon,\delta}-\psi_{\varepsilon}))^{p-2}\,dx\ge0.
\]
Note that we used the facts that $F$ is convex, and $u_{\varepsilon,\delta}-\psi_{\varepsilon}$
vanishes on $\partial U_{\varepsilon}$. By employing this inequality
in (\ref{eq: b^p}) we get 
\begin{align*}
\int_{U_{\varepsilon}}(\beta_{\delta}(u_{\varepsilon,\delta}-\psi_{\varepsilon}))^{p}\,dx & \le\int_{U_{\varepsilon}}[D_{i}(D_{i}F(D\psi_{\varepsilon}))-g'(u_{\varepsilon,\delta})](\beta_{\delta}(u_{\varepsilon,\delta}-\psi_{\varepsilon}))^{p-1}\,dx\\
 & \le C_{\varepsilon}\int_{U_{\varepsilon}}(\beta_{\delta}(u_{\varepsilon,\delta}-\psi_{\varepsilon}))^{p-1}\,dx\\
 & \le C_{\varepsilon}|U|^{\frac{1}{p}}\Big(\int_{U_{\varepsilon}}(\beta_{\delta}(u_{\varepsilon,\delta}-\psi_{\varepsilon}))^{p}\,dx\Big)^{\frac{p-1}{p}}.
\end{align*}
Here $C_{\varepsilon}$ is a constant independent of $\delta$, and
$|U|$ is the Lebesgue measure of $U$. Also in the second line we
used the uniform boundedness of $u_{\varepsilon,\delta}$, and in
the last line we used Holder's inequality. Thus we have 
\[
\|\beta_{\delta}(u_{\varepsilon,\delta}-\psi_{\varepsilon})\|_{L^{p}(U_{\varepsilon})}\le C_{\varepsilon}|U|^{\frac{1}{p}}.
\]

By sending $p\to\infty$ we get 
\[
\|\beta_{\delta}(u_{\varepsilon,\delta}-\psi_{\varepsilon})\|_{L^{\infty}(U_{\varepsilon})}\le C_{\varepsilon}.
\]
Similarly we obtain $\|\beta_{\delta}(\phi_{\varepsilon}-u_{\varepsilon,\delta})\|_{L^{\infty}(U_{\varepsilon})}\le C_{\varepsilon}$.
Consequently we have 
\begin{eqnarray}
u_{\varepsilon,\delta}-\psi_{\varepsilon}\le\delta(C_{\varepsilon}+1), &  & \phi_{\varepsilon}-u_{\varepsilon,\delta}\le\delta(C_{\varepsilon}+1).\label{eq: 1 in Reg u_e}
\end{eqnarray}
Utilizing these bounds, and the fact that $u_{\varepsilon,\delta}$
is uniformly bounded, in equation (\ref{eq: E-L u-e,d}), gives us
\[
\|D_{i}(D_{i}F(Du_{\varepsilon,\delta}))\|_{L^{\infty}(U_{\varepsilon})}\le C,
\]
for some $C$ independent of $\delta$. Equivalently we have the quasilinear
elliptic equation 
\[
-D_{i}(D_{i}F(Du_{\varepsilon,\delta}))=b_{\delta}(x),
\]
and $\|b_{\delta}\|_{L^{\infty}(U_{\varepsilon})}\le C$. Then Theorem
15.9 of \citep{MR1814364} implies that $\|Du_{\varepsilon,\delta}\|_{C^{0}(\overline{U}_{\varepsilon})}\le C$,
for some $C$ independent of $\delta$. Thus by Theorem 13.2 of \citep{MR1814364}
we have $\|u_{\varepsilon,\delta}\|_{C^{1,\alpha_{2}}(\overline{U}_{\varepsilon})}\le C$,
for some $C,\alpha_{2}>0$ independent of $\delta$.

Now we have 
\[
\|a_{ij,\delta}D_{ij}^{2}u_{\varepsilon,\delta}\|_{L^{\infty}(U_{\varepsilon})}=\|D_{i}(D_{i}F(Du_{\varepsilon,\delta}))\|_{L^{\infty}(U_{\varepsilon})}\le C,
\]
where $a_{ij,\delta}=D_{ij}^{2}F(Du_{\varepsilon,\delta})$. Then
by Theorem 9.13 of \citep{MR1814364} we have 
\[
\|u_{\varepsilon,\delta}\|_{W^{2,p}(U_{\varepsilon})}\le C_{p},
\]
for all $p<\infty$, and some $C_{p}$ independent of $\delta$. Here
we used the fact that $a_{ij,\delta}$'s have a uniform modulus of
continuity independently of $\delta$, due to the uniform boundedness
of the $C^{\alpha_{2}}$ norm of $Du_{\varepsilon,\delta}$. Therefore
there is a sequence $\delta_{i}\to0$ such that $u_{\varepsilon,\delta_{i}}$
weakly converges in $W^{2,p}(U_{\varepsilon})$ to a function $\tilde{u}_{\varepsilon}$.
In addition, we can assume that $u_{\varepsilon,\delta_{i}},Du_{\varepsilon,\delta_{i}}$
uniformly converge to $\tilde{u}_{\varepsilon},D\tilde{u}_{\varepsilon}$,
since $\|u_{\varepsilon,\delta}\|_{C^{1,\alpha_{2}}(\overline{U}_{\varepsilon})}$
is bounded independently of $\delta$.

Finally, we want to show that $\tilde{u}_{\varepsilon}=u_{\varepsilon}$.
Note that by (\ref{eq: 1 in Reg u_e}) we have $\phi_{\varepsilon}\leq\tilde{u}_{\varepsilon}\leq\psi_{\varepsilon}$.
Hence, it suffices to show that $\tilde{u}_{\varepsilon}$ is the
minimizer of $J_{\varepsilon}$ over $W_{\phi_{\varepsilon},\psi_{\varepsilon}}$.
Take $v\in W_{\phi_{\varepsilon},\psi_{\varepsilon}}\subset\phi_{\varepsilon}+H_{0}^{1}(U_{\varepsilon})$.
Then we have 
\[
J_{\varepsilon}[u_{\varepsilon,\delta_{i}}]\le J_{\varepsilon,\delta_{i}}[u_{\varepsilon,\delta_{i}}]\le J_{\varepsilon,\delta_{i}}[v]=J_{\varepsilon}[v].
\]
Note that the extra terms in $J_{\varepsilon,\delta}$ (defined in
(\ref{eq: I_e,d})) vanish for this $v$, since $\phi_{\varepsilon}\le v\le\psi_{\varepsilon}$.
Now sending $i\to\infty$ gives the desired due to the uniform convergence
of $u_{\varepsilon,\delta_{i}},Du_{\varepsilon,\delta_{i}}$ to $\tilde{u}_{\varepsilon},D\tilde{u}_{\varepsilon}$.
\end{proof}
Since $u_{\varepsilon}\in H^{2}(U_{\varepsilon})$, we can integrate
by parts in (\ref{eq: var ineq u_e}), and use appropriate test functions
in place of $v$, to obtain 
\begin{equation}
\begin{cases}
-D_{i}(D_{i}F(Du_{\varepsilon}))+g'(u_{\varepsilon})=0 & \textrm{ if }\phi_{\varepsilon}<u_{\varepsilon}<\psi_{\varepsilon},\\
-D_{i}(D_{i}F(Du_{\varepsilon}))+g'(u_{\varepsilon})\le0 & \textrm{ a.e. if }\phi_{\varepsilon}<u_{\varepsilon}\le\psi_{\varepsilon},\\
-D_{i}(D_{i}F(Du_{\varepsilon}))+g'(u_{\varepsilon})\ge0 & \textrm{ a.e. if }\phi_{\varepsilon}\le u_{\varepsilon}<\psi_{\varepsilon}.
\end{cases}\label{eq: diff ineq u_e}
\end{equation}
Note that $u_{\varepsilon}$ is $C^{2,\bar{\alpha}}$ on the open
set 
\begin{equation}
E_{\varepsilon}:=\{x\in U_{\varepsilon}:\phi_{\varepsilon}(x)<u_{\varepsilon}(x)<\psi_{\varepsilon}(x)\},\label{eq: E_e}
\end{equation}
due to the Schauder estimates (see Theorem 6.13 of \citep{MR1814364}).
\begin{lem}
Suppose the Assumptions \ref{assu: 1},\ref{assu: 3} hold. Then
we have 
\[
Du_{\varepsilon}\in K^{\circ}\qquad\textrm{ in }U_{\varepsilon}.
\]
\end{lem}
\begin{proof}
First note that $Du_{\varepsilon}$ is continuous on $\overline{U}_{\varepsilon}$.
Now since $u_{\varepsilon}=\phi_{\varepsilon}=\psi_{\varepsilon}$
on $\partial U_{\varepsilon}$, we have $D_{\xi}u_{\varepsilon}=D_{\xi}\phi_{\varepsilon}=D_{\xi}\psi_{\varepsilon}$
for any direction $\xi$ tangent to $\partial U_{\varepsilon}$. Also
as $\phi_{\varepsilon}\leq u_{\varepsilon}\leq\psi_{\varepsilon}$
in $U_{\varepsilon}$, we have $D_{\nu}\phi_{\varepsilon}\leq D_{\nu}u_{\varepsilon}\leq D_{\nu}\psi_{\varepsilon}$
on $\partial U_{\varepsilon}$, where $\nu$ is the inward normal
to $\partial U_{\varepsilon}$. Hence by (\ref{eq: gen Cauchy-Schwartz 2}),
and the fact that $D\phi_{\varepsilon},D\psi_{\varepsilon}\in K^{\circ}$,
we get 
\[
\gamma^{\circ}(Du_{\varepsilon})\le1\qquad\textrm{ on }\partial U_{\varepsilon}.
\]
This bound also holds on the sets $\{u_{\varepsilon}=\psi_{\varepsilon}\}$
and $\{u_{\varepsilon}=\phi_{\varepsilon}\}$, as either $\psi_{\varepsilon}-u_{\varepsilon}$
or $u_{\varepsilon}-\phi_{\varepsilon}$ attains its minimum there,
so $Du_{\varepsilon}$ equals $D\psi_{\varepsilon}$ or $D\phi_{\varepsilon}$
over them. 

To obtain the bound on the open set $E_{\varepsilon}$, note that
for any vector $\xi$ with $\gamma(\xi)=1$, $D_{\xi}u_{\varepsilon}$
is a weak solution to the elliptic equation 
\[
-D_{i}(a_{ij}D_{j}D_{\xi}u_{\varepsilon})+bD_{\xi}u_{\varepsilon}=0\qquad\textrm{ in }E_{\varepsilon},
\]
where $a_{ij}:=D_{ij}^{2}F(Du_{\varepsilon})$, and $b:=g''(u_{\varepsilon})$.
Now suppose that $D_{\xi}u_{\varepsilon}$ attains its maximum at
$x_{0}\in E_{\varepsilon}$ with $D_{\xi}u_{\varepsilon}(x_{0})>1$.
Then the strong maximum principle (Theorem 8.19 of \citep{MR1814364})
implies that $D_{\xi}u_{\varepsilon}$ is constant over $E_{\varepsilon}$.
This contradicts the fact that $D_{\xi}u_{\varepsilon}\le1$ on $\partial E_{\varepsilon}$.
Thus we must have $D_{\xi}u_{\varepsilon}\le1$ on $E_{\varepsilon}$;
and as $\xi$ is arbitrary, we get the desired bound using (\ref{eq: gen Cauchy-Schwartz 2}).
\end{proof}
\begin{lem}
Suppose the Assumptions \ref{assu: 1},\ref{assu: 3} hold. Then for
small enough $\varepsilon$ we have 
\begin{align}
 & |D_{i}(D_{i}F(Du_{\varepsilon}))|\le C_{4}+\frac{nc_{9}C_{3}}{d-\varepsilon} &  & \hspace{-1cm}\textrm{a.e. on }U_{\varepsilon},\nonumber \\
 & |D^{2}\psi_{\varepsilon}|\le\frac{1}{c_{8}}\Big(C_{4}+\frac{nc_{9}C_{3}}{d-\varepsilon}\Big) &  & \hspace{-1cm}\textrm{a.e. on }\{u_{\varepsilon}=\psi_{\varepsilon}\},\label{eq: bd D2 u_e}\\
 & |D^{2}\phi_{\varepsilon}|\le\frac{1}{c_{8}}\Big(C_{4}+\frac{nc_{9}C_{3}}{d-\varepsilon}\Big) &  & \hspace{-1cm}\textrm{a.e. on }\{u_{\varepsilon}=\phi_{\varepsilon}\},\nonumber 
\end{align}
where $d$ is the Euclidean distance to $\partial U$, and $C_{4}:=c_{5}(\max_{\overline{U}}\{|\rho|,|\bar{\rho}|\}+5C_{1}+1)$.
\end{lem}
\begin{rem*}
Note that for a function $f$ 
\[
|D^{2}f|=\max_{|\xi|=1}|D_{\xi\xi}^{2}f|=\max\{|\lambda_{i}|:\lambda_{i}\textrm{ is an eigenvalue of }D^{2}f\}.
\]
\end{rem*}
\begin{proof}
On the open set $E_{\varepsilon}\subset U_{\varepsilon}$ we have
\begin{align*}
|D_{i}(D_{i}F(Du_{\varepsilon}))| & =|g'(u_{\varepsilon})|\le c_{5}(|u_{\varepsilon}|+1)\\
 & \le c_{5}(\max_{\overline{U}}\{|\phi_{\varepsilon}|,|\psi_{\varepsilon}|\}+1)\le c_{5}(\max_{\overline{U}}\{|\rho|,|\bar{\rho}|\}+5C_{1}\varepsilon+1)\le C_{4}.
\end{align*}
Next consider the closed subset of $U_{\varepsilon}$ over which
$u_{\varepsilon}=\psi_{\varepsilon}$. By (\ref{eq: diff ineq u_e})
we have 
\[
D_{i}(D_{i}F(Du_{\varepsilon}))\ge g'(u_{\varepsilon})\ge-C_{4}\qquad\textrm{ a.e. on }\{u_{\varepsilon}=\psi_{\varepsilon}\}.
\]
Since both $u_{\varepsilon},\psi_{\varepsilon}$ are twice weakly
differentiable, we have (see Theorem 4.4 of \citep{MR3409135}) 
\[
D_{i}(D_{i}F(Du_{\varepsilon}))=D_{i}(D_{i}F(D\psi_{\varepsilon}))\qquad\textrm{ a.e. on }\{u_{\varepsilon}=\psi_{\varepsilon}\}.
\]
But we have 
\begin{align*}
D_{i}(D_{i}F(D\psi_{\varepsilon}))=D_{ij}^{2}F(D\psi_{\varepsilon})D_{ij}^{2}\psi_{\varepsilon} & =\mathrm{tr}[D^{2}F(D\psi_{\varepsilon})D^{2}\psi_{\varepsilon}]\\
 & =\underset{i\le n}{\sum}D_{\xi_{i}\xi_{i}}^{2}F(D\psi_{\varepsilon})D_{\xi_{i}\xi_{i}}^{2}\psi_{\varepsilon},
\end{align*}
where $\xi_{1},\cdots,\xi_{n}$ is an orthonormal basis of eigenvectors
of $D^{2}\psi_{\varepsilon}$. Thus, by using (\ref{eq: Bnds}),
(\ref{eq: bd D2 phi}) we get 
\[
D_{i}(D_{i}F(Du_{\varepsilon}))\le\underset{i\le n}{\sum}D_{\xi_{i}\xi_{i}}^{2}F(D\psi_{\varepsilon})\frac{C_{3}}{d(x)-\varepsilon}\le\frac{nc_{9}C_{3}}{d(x)-\varepsilon}\qquad\textrm{ a.e. on }\{u_{\varepsilon}=\psi_{\varepsilon}\}.
\]
We have similar bounds on the set $\{u_{\varepsilon}=\phi_{\varepsilon}\}$.
These bounds easily give the first inequality of (\ref{eq: bd D2 u_e}).

On the other hand for a.e. $x\in\{u_{\varepsilon}=\psi_{\varepsilon}\}$
we have 
\[
\underset{i\le n}{\sum}D_{\xi_{i}\xi_{i}}^{2}F(D\psi_{\varepsilon})D_{\xi_{i}\xi_{i}}^{2}\psi_{\varepsilon}=D_{i}(D_{i}F(D\psi_{\varepsilon}))=D_{i}(D_{i}F(Du_{\varepsilon}))\ge g'(u_{\varepsilon})\ge-C_{4}.
\]
Thus again by using (\ref{eq: Bnds}), (\ref{eq: bd D2 phi}) we get
\[
D_{\xi_{j}\xi_{j}}^{2}F(D\psi_{\varepsilon})D_{\xi_{j}\xi_{j}}^{2}\psi_{\varepsilon}\ge-C_{4}-\underset{i\ne j}{\sum}D_{\xi_{i}\xi_{i}}^{2}F(D\psi_{\varepsilon})D_{\xi_{i}\xi_{i}}^{2}\psi_{\varepsilon}\ge-C_{4}-\frac{(n-1)c_{9}C_{3}}{d(x)-\varepsilon}.
\]
Hence 
\[
D_{\xi_{j}\xi_{j}}^{2}\psi_{\varepsilon}\ge-\frac{C_{4}}{c_{8}}-\frac{(n-1)\frac{c_{9}}{c_{8}}C_{3}}{d(x)-\varepsilon}.
\]
The reverse inequality is given by (\ref{eq: bd D2 phi}) (Keep in
mind that $\frac{c_{9}}{c_{8}}\ge1$). Note that the numbers $D_{\xi_{j}\xi_{j}}^{2}\psi_{\varepsilon}$
are the eigenvalues of $D^{2}\psi_{\varepsilon}$, so we get the desired
bound. The case of $D^{2}\phi_{\varepsilon}$ is similar.
\end{proof}
\begin{thm}
\label{thm: Reg u}Suppose the Assumptions \ref{assu: 1},\ref{assu: 3}
hold. Then we have 
\[
u\in W_{\mathrm{loc}}^{2,\infty}(U)=C_{\mathrm{loc}}^{1,1}(U).
\]
\end{thm}
\begin{proof}
We choose a decreasing sequence $\varepsilon_{k}\to0$ such that $\overline{U}_{\varepsilon_{k}}\subset U_{\varepsilon_{k+1}}$
(this is possible by (\ref{eq: U_e subst U_e'})). For convenience
we use $U_{k},u_{k},\phi_{k},\psi_{k}$ instead of $U_{\varepsilon_{k}},u_{\varepsilon_{k}},\phi_{\varepsilon_{k}},\psi_{\varepsilon_{k}}$.
Consider the sequence $u_{k}|_{U_{3}}$ for $k>3$. By (\ref{eq: bd D2 u_e}),
(\ref{eq: inclusions}) we have 
\[
\|D_{i}(D_{i}F(Du_{k}))\|_{L^{\infty}(U_{3})}\le C,
\]
for some $C$ independent of $k$. Let $g_{k}:=D_{i}(D_{i}F(Du_{k}))$.
Then $Du_{k}$ is a weak solution to the elliptic equation 
\[
-D_{i}(a_{ij,k}D_{j}Du_{k})+Dg_{k}=0,
\]
where $a_{ij,k}:=D_{ij}^{2}F(Du_{k})$. Thus by Theorem 8.24 of \citep{MR1814364}
we have 
\[
\|Du_{k}\|_{C^{\alpha_{0}}(\overline{U}_{2})}\le C,
\]
for some $C,\alpha_{0}>0$ independent of $k$. Here we used the fact
that $Du_{k},g_{k},a_{ij,k}$ are uniformly bounded independently
of $k$ (Remember that $Du_{k}\in K^{\circ}$).

Now we have 
\[
\|a_{ij,k}D_{ij}^{2}u_{k}\|_{L^{\infty}(U_{2})}=\|D_{i}(D_{i}F(Du_{k}))\|_{L^{\infty}(U_{2})}\le C.
\]
Then by Theorem 9.11 of \citep{MR1814364} we have 
\begin{equation}
\|u_{k}\|_{W^{2,p}(U_{1})}\le C_{p},\label{eq: u_k in W2,p}
\end{equation}
for all $p<\infty$, and some $C_{p}$ independent of $k$. Here we
used the fact that $a_{ij,k}$'s have a uniform modulus of continuity
independently of $k$, due to the uniform boundedness of the $C^{\alpha_{0}}$
norm of $Du_{k}$. Consequently, as $\partial U_{1}$ is smooth, for
every $\tilde{\alpha}<1$, $\|u_{k}\|_{C^{1,\tilde{\alpha}}(\overline{U}_{1})}$
is bounded independently of $k$.

Therefore there is a subsequence of $u_{k}$'s, which we denote by
$u_{k_{1}}$, that weakly converges in $W^{2,p}(U_{1})$ to a function
$\tilde{u}_{1}$. In addition, we can assume that $u_{k_{1}},Du_{k_{1}}$
uniformly converge to $\tilde{u}_{1},D\tilde{u}_{1}$. Now we can
repeat this process with $u_{k_{1}}|_{U_{4}}$ and get a function
$\tilde{u}_{2}$ in $W^{2,p}(U_{2})$, which agrees with $\tilde{u}_{1}$
on $U_{1}$. Continuing this way with subsequences $u_{k_{l}}$ for
each positive integer $l$, we can finally construct a $C^{1}$ function
$\tilde{u}$ in $W_{\textrm{loc}}^{2,p}(U)$ (note that $U=\bigcup U_{k}$
by (\ref{eq: U=00003DU U_e})). It is obvious that $D\tilde{u}\in K^{\circ}$
and $-\bar{\rho}\le\tilde{u}\le\rho$, since $Du_{k}\in K^{\circ}$
and $\phi_{k}\le u_{k}\le\psi_{k}$ for every $k$. In particular
we have $\tilde{u}\in W_{K^{\circ},\varphi}\subset W_{\bar{\rho},\rho}$.

Now we want to show that $u=\tilde{u}$. Due to the uniqueness of
the minimizer, it is enough to show that $\tilde{u}$ is the minimizer
of $J$ over $W_{K^{\circ},\varphi}(U)$. As it is well known, it
suffices to show that (see, for example, the proof of Theorem 3.37
in \citep{MR2361288}) 
\begin{equation}
\int_{U}D_{i}F(D\tilde{u})D_{i}(v-\tilde{u})+g'(\tilde{u})(v-\tilde{u})\,dx\ge0,\label{eq: 1 in Reg u}
\end{equation}
for every $v\in W_{K^{\circ},\varphi}(U)$. Note that $v$ is Lipschitz
continuous. First suppose that $v>-\bar{\rho}$ on $U$, and $v=\rho$
on $\{x\in U:\rho(x)+\bar{\rho}(x)\le\delta\}$ for some $\delta>0$.
Let $v_{k}:=\eta_{\varepsilon_{k}}\ast v$ be the mollification of
$v$. Then for large enough $k$ we have $\phi_{k}\le v_{k}\le\psi_{k}$
on $U_{k}$. Because for large enough $k$ we have $\phi_{k}\le v_{k}=\psi_{k}$
on $\overline{U}_{k}\cap\{\rho+\bar{\rho}\le\delta/2\}$. (Note that
for $x\in U\cap\{\rho+\bar{\rho}\le\delta/2\}$ we have $B_{\varepsilon_{k}}(x)\subset U\cap\{\rho+\bar{\rho}\le\delta\}$
for large enough $k$, due to the Lipschitz continuity of $\rho,\bar{\rho}$.)
On the other hand, $v-(-\bar{\rho})$ has a positive minimum on the
compact set $\{\rho+\bar{\rho}\ge\delta/4\}\cap U$, so by (\ref{eq: psi, phi})
we have $v_{k}-\phi_{k}=(v-(-\bar{\rho}))\ast\eta_{\varepsilon_{k}}-\delta_{\varepsilon_{k}}\ge0$
on $\{\rho+\bar{\rho}\ge\delta/2\}\cap U$, for large enough $k$.
Hence by (\ref{eq: var ineq u_e}) we must have 
\[
\int_{U_{k}}D_{i}F(Du_{k})D_{i}(v_{k}-u_{k})+g'(u_{k})(v_{k}-u_{k})\,dx\ge0.
\]
By taking the limit through the diagonal sequence $u_{l_{l}}$, and
using the Dominated Convergence Theorem, we get (\ref{eq: 1 in Reg u})
for this special $v$.

It is easy to see that an arbitrary test function $v$ in $W_{K^{\circ},\varphi}$
can be approximated by such special test functions. Just consider
the functions $v_{\delta}:=\min\{v+\delta,\rho\}$. Then we have $v_{\delta}>-\bar{\rho}$
on $U$, since $v\ge-\bar{\rho}$, and $\rho>-\bar{\rho}$ on $U$.
Also on $\{\rho+\bar{\rho}\le\delta\}\cap U$ we have $\rho\le-\bar{\rho}+\delta\le v+\delta$,
so $v_{\delta}=\rho$ there. It is also easy to see that $v_{\delta}\in W_{K^{\circ},\varphi}$,
and $v_{\delta}\to v$ in $H^{1}(U)$. Therefore we get (\ref{eq: 1 in Reg u})
for all $v\in W_{K^{\circ},\varphi}$, as desired.

It remains to show that $u$ belongs to $W_{\textrm{loc}}^{2,\infty}(U)$.
First note that $D^{2}u_{k}=D^{2}\phi_{k}$ a.e. on $\{u_{k}=\phi_{k}\}$,
hence by (\ref{eq: bd D2 u_e}) $D^{2}u_{k}$ is bounded there independently
of $k$. Similarly, $D^{2}u_{k}$ is bounded on $\{u_{k}=\psi_{k}\}$
independently of $k$. Now take $x_{0}\in U$ and suppose that $B_{r}(x_{0})\subset U$.
Let $l$ be large enough so that $B_{r}(x_{0})\subset U_{l}$. Set
$v_{k}(y):=u_{k}(x_{0}+ry)$ for $y\in B_{1}(0)$, and $k\ge l$.
Then by (\ref{eq: diff ineq u_e}) and the above argument we have
\[
\begin{cases}
D_{ij}^{2}F(\frac{1}{r}Dv_{k})D_{ij}^{2}v_{k}=r^{2}g'(v_{k}) & \textrm{ a.e. in }B_{1}(0)\cap\Omega_{k},\\
|D^{2}v_{k}|\le C & \textrm{ a.e. in }B_{1}(0)-\Omega_{k},
\end{cases}
\]
for some $C$ independent of $k$. Here $\Omega_{k}:=\{y\in B_{1}(0):u_{k}(x_{0}+ry)\in E_{\varepsilon_{k}}\}$.

Now recall that $\|u_{k}\|_{W^{2,n}(B_{r}(x_{0}))}$, $\|g'(u_{k})\|_{L^{\infty}(B_{r}(x_{0}))}$
are bounded independently of $k$, due to (\ref{eq: u_k in W2,p}),
and the fact that $\phi_{k}\le u_{k}\le\psi_{k}$ and $\phi_{k},\psi_{k}$
are bounded independently of $k$. Therefore $\|v_{k}\|_{W^{2,n}(B_{1}(0))}$
and $\|g'(v_{k})\|_{L^{\infty}(B_{1}(0))}$ are bounded independently
of $k$ too. Also note that the Holder norms of $D_{ij}^{2}F(\frac{1}{r}Dv_{k}),r^{2}g'(v_{k})$
are bounded independently of $k$, since for every $\tilde{\alpha}<1$,
$\|u_{k}\|_{C^{1,\tilde{\alpha}}(\overline{U}_{l})}$ is bounded independently
of $k$. Thus we can apply the result of \citep{Indrei-Minne} to
deduce that 
\[
|D^{2}v_{k}|\le\bar{C}\qquad\textrm{ a.e. in }B_{\frac{1}{2}}(0),
\]
for some $\bar{C}$ independent of $k$. Therefore 
\[
|D^{2}u_{k}|\le C\qquad\textrm{ a.e. in }B_{\frac{r}{2}}(x_{0}),
\]
for some $C$ independent of $k$. Hence, $u_{k}$ is a bounded sequence
in $W^{2,\infty}(B_{\frac{r}{2}}(x_{0}))$. Consider the diagonal
subsequence $u_{l_{l}}$. Then a subsequence of it converges weakly
star in $W^{2,\infty}(B_{\frac{r}{2}}(x_{0}))$. But the limit must
be $u$; so we get $u\in W^{2,\infty}(B_{\frac{r}{2}}(x_{0}))$, as
desired.
\end{proof}

\section{\label{sec: Ridge, Elast, Plast}The ridge, and the elastic and plastic
regions}

Remember that we say $y\in\partial U$ is a $\rho$-closest point
to $x$ if $\rho(x)=\gamma(x-y)+\varphi(y)$. Similarly, we say $y\in\partial U$
is a $\bar{\rho}$-closest point to $x$ if $\bar{\rho}(x)=\gamma(y-x)-\varphi(y)$.
Note that an obvious consequence of the strict Lipschitz property
(\ref{eq: phi strct Lip}) for $\varphi$ is that every $y\in\partial U$
is the unique $\rho$-closest point and $\bar{\rho}$-closest point
on $\partial U$ to itself.
\begin{lem}
\label{lem: segment to the closest pt}Suppose $y$ is one of the
$\rho$-closest points on $\partial U$ to $x\in U$. Then 
\begin{enumerate}
\item[\upshape{(a)}] $y$ is a $\rho$-closest point on $\partial U$ to every point of
$]x,y[$. Therefore $\rho$ varies linearly along the line segment
$[x,y]$.
\item[\upshape{(b)}] If in addition the strict Lipschitz property (\ref{eq: phi strct Lip})
for $\varphi$ holds, then we have $]x,y[\subset U$.
\item[\upshape{(c)}] If in addition $\gamma$ is strictly convex, and the strict Lipschitz
property (\ref{eq: phi strct Lip}) for $\varphi$ holds, then $y$
is the unique $\rho$-closest point on $\partial U$ to the points
of $]x,y[$.
\end{enumerate}
\end{lem}
\begin{proof}
(a) Let $z\in]x,y[$. Suppose to the contrary that there is $\tilde{y}\in\partial U-\{y\}$
such that 
\[
\gamma(z-\tilde{y})+\varphi(\tilde{y})<\gamma(z-y)+\varphi(y).
\]
Then we have 
\begin{align*}
\gamma(x-\tilde{y})+\varphi(\tilde{y}) & \le\gamma(x-z)+\gamma(z-\tilde{y})+\varphi(\tilde{y})\\
 & <\gamma(x-z)+\gamma(z-y)+\varphi(y)=\gamma(x-y)+\varphi(y),
\end{align*}
which is a contradiction. Hence $y$ is a $\rho$-closest point to
$z$.

Therefore the points in the segment $[x,y]$ have $y$ as a $\rho$-closest
point on $\partial U$. Hence for $0\le t\le\gamma(x-y)$ we have
\begin{align*}
\rho\big(x-\frac{t}{\gamma(x-y)}(x-y)\big) & =\gamma\big(x-\frac{t}{\gamma(x-y)}(x-y)-y\big)+\varphi(y)\\
 & =\big(1-\frac{t}{\gamma(x-y)}\big)\gamma(x-y)+\varphi(y)=\gamma(x-y)-t+\varphi(y).
\end{align*}
Thus $\rho$ varies linearly along the segment.

(b) Suppose to the contrary that there is $\tilde{z}\in]x,y[\cap\partial U$.
But then we have 
\[
\gamma(x-\tilde{z})+\varphi(\tilde{z})<\gamma(x-\tilde{z})+\gamma(\tilde{z}-y)+\varphi(y)=\gamma(x-y)+\varphi(y),
\]
which is a contradiction. 

(c) Suppose $z\in]x,y[$, and $\tilde{y}\in\partial U-\{y\}$ is another
$\rho$-closest point to $z$. Hence we have 
\[
\gamma(z-\tilde{y})+\varphi(\tilde{y})=\gamma(z-y)+\varphi(y).
\]
If $\tilde{y}$ belongs to the line containing $x,z,y$, then there
are two cases. If $\tilde{y}$ is on the same side of $x$ as $y$,
then $\tilde{y}$ cannot lie between $x,y$, since $]x,y[\subset U$.
Thus $y$ is between $\tilde{y},x$, and therefore it is also between
$\tilde{y},z$. But this is a contradiction since $]\tilde{y},z[\subset U$.
On the other hand if $\tilde{y},y$ are on different sides of $x$,
we have 
\[
\gamma(x-\tilde{y})+\varphi(\tilde{y})<\gamma(z-\tilde{y})+\varphi(\tilde{y})=\gamma(z-y)+\varphi(y)<\gamma(x-y)+\varphi(y),
\]
which is also a contradiction. Finally suppose that $x,z,\tilde{y}$
are not collinear. Then by strict convexity of $\gamma$ we get 
\begin{align*}
\gamma(x-\tilde{y})+\varphi(\tilde{y}) & <\gamma(x-z)+\gamma(z-\tilde{y})+\varphi(\tilde{y})\\
 & =\gamma(x-z)+\gamma(z-y)+\varphi(y)=\gamma(x-y)+\varphi(y),
\end{align*}
which is a contradiction too. Thus $y$ is the unique $\rho$-closest
point to $z$.
\end{proof}
\begin{lem}
\label{lem: rho not C1}Suppose $\gamma$ is strictly convex, and
the strict Lipschitz property (\ref{eq: phi strct Lip}) for $\varphi$
holds. If for some point $x\in U$ there are two different points
$y,z\in\partial U$ so that 
\[
\rho(x)=\gamma(x-y)+\varphi(y)=\gamma(x-z)+\varphi(z),
\]
then $\rho$ is not differentiable at $x$.
\end{lem}
\begin{proof}
We know that the points in the segment $[x,y]$ have $y$ as a $\rho$-closest
point on $\partial U$. Hence as we have seen in the proof of the
previous lemma, for $0\le t\le\gamma(x-y)$ we have 
\[
\rho\big(x-\frac{t}{\gamma(x-y)}(x-y)\big)=\gamma(x-y)-t+\varphi(y).
\]
Now suppose to the contrary that $\rho$ is differentiable at $x$.
Then by differentiating the above equality (and the similar formula
for $z$) with respect to $t$, we get 
\[
\big\langle D\rho(x),\frac{x-y}{\gamma(x-y)}\big\rangle=1=\big\langle D\rho(x),\frac{x-z}{\gamma(x-z)}\big\rangle.
\]
On the other hand, it is easy to show that $\gamma^{\circ}(D\rho(x))\le1$.
To do this just note that 
\[
\rho(x+tw)-\rho(x)\le\gamma(x+tw-x)=t\gamma(w).
\]
By taking the limit as $t\to0^{+}$, we get $\langle D\rho(x),w\rangle\le\gamma(w)$.
Then we get the desired result by (\ref{eq: gen Cauchy-Schwartz 2}).

Now note that if two vectors $w,\tilde{w}$ satisfy $\gamma(w)=1=\gamma(\tilde{w})$
and 
\[
\langle D\rho(x),w\rangle=1=\langle D\rho(x),\tilde{w}\rangle,
\]
then one of them is a positive multiple of the other, and consequently
the two vectors are equal.%
{} Since otherwise by strict convexity of $\gamma$ and inequality (\ref{eq: gen Cauchy-Schwartz}),
we get 
\[
\big\langle D\rho(x),\frac{w+\tilde{w}}{2}\big\rangle\leq\gamma^{\circ}(D\rho(x))\gamma\big(\frac{w+\tilde{w}}{2}\big)<\gamma^{\circ}(D\rho(x))\frac{\gamma(w)+\gamma(\tilde{w})}{2}=1.
\]
However we must have $\langle D\rho(x),(\frac{w+\tilde{w}}{2})\rangle=\frac{1+1}{2}=1$,
which is a contradiction.

Therefore we must have 
\[
\frac{x-y}{\gamma(x-y)}=\frac{x-z}{\gamma(x-z)}.
\]
This implies that $x,y,z$ are collinear, and $y,z$ are on the same
side of $x$. Then either $y$ is between $x,z$, or $z$ is between
$x,y$. But both of these cases are impossible since $]x,y[$ and
$]x,z[$ are subsets of $U$. Thus $\rho$ cannot be differentiable
at $x$.
\end{proof}
\begin{lem}
\label{lem: segment is plastic}Suppose $u\in C^{1}(U)$, and the
strict Lipschitz property (\ref{eq: phi strct Lip}) for $\varphi$
holds. If $x\in P^{+}$, and $y$ is a $\rho$-closest point on $\partial U$
to $x$, then $[x,y[\subset P^{+}$. Similarly, if $x\in P^{-}$,
and $y$ is a $\bar{\rho}$-closest point on $\partial U$ to $x$,
then $[x,y[\subset P^{-}$.%
\end{lem}
\begin{proof}
Note that $[x,y[\subset U$. Suppose $x\in P^{-}$; the other case
is similar. We have 
\[
u(x)=-\bar{\rho}(x)=-\gamma(y-x)+\varphi(y).
\]
Let $v:=u-(-\bar{\rho})\ge0$, and $\xi:=\frac{y-x}{\gamma(y-x)}=-\frac{x-y}{\bar{\gamma}(x-y)}$.
Then $\bar{\rho}$ varies linearly along the segment $]x,y[$, since
$y$ is a $\bar{\rho}$-closest point to the points of the segment.
So we have $D_{\xi}(-\bar{\rho})=D_{-\xi}\bar{\rho}=1$ along the
segment, as shown in the proof of Lemma \ref{lem: rho not C1}. Note
that we do not assume the differentiability of $\bar{\rho}$; and
$D_{-\xi}\bar{\rho}$ is just the derivative of the restriction of
$\bar{\rho}$ to the segment $]x,y[$. Now since 
\[
D_{\xi}u=\langle Du,\xi\rangle\le\gamma^{\circ}(Du)\gamma(\xi)\le1,
\]
we have $D_{\xi}v\le0$ along $]x,y[$. Thus as $v(x)=v(y)=0$, and
$v$ is continuous on the closed segment $[x,y]$, we must have $v\equiv0$
on $[x,y]$. Therefore $u=-\bar{\rho}$ along the segment as desired.
\end{proof}
Now recall that we have $Du\in K^{\circ}$, which is equivalent to
$\gamma^{\circ}(Du)\le1$. The next lemma tells us when we hit the
gradient constraint, i.e. when $Du\in\partial K^{\circ}$, or equivalently
when $\gamma^{\circ}(Du)=1$. The answer is that we hit the gradient
constraint exactly when we hit one of the obstacles $-\bar{\rho},\rho$.
\begin{lem}
\label{lem: E , P}Suppose the Assumption \ref{assu: 1} holds, and
$u\in C_{\mathrm{loc}}^{1,1}(U)$. Also suppose that the strict Lipschitz
property (\ref{eq: phi strct Lip}) for $\varphi$ holds, and $\gamma$
is strictly convex. Then we have 
\begin{align*}
 & P=\{x\in U:\gamma^{\circ}(Du(x))=1\},\\
 & E=\{x\in U:\gamma^{\circ}(Du(x))<1\}.
\end{align*}
\end{lem}
\begin{proof}
First suppose $x\in P^{-}$; the case of $P^{+}$ is similar. Then
we have 
\[
u(x)=-\bar{\rho}(x)=-\gamma(y-x)+\varphi(y),
\]
for some $y\in\partial U$. Thus by Lemma \ref{lem: segment is plastic},
$u=-\bar{\rho}$ along the segment $[x,y[$. We also know that $\bar{\rho}$
varies linearly along the segment $[x,y[$, since $y$ is a $\bar{\rho}$-closest
point to the points of the segment. Hence we have $D_{\xi}u(x)=1$
for $\xi:=\frac{y-x}{\gamma(y-x)}$, as shown in the proof of Lemma
\ref{lem: rho not C1}. Therefore $\gamma^{\circ}(Du(x))$ cannot
be less than $1$ due to the equation (\ref{eq: gen Cauchy-Schwartz 2}).

Next, assume that $\gamma^{\circ}(Du(x))=1$. Then by (\ref{eq: gen Cauchy-Schwartz 2}),
there is $\tilde{\xi}$ with $\gamma(\tilde{\xi})=1$ such that $D_{\tilde{\xi}}u(x)=1$.
Suppose to the contrary that $x\in E$, i.e. $-\bar{\rho}(x)<u(x)<\rho(x)$.
By (\ref{eq: diff ineq}) we know $D_{\tilde{\xi}}u$ is a weak solution
to the elliptic equation 
\[
-D_{i}(a_{ij}D_{j}D_{\tilde{\xi}}u)+bD_{\tilde{\xi}}u=0\qquad\textrm{ in }E,
\]
where $a_{ij}:=D_{ij}^{2}F(Du)$, and $b:=g''(u)$. On the other hand
\[
D_{\tilde{\xi}}u=\langle Du,\tilde{\xi}\rangle\le\gamma^{\circ}(Du)\gamma(\tilde{\xi})\le1
\]
on $U$. Let $E_{1}$ be the connected component of $E$ that contains
$x$. Then the strong maximum principle (Theorem 8.19 of \citep{MR1814364})
implies that $D_{\tilde{\xi}}u\equiv1$ over $E_{1}$. Note that we
can work in open subsets of $E_{1}$ which are compactly contained
in $E_{1}$; so we do not need the global integrability of $D^{2}u$
to apply the maximum principle.

Now consider the line passing through $x$ in the $\tilde{\xi}$ direction,
and suppose it intersects $\partial E_{1}$ for the first time in
$y:=x-\tau\tilde{\xi}$ for some $\tau>0$. If $y\in\partial U$,
then for $t>0$ we have 
\[
\frac{d}{dt}[u(y+t\tilde{\xi})]=D_{\tilde{\xi}}u(y+t\tilde{\xi})=1=\frac{d}{dt}[t\gamma(\tilde{\xi})]=\frac{d}{dt}[\gamma(y+t\tilde{\xi}-y)].
\]
Thus as $u(y)=\varphi(y)$, we get $u(x)=u(y+\tau\tilde{\xi})=\gamma(x-y)+\varphi(y)\ge\rho(x)$;
which is a contradiction. Now if $y\in U$, then as it also belongs
to $\partial E$ we have $y\in P$. If $u(y)=\rho(y)=\gamma(y-\tilde{y})+\varphi(\tilde{y})$
for some $\tilde{y}\in\partial U$, similarly to the above we obtain
\begin{align*}
u(x) & =\gamma(x-y)+u(y)\\
 & =\gamma(x-y)+\gamma(y-\tilde{y})+\varphi(\tilde{y})\ge\gamma(x-\tilde{y})+\varphi(\tilde{y})\ge\rho(x),
\end{align*}
which is again a contradiction.

On the other hand, if $u(y)=-\bar{\rho}(y)=-\gamma(\tilde{y}-y)+\varphi(\tilde{y})$
for some $\tilde{y}\in\partial U$, then by Lemma \ref{lem: segment is plastic}
we have $u=-\bar{\rho}$ on the segment $[y,\tilde{y}[$; and consequently
$D_{\hat{\xi}}u(y)=1$, where $\hat{\xi}:=\frac{\tilde{y}-y}{\gamma(\tilde{y}-y)}$.
Since $u$ is differentiable we must have $\tilde{\xi}=\hat{\xi}$,
as shown in the proof of Lemma \ref{lem: rho not C1}. Therefore $x,y,\tilde{y}$
are collinear, and $x,\tilde{y}$ are on the same side of $y$. But
$\tilde{y}$ cannot belong to $]y,x[\subset E_{1}\subset E\subset U$.
Hence we must have $x\in]y,\tilde{y}[\subset P^{-}$, which means
$u(x)=-\bar{\rho}(x)$; and this is a contradiction.
\end{proof}
\begin{rem*}
In the above proof, we only used the strict convexity of $\gamma$
in the last paragraph. So without this assumption we have 
\[
P\subset\{x\in U:\gamma^{\circ}(Du(x))=1\},\qquad E\supset\{x\in U:\gamma^{\circ}(Du(x))<1\}.
\]
Furthermore, if we can drop one of the obstacles, then we do not need
the argument given in the last paragraph, and we can conclude that
the above lemma holds without assuming the strict convexity of $\gamma$.
(Note that if we only have the obstacle $-\bar{\rho}$, then in the
above proof we have to look for a point of the form $x+\tau\tilde{\xi}\in\partial E_{1}$
for some $\tau>0$.) For example, when $g$ is decreasing, we can
show that $u\ge0$ (since $J[u^{+}]\le J[u]$). Thus if in addition
$\varphi=0$, then $u$ does not touch the lower obstacle, since in
this case we have $-\bar{\rho}<0$.
\end{rem*}
\begin{prop}
\label{prop: ridge0 elastic}Suppose the Assumption \ref{assu: 1}
holds, and $u\in C^{1}(U)$. Also suppose that the strict Lipschitz
property (\ref{eq: phi strct Lip}) for $\varphi$ holds, and $\gamma$
is strictly convex. Then we have 
\begin{eqnarray*}
R_{\rho,0}\cap P^{+}=\emptyset, & \hspace{2cm} & R_{\bar{\rho},0}\cap P^{-}=\emptyset.
\end{eqnarray*}
\end{prop}
\begin{proof}
Let us show that $R_{\bar{\rho},0}\cap P^{-}=\emptyset$; the other
case is similar. Suppose to the contrary that $x\in R_{\bar{\rho},0}\cap P^{-}$.
Then there are at least two distinct points $y,z\in\partial U$ such
that 
\[
\bar{\rho}(x)=\gamma(y-x)-\varphi(y)=\gamma(z-x)-\varphi(z).
\]
Now by Lemma \ref{lem: segment is plastic}, we have $[x,y[,[x,z[\subset P^{-}$.
In other words, $u=-\bar{\rho}$ on both of these segments. Therefore,
we can argue as in the proof of Lemma \ref{lem: rho not C1} to obtain
\[
\big\langle Du(x),\frac{y-x}{\gamma(y-x)}\big\rangle=1=\big\langle Du(x),\frac{z-x}{\gamma(z-x)}\big\rangle;
\]
and to get a contradiction with the fact that $\gamma^{\circ}(Du(x))\le1$.
\end{proof}

\bibliographystyle{plainnat}
\bibliography{/Volumes/A/Dropbox/Bibliography-Jan-2021}

\def\cprime{$'$} \def\cprime{$'$} \def\cprime{$'$} \def\cprime{$'$}
  \def\cprime{$'$} \def\cprime{$'$} \def\cprime{$'$} \def\cprime{$'$}
\begin{thebibliography}{48}
\providecommand{\natexlab}[1]{#1}
\providecommand{\url}[1]{\texttt{#1}}
\expandafter\ifx\csname urlstyle\endcsname\relax
  \providecommand{\doi}[1]{doi: #1}\else
  \providecommand{\doi}{doi: \begingroup \urlstyle{rm}\Url}\fi

\bibitem[Andersson et~al.(2013)Andersson, Lindgren, and
  Shahgholian]{andersson2013optimal}
J.~Andersson, E.~Lindgren, and H.~Shahgholian.
\newblock Optimal regularity for the no-sign obstacle problem.
\newblock \emph{Comm. Pure Appl. Math.}, 66\penalty0 (2):\penalty0 245--262,
  2013.

\bibitem[Barles and Soner(1998)]{barles1998option}
G.~Barles and H.~M. Soner.
\newblock Option pricing with transaction costs and a nonlinear black-scholes
  equation.
\newblock \emph{Finance Stochast.}, 2\penalty0 (4):\penalty0 369--397, 1998.

\bibitem[Brezis and Stampacchia(1968)]{MR0239302}
H.~Brezis and G.~Stampacchia.
\newblock Sur la r\'egularit\'e de la solution d'in\'equations elliptiques.
\newblock \emph{Bull. Soc. Math. France}, 96:\penalty0 153--180, 1968.

\bibitem[Caffarelli and Friedman(1979)]{MR534111}
L.~A. Caffarelli and A.~Friedman.
\newblock The free boundary for elastic-plastic torsion problems.
\newblock \emph{Trans. Amer. Math. Soc.}, 252:\penalty0 65--97, 1979.

\bibitem[Caffarelli and Rivi{\`e}re(1976)]{MR0412940}
L.~A. Caffarelli and N.~M. Rivi{\`e}re.
\newblock Smoothness and analyticity of free boundaries in variational
  inequalities.
\newblock \emph{Ann. Scuola Norm. Sup. Pisa Cl. Sci. (4)}, 3\penalty0
  (2):\penalty0 289--310, 1976.

\bibitem[Caffarelli and Rivi{\`e}re(1977)]{MR0521411}
L.~A. Caffarelli and N.~M. Rivi{\`e}re.
\newblock The smoothness of the elastic-plastic free boundary of a twisted bar.
\newblock \emph{Proc. Amer. Math. Soc.}, 63\penalty0 (1):\penalty0 56--58,
  1977.

\bibitem[Caffarelli and Rivi{\`e}re(1979)]{MR513957}
L.~A. Caffarelli and N.~M. Rivi{\`e}re.
\newblock The {L}ipschitz character of the stress tensor, when twisting an
  elastic plastic bar.
\newblock \emph{Arch. Ration. Mech. Anal.}, 69\penalty0 (1):\penalty0 31--36,
  1979.

\bibitem[Caffarelli et~al.(1980)Caffarelli, Friedman, and Pozzi]{MR563207}
L.~A. Caffarelli, A.~Friedman, and G.~Pozzi.
\newblock Reflection methods in the elastic-plastic torsion problem.
\newblock \emph{Indiana Univ. Math. J.}, 29\penalty0 (2):\penalty0 205--228,
  1980.

\bibitem[Cannarsa and Sinestrari(2004)]{cannarsa2004semiconcave}
P.~Cannarsa and C.~Sinestrari.
\newblock \emph{Semiconcave Functions, Hamilton-Jacobi Equations, and Optimal
  Control}.
\newblock Progress in Nonlinear Differential Equations and Their Applications.
  Birkh{\"a}user Boston, 2004.

\bibitem[Choe and Shim(1995{\natexlab{a}})]{MR1310935}
H.~J. Choe and Y.-S. Shim.
\newblock On the variational inequalities for certain convex function classes.
\newblock \emph{J. Differential Equations}, 115\penalty0 (2):\penalty0
  325--349, 1995{\natexlab{a}}.

\bibitem[Choe and Shim(1995{\natexlab{b}})]{MR1315349}
H.~J. Choe and Y.-S. Shim.
\newblock Degenerate variational inequalities with gradient constraints.
\newblock \emph{Ann. Scuola Norm. Sup. Pisa Cl. Sci. (4)}, 22\penalty0
  (1):\penalty0 25--53, 1995{\natexlab{b}}.

\bibitem[Crasta and Malusa(2007)]{MR2336304}
G.~Crasta and A.~Malusa.
\newblock The distance function from the boundary in a {M}inkowski space.
\newblock \emph{Trans. Amer. Math. Soc.}, 359\penalty0 (12):\penalty0
  5725--5759 (electronic), 2007.

\bibitem[Dacorogna(2008)]{MR2361288}
B.~Dacorogna.
\newblock \emph{Direct methods in the calculus of variations}, volume~78 of
  \emph{Applied Mathematical Sciences}.
\newblock Springer, New York, second edition, 2008.

\bibitem[De~Silva and Savin(2010)]{MR2605868}
D.~De~Silva and O.~Savin.
\newblock Minimizers of convex functionals arising in random surfaces.
\newblock \emph{Duke Math. J.}, 151\penalty0 (3):\penalty0 487--532, 2010.

\bibitem[Evans(1979)]{MR529814}
L.~C. Evans.
\newblock A second-order elliptic equation with gradient constraint.
\newblock \emph{Comm. Partial Differential Equations}, 4\penalty0 (5):\penalty0
  555--572, 1979.

\bibitem[Evans and Gariepy(2015)]{MR3409135}
L.~C. Evans and R.~F. Gariepy.
\newblock \emph{Measure theory and fine properties of functions}.
\newblock Textbooks in Mathematics. CRC Press, Boca Raton, FL, revised edition,
  2015.

\bibitem[Fan(2007)]{fan2007global}
X.~Fan.
\newblock Global {$C^{1, \alpha }$} regularity for variable exponent elliptic
  equations in divergence form.
\newblock \emph{J. Differential Equations}, 235\penalty0 (2):\penalty0
  397--417, 2007.

\bibitem[Figalli and Shahgholian(2014)]{MR3198649}
A.~Figalli and H.~Shahgholian.
\newblock A general class of free boundary problems for fully nonlinear
  elliptic equations.
\newblock \emph{Arch. Ration. Mech. Anal.}, 213\penalty0 (1):\penalty0
  269--286, 2014.

\bibitem[Friedman(1982)]{MR679313}
A.~Friedman.
\newblock \emph{Variational Principles And Free-Boundary Problems}.
\newblock Pure and Applied Mathematics. John Wiley \& Sons, Inc., New York,
  1982.

\bibitem[Friedman and Pozzi(1980)]{MR552267}
A.~Friedman and G.~Pozzi.
\newblock The free boundary for elastic-plastic torsion problems.
\newblock \emph{Trans. Amer. Math. Soc.}, 257\penalty0 (2):\penalty0 411--425,
  1980.

\bibitem[Gerhardt(1975)]{MR0385296}
C.~Gerhardt.
\newblock Regularity of solutions of nonlinear variational inequalities with a
  gradient bound as constraint.
\newblock \emph{Arch. Ration. Mech. Anal.}, 58\penalty0 (4):\penalty0 309--315,
  1975.

\bibitem[Giaquinta(1983)]{MR717034}
M.~Giaquinta.
\newblock \emph{Multiple integrals in the calculus of variations and nonlinear
  elliptic systems}, volume 105 of \emph{Annals of Mathematics Studies}.
\newblock Princeton University Press, Princeton, NJ, 1983.

\bibitem[Giaquinta and Giusti(1984)]{MR749677}
M.~Giaquinta and E.~Giusti.
\newblock Global {$C\sp{1,\alpha }$}-regularity for second order quasilinear
  elliptic equations in divergence form.
\newblock \emph{J. Reine Angew. Math.}, 351:\penalty0 55--65, 1984.

\bibitem[Gilbarg and Trudinger(2001)]{MR1814364}
D.~Gilbarg and N.~S. Trudinger.
\newblock \emph{Elliptic Partial Differential Equations Of Second Order}.
\newblock Classics in Mathematics. Springer-Verlag, Berlin, 2001.

\bibitem[Giuffr{\`e} et~al.(2015)Giuffr{\`e}, Maugeri, and
  Puglisi]{giuffre2015lagrange}
S.~Giuffr{\`e}, A.~Maugeri, and D.~Puglisi.
\newblock Lagrange multipliers in elastic--plastic torsion problem for
  nonlinear monotone operators.
\newblock \emph{J. Differential Equations}, 259\penalty0 (3):\penalty0
  817--837, 2015.

\bibitem[Hatcher(2002)]{hatcher2002algebraic}
A.~Hatcher.
\newblock \emph{Algebraic Topology}.
\newblock Cambridge University Press, 2002.

\bibitem[Hynd(2012)]{Hynd-2012}
R.~Hynd.
\newblock The eigenvalue problem of singular ergodic control.
\newblock \emph{Comm. Pure Appl. Math.}, 65\penalty0 (5):\penalty0 649--682,
  2012.

\bibitem[Hynd(2017)]{Hynd2017}
R.~Hynd.
\newblock An eigenvalue problem for a fully nonlinear elliptic equation with
  gradient constraint.
\newblock \emph{Calc. Var. Partial Differential Equations}, 56\penalty0
  (2):\penalty0 34, 2017.

\bibitem[Hynd and Mawi(2016)]{Hynd}
R.~Hynd and H.~Mawi.
\newblock On hamilton-jacobi-bellman equations with convex gradient
  constraints.
\newblock \emph{Interfaces Free Bound.}, 18\penalty0 (3):\penalty0 291--315,
  2016.

\bibitem[Indrei(2013)]{indrei2013free}
E.~Indrei.
\newblock Free boundary regularity in the optimal partial transport problem.
\newblock \emph{J. Funct. Anal.}, 264\penalty0 (11):\penalty0 2497--2528, 2013.

\bibitem[Indrei and Minne(2016{\natexlab{a}})]{Indrei-Minne}
E.~Indrei and A.~Minne.
\newblock Regularity of solutions to fully nonlinear elliptic and parabolic
  free boundary problems.
\newblock \emph{Ann. Inst. H. Poincar\'e Anal. Non Lin\'eaire}, 33\penalty0
  (5):\penalty0 1259--1277, 2016{\natexlab{a}}.

\bibitem[Indrei and Minne(2016{\natexlab{b}})]{indrei2016nontransversal}
E.~Indrei and A.~Minne.
\newblock Nontransversal intersection of free and fixed boundaries for fully
  nonlinear elliptic operators in two dimensions.
\newblock \emph{Anal. PDE}, 9\penalty0 (2):\penalty0 487--502,
  2016{\natexlab{b}}.

\bibitem[Indrei and Nurbekyan(2016)]{indrei2016regularity}
E.~Indrei and L.~Nurbekyan.
\newblock Regularity of shadows and the geometry of the singular set associated
  to a {M}onge--{A}mp{\`e}re equation.
\newblock \emph{Comm. Anal. Geom.}, 24\penalty0 (4):\penalty0 793--820, 2016.

\bibitem[Ishii and Koike(1983)]{MR693645}
H.~Ishii and S.~Koike.
\newblock Boundary regularity and uniqueness for an elliptic equation with
  gradient constraint.
\newblock \emph{Comm. Partial Differential Equations}, 8\penalty0 (4):\penalty0
  317--346, 1983.

\bibitem[Jensen(1983)]{MR697646}
R.~Jensen.
\newblock Regularity for elastoplastic type variational inequalities.
\newblock \emph{Indiana Univ. Math. J.}, 32\penalty0 (3):\penalty0 407--423,
  1983.

\bibitem[Li and Nirenberg(2005)]{MR2094267}
Y.~Li and L.~Nirenberg.
\newblock The distance function to the boundary, {F}insler geometry, and the
  singular set of viscosity solutions of some {H}amilton-{J}acobi equations.
\newblock \emph{Comm. Pure Appl. Math.}, 58\penalty0 (1):\penalty0 85--146,
  2005.

\bibitem[Lions(1982)]{MR667669}
P.-L. Lions.
\newblock \emph{Generalized solutions of {H}amilton-{J}acobi equations},
  volume~69 of \emph{Research Notes in Mathematics}.
\newblock Pitman (Advanced Publishing Program), Boston, Mass.-London, 1982.

\bibitem[Possama{\"\i} et~al.(2015)Possama{\"\i}, Soner, and
  Touzi]{possamai2015homogenization}
D.~Possama{\"\i}, H.~M. Soner, and N.~Touzi.
\newblock Homogenization and asymptotics for small transaction costs: the
  multidimensional case.
\newblock \emph{Comm. Partial Differential Equations}, 40\penalty0
  (11):\penalty0 2005--2046, 2015.

\bibitem[Safdari(2015)]{Safdari20151}
M.~Safdari.
\newblock The free boundary of variational inequalities with gradient
  constraints.
\newblock \emph{Nonlinear Anal. Theory Methods Appl.}, 123-124:\penalty0 1 --
  22, 2015.

\bibitem[Safdari(2017)]{safdari2017shape}
M.~Safdari.
\newblock On the shape of the free boundary of variational inequalities with
  gradient constraints.
\newblock \emph{Interfaces Free Bound.}, 19\penalty0 (2):\penalty0 183--200,
  2017.

\bibitem[Safdari(2018)]{MR1}
M.~Safdari.
\newblock The regularity of some vector-valued variational inequalities with
  gradient constraints.
\newblock \emph{Comm. Pure Appl. Anal.}, 17\penalty0 (2):\penalty0 413--428,
  2018.

\bibitem[Safdari(2019)]{Paper-4}
M.~Safdari.
\newblock The distance function from the boundary of a domain with corners.
\newblock \emph{Nonlinear Anal.}, 181:\penalty0 294--310, 2019.

\bibitem[Safdari(2021)]{SAFDARI2021358}
M.~Safdari.
\newblock Double obstacle problems and fully nonlinear {PDE} with non-strictly
  convex gradient constraints.
\newblock \emph{J. Differential Equations}, 278:\penalty0 358 -- 392, 2021.

\bibitem[Schneider(2014)]{MR3155183}
R.~Schneider.
\newblock \emph{Convex bodies: the {B}runn-{M}inkowski theory}, volume 151 of
  \emph{Encyclopedia of Mathematics and its Applications}.
\newblock Cambridge University Press, Cambridge, expanded edition, 2014.

\bibitem[Ting(1966)]{MR0184503}
T.~W. Ting.
\newblock The ridge of a {J}ordan domain and completely plastic torsion.
\newblock \emph{J. Math. Mech.}, 15:\penalty0 15--47, 1966.

\bibitem[Treu and Vornicescu(2000)]{MR1797872}
G.~Treu and M.~Vornicescu.
\newblock On the equivalence of two variational problems.
\newblock \emph{Calc. Var. Partial Differential Equations}, 11\penalty0
  (3):\penalty0 307--319, 2000.

\bibitem[Weil(1973)]{MR0467623}
W.~Weil.
\newblock Ein {A}pproximationssatz f\"ur konvexe {K}\"orper.
\newblock \emph{Manuscripta Math.}, 8:\penalty0 335--362, 1973.

\bibitem[Wiegner(1981)]{MR607553}
M.~Wiegner.
\newblock The {$C\sp{1,1}$}-character of solutions of second order elliptic
  equations with gradient constraint.
\newblock \emph{Comm. Partial Differential Equations}, 6\penalty0 (3):\penalty0
  361--371, 1981.

\end{thebibliography}

\end{document}